\documentclass[11pt,oneside,english,final]{amsart}

\usepackage[notcite,notref,color]{showkeys}  

\usepackage{amstext,amsthm,amssymb,bbm,amsmath}
\usepackage[T1]{fontenc}
\usepackage[tmargin=2cm,bmargin=2cm,left=2.2cm,right=2.2cm]{geometry}

\usepackage{graphicx}
\usepackage{mathtools}

\usepackage[colorlinks,bookmarks,linkcolor=black,citecolor=black,urlcolor=black]{hyperref} 
\usepackage{cleveref} 
\usepackage[doi=false,url=false,isbn=false,maxbibnames=100]{biblatex}
\usepackage{enumitem}
\usepackage[english,algoruled,lined,noresetcount,norelsize]{algorithm2e}

\newtheorem{theorem}{Theorem}[section]
\newtheorem*{theorem*}{Theorem}
\newtheorem{lemma}[theorem]{Lemma}
\newtheorem{claim}[theorem]{Claim}

\newtheorem{fact}[theorem]{Fact}
\newtheorem{conjecture}[theorem]{Conjecture}

\newtheorem{setting}[theorem]{Setting}

\usepackage{comment}
\newtheorem{definition}[theorem]{Definition}
\newtheorem*{definition*}{Definition}

\theoremstyle{remark}

\newcommand{\oldqed}{}
\newcommand{\qedClaim}{\hfill\scalebox{.6}{\(\Box\)}}
\newenvironment{claimproof}[1][Proof]{
  \renewcommand{\oldqed}{\qedsymbol}
  \renewcommand{\qedsymbol}{\qedClaim}
  \begin{proof}[#1]
}{
  \end{proof}
  \renewcommand{\qedsymbol}{\oldqed}
} 

\setlist{itemsep=2pt,parsep=1pt,topsep=3pt,partopsep=0pt}  
\setenumerate{leftmargin=*,labelindent=\parindent} 

\def\itm#1{\rm ({#1})} 
\def\itmit#1{\itm{\it #1\,}} 
\def\rom{\itmit{\roman{*}}} 
\def\abc{\itmit{\alph{*}}}

\def\itmarab#1{\mbox{\itm{{\it #1\,}\arabic{*}\hspace{.05em}}}}
\def\itmarabp#1{\mbox{\itm{{\it #1\,}\arabic{*}'\hspace{.05em}}}}

\newcommand{\By}[2]{\overset{\mbox{\tiny{#1}}}{#2}}
\newcommand{\ByRef}[2]{   \By{\eqref{#1}}{#2} }

\newcommand{\leBy}[1]{    \By{#1}{\le} }

\newcommand{\lByRef}[1]{  \ByRef{#1}{<} }

\newcommand{\vass}[1]{\left| #1\right|}

\newcommand{\ind}{\mathbbm{1}}

\newcommand{\eps}{\varepsilon}
\renewcommand{\rho}{\varrho}
\renewcommand{\subset}{\subseteq}
\newcommand{\dcup}{\dot\cup}
\newcommand{\tmu}{\tilde{\mu}}

\newcommand{\EE}{\mathbb{E}}
\newcommand{\Prob}{\mathbb{P}}
\newcommand{\RR}{\mathbb{R}}
\newcommand{\one}{\mathbf{1}}

\newcommand{\nats}{\mathbb{N}}
\newcommand{\reals}{\mathbb{R}}
\newcommand{\bg}{\mathrm{big}}
\newcommand{\smll}{\mathrm{small}}
\renewcommand{\epsilon}{\varepsilon}
\newcommand{\bfi}{\mathbf{i}}

\newcommand{\cG}{\mathcal{G}}
\newcommand{\cE}{\mathcal{E}}
\newcommand{\cB}{\mathcal{B}}
\newcommand{\cS}{\mathcal{S}}

\newcommand{\lin}{\textrm{lin}}

\title{A generalised transference principle}

\author[P. Allen]{Peter Allen}
\address{(PA) London School of Economics, Department of Mathematics, Houghton Street, London WC2A 2AE, UK}
\email{p.d.allen@lse.ac.uk}
\author[J. B\"ottcher]{Julia B\"ottcher}
\address{(JB) London School of Economics, Department of Mathematics, Houghton Street, London WC2A 2AE, UK}
\email{j.boettcher@lse.ac.uk}
\author[J. Lada]{Joanna Lada}
\address{(JL) London School of Economics, Department of Mathematics, Houghton Street, London WC2A 2AE, UK}
\email{j.m.lada@lse.ac.uk}
    \author[D. Mergoni Cecchelli]{Domenico Mergoni Cecchelli}
\address{(DMC) CUNEF Universidad, Department of Quantitative Methods, C. de Almansa, Madrid 28040, ES}
\email{domenico.mergoni@cunef.edu}
\date{}

\begin{document}
\begin{abstract}
    The last two decades have witnessed a growing trend towards proving sparse random analogues of combinatorial theorems. One unified approach to proving such theorems, formalised by Conlon and Gowers [Ann.~of Math.~2016], involves establishing a `transference principle' which allows one to translate between robust properties in the dense setting and the sparse $p$-random setting, provided $p$ is not too small. Our results provide a more general transference theorem, extending the results of Conlon and Gowers and also those of Schacht [Ann.~of Math.~2016]. Among a variety of other applications, we use this to obtain a sparse counting lemma for graphs and hypergraphs which are not necessarily strictly balanced. Our method achieves asymptotically optimal bounds on the probability \(p\), and the probability of success.
\end{abstract}
\maketitle

\section{Introduction}
Many fundamental problems in extremal combinatorics ask how to maximise the size of a discrete structure subject to a given forbidden substructure not being present. Classical results in this area include Tur\'an's theorem~\cite{turan1941extremalaufgabe}, which establishes the maximum number of edges in an $n$-vertex graph containing no $k$-vertex complete graph; Szemer\'edi's theorem~\cite{szemeredi1975sets}, which upper bounds the size of a set of integers in $[n]:=\{1,\dots,n\}$ containing no (non-trivial) $k$-term arithmetic progression; and Ramsey's theorem~\cite{ramsey1930problem}, which upper bounds the number $n$ of vertices of a complete graph that admits an $r$-edge-colouring containing no monochromatic complete graph on $k$ vertices.

A more recent development is the study of sparse versions of these theorems. The question here is whether similar results hold in sparse random, or pseudorandom, discrete structures. We denote by $G(N,p)$ the \emph{binomial random graph}, which is the graph on vertex set $[N]$ obtained by including each pair of vertices independently with probability~$p$ as an edge.
The Tur\'an-type problem in the sparse setting then asks what the maximum relative edge density of a subgraph $G$ of a typical $\Gamma=G(N,p)$ which contains no $k$-vertex complete graph is, and whether this is the same as in Tur\'an's theorem.
Here, by `typical' we mean that the probability that this desired event does not hold is small.
Similarly, one can ask for sparse versions of Ramsey's theorem. This means we are aiming to find conditions to guarantee monochromatic copies of a complete graph on $k$ vertices in any $r$-edge-colouring of $\Gamma$. And finally, sparse versions of Szemer\'edi's theorem ask for conditions that guarantee that a subset of a given density of a random subset of $[n]$ contains $k$-term arithmetic progressions.

The research on these sparse problems was initiated in the 1990s with the study of Ramsey's theorem in the sparse random setting. Over a series of papers, {\L}uczak, Ruci\'nski, and Voigt~\cite{LRV} and R\"odl and Ruci\'nski~\cite{RR93,RR94,RR} gave a complete answer 
to this Ramsey question. Indeed,  R\"odl and Ruci\'nski identified in~\cite{RR} a threshold such that, almost all graphs having edge density over said threshold admit a monochromatic $k$-clique in all their $r$-edge-colourings.
Around the same time the first sparse random Tur\'an-type and Szemer\'edi-type results were proved~\cite{BabSimSpe, FraRod, Furedi_cycle, GerSchSte, HaxKohLuc_evencycles, HaxKohLuc_oddcycles, 
KLR_progressions, KLRConj, KohRodSch, SzaVu}, but a full resolution of these problems remained open.

Arguably, the next major development in the area was due to Green and Tao~\cite{GreenTao}, who studied arithmetic progressions in sparse pseudorandom sets. Their primary focus was proving that the primes contain arbitrarily long arithmetic progressions, but the technical heart of their work was a \emph{transference principle} which, roughly, states that for any subset $Y$ of a sparse pseudorandom set $X\subset [n]$, there is a \emph{dense model}, i.e., a subset $Z$ of $[n]$ whose density in $[n]$ roughly matches the relative density of $Y$ in $X$, and which contains many $k$-term arithmetic progressions only if $Y$ also contains relatively many $k$-term arithmetic progressions. The point here is that Szemer\'edi's theorem, together with an argument of Varnavides~\cite{Varnavides}, says that any subset of $\{1,\dots,n\}$ of positive density contains many $k$-term arithmetic progressions. Combining these tools, it immediately follows that if $Y$ has positive relative density, it also contains $k$-term arithmetic progressions, and in this sense the sparse result is \emph{transferred} from the dense result.

A few years later, Conlon and Gowers~\cite{ConGow} generalised this tool and obtained a transference principle for sparse random combinatorial theorems. This result solves all the problems concerning sparse random structures mentioned so far in a unified way (and some more we did not mention), and gives optimal bounds on the probability $p$ of the sparse random set one is allowed to work with. In particular, Conlon and Gowers fully resolved the random Tur\'an problem for complete graphs. Moreover, they partly answered the analogous question for general graphs~$H$ instead of merely complete graphs: What is the maximum relative density of a subgraph~$G$ of $\Gamma=G(N,p)$ that contains no copy of a fixed graph~$H$?

A limitation of the Conlon--Gowers approach is that it can only handle `strictly balanced' graphs $H$ (see Definition~\ref{def:strbal}). Around the same time, Schacht~\cite{schacht2016extremal} used quite different methods to fully resolve the random Tur\'an question, overcoming this limitation. Similar ideas were used by Friedgut, R\"odl and Schacht~\cite{FRS} for further progress on random Ramsey-type problems.
A few years later, the celebrated hypergraph container method of Balogh, Morris and Samotij~\cite{Balogh2015} and Saxton and Thomason~\cite{Saxton2015} gave a third route to all these results. For newer developments concerning the method of hypergraph containers see also \cite{BalSam_EfficientContainers,BerDelTowTse,CamSam_OptContainers,MorSamSax,Nenadov_ProbContainers}.

\medskip

The goal of this paper is to generalise the Conlon--Gowers transference principle as well as the result of Schacht, but our main result achieves substantially more. Indeed, modern extremal combinatorics additionally studies so-called counting problems: How many copies of some fixed $H$ can a graph $G$ on $n$ vertices with a given density have? How does the structure of $G$ affect this? 

For answering questions of this type, the \emph{regularity method} has been central. It consists of a regularity lemma, which asserts that any discrete structure of a given type is `almost' built from a bounded number of pseudorandom parts, and a counting lemma, which gives a formula that approximates the number of substructures in a pseudorandom setting of given densities. 
In the dense graph case, the Szemer\'edi regularity lemma~\cite{SzeReg} states that we can partition the vertices of any graph into a bounded number of pieces, such that most of the pairs of parts induce $\eps$-regular pairs (pseudorandom bipartite graphs). A corresponding counting lemma then, given parts with some $\eps$-regular pairs between them, provides a formula approximating the number of copies of a graph~$H$ with one vertex (or more vertices) in each part and edges only on pairs that are $\eps$-regular.

In the sparse setting, the required regularity lemmas have been known since the 1990s~\cite{KohSparse,KRSparse}, but the corresponding sparse counting lemmas proved more challenging. The well-known K{\L}R conjecture~\cite{KLRConj} asserts that a \emph{weak sparse counting lemma} for graphs is possible: If the conjectured counting formula predicts a positive relative density of $H$-copies, then a copy of $H$ exists. 
This is weak in the sense that it only requires one copy instead of the correct number of copies.

Greatly generalising earlier work on special cases~\cite{KLR_progressions, GerKohRodSte, GPSST, GerSchSte}, this conjecture was proven by Balogh, Morris and Samotij~\cite{Balogh2015} and independently by Saxton and Thomason~\cite{Saxton2015}.
In fact, Saxton and Thomason~\cite{Saxton2015} not only show that one $H$-copy exists, but a positive (albeit much smaller than predicted) relative density of $H$-copies. This was re-proven by Conlon, Gowers, Samotij and Schacht~\cite{CGSS}, who also show that for strictly balanced $H$
a \emph{full sparse counting lemma} holds, that is, a version where we obtain the predicted relative density of $H$-copies.
The former part of~\cite{CGSS} builds on the work of Schacht~\cite{schacht2016extremal}, while the latter uses the Conlon--Gowers transference principle~\cite{ConGow}.

In this paper, we introduce an abstract counting statement which, following established terminology, we also call \emph{transference principle}, and from which all the results mentioned above (and others) follow. In particular, improving on the work of Conlon, Gowers, Samotij and Schacht~\cite{CGSS}, we prove a full sparse counting lemma for all graphs, and also for all hypergraphs. Roughly speaking, if one is working with a combinatorial structure to which the hypergraph container method applies in the $p$-random setting, then our transference principle applies in the $p$-random setting too, and furthermore the failure probability of our transference principle is asymptotically optimal.
While this counting statement is a theorem in combinatorics, its proof relies substantially on methods of functional analysis together with linear and convex optimisation over polytopes. 

\medskip

Our main theorem (Theorem~\ref{thm:main}) needs a substantial amount of further definitions and notation. This is why in the introduction we restrict ourselves to stating two simplified versions of our transference principle for graphs. This will also be used to build intuition for the more general features our main result allows for. We then provide a list of example applications of our main result, including Tur\'an-type, Ramsey-type, and Szemer\'edi-type results. We will conclude the introduction by outlining the structure of the remainder of our paper.

\subsection{A simple version of our transference principle for graphs}
In the case of graphs, for a fixed target graph $H$, the aim of our transference principle is to provide for any $G\subseteq G(N,p)$ a graph $G'$ which 
\begin{enumerate}[label=\rom]
\item\label{itm:dense_model_i}
has similar density within $K_N$ as the density of $G$ within $G(N,p)$, and
\item satisfies that the number of embeddings of our target graph~$H$ in $G'$ relative to $K_N$ is similar to the number of embeddings of $H$ in $G$ relative to the expected number in $G(N,p)$. 
\end{enumerate}
We call such a graph $G'$ a \emph{dense model} of $G$ (with respect to the counting of $H$), and the density referred to in~\ref{itm:dense_model_i} is the \emph{relative edge density} of~$G'$ and~$G$, respectively.

For a given $p\in (0,1)$ and a positive integer $N$, the expected number 
of edges in the random graph $\Gamma=G(N,p)$ is $p\binom{N}{2}$ and the number
$c(H,\Gamma)$ of embeddings of~$H$ in $\Gamma$ (i.e.\ injective maps from $V(H)$ to $V(\Gamma)$ which map edges to edges) has expectation $\EE\big(c(H,\Gamma)\big)=\frac{N!}{(N-v(H))!} p^{e(H)}$, where we write $v(H)$ for the number of vertices of~$H$, and $e(H)$ for the number of its edges.
Hence, for $p\ll N^{-(v(H)-2)/(e(H)-1)}$
the expected number of edges in $G(N,p)$ is much higher than the expected number of $H$-copies and hence we can remove (in expectation) all $H$-copies from $\Gamma$ by removing only a tiny fraction of its edges, giving an $H$-free subgraph~$G$ of~$\Gamma$ with relative density arbitrarily close to~$1$. Therefore, a transference principle cannot hold for $p\ll N^{-(v(H)-2)/(e(H)-1)}$. 

More generally, the same is true if we can make the analogous argument for some subgraph $H'\subseteq H$ for some range of~$p$. This gives rise to the following parameter, which determines when a transference principle cannot hold for this trivial reason.

\begin{definition}[maximum $2$-density, strictly balanced]
\label{def:strbal}
  For a graph~$H$ with at least $2$ edges, the \emph{maximum $2$-density} of~$H$ is
  \[m_2(H):=\max_{H'\subseteq H, v(H')\geq 3}\,\,\frac{e(H')-1}{v(H')-2}\,.\]
  If this maximum is attained only for $H'=H$, then we say that~$H$ is \emph{strictly balanced}.
\end{definition}

As explained above, for $p\ll N^{-1/m_2(H)}$, typically there are $H$-free subgraphs of $G(N,p)$ of high relative density. However, this argument only covers what happens in expectation. Even for larger~$p$, it could still be the case that $G(N,p)$ contains far more $H$-copies than expected. 
Also in this case a transference principle is not possible: we will not be able to find an appropriate dense model as $G'$ cannot contain more copies of $H$ than $K_N$. This contextualises the following theorem, which states that with high probability we can find dense models for subgraphs~$G$ of $G(N,p)$ when $p\gg N^{-1/m_2(H)}$ and the bad event just mentioned does not hold. For any real numbers \(x,y,z\), we write \(x=y\pm z\) to indicate \(y-z\leq x\leq y+z\).

\begin{theorem}[simplified transference for graphs]\label{CountingThmForGraphs}
  Let \( H \) be a graph with $m_2(H)\ge1$ and let \( \eps > 0 \).  Let \(
  \eta_{N,p} \) be the probability that the number of $H$-copies in \(G(N,p) \)
  exceeds \( (1+\tfrac\eps3) p^{e(H)} N^{v(H)} \).  Then there exists a
  constant \( C > 0 \) such that for \( p \ge C N^{-1/m_2(H)} \) the random graph
  $\Gamma = G(N,p)$ satisfies the following with probability at least \( 1 -
  \eta_{N,p} -\exp(-pN^2/C)\).
  For every subgraph \( G \subseteq \Gamma \) there exists a graph $G'$ on $V(G)$ with
    \[
    e(G) p^{-1} = e(G') \pm \eps N^2
    \qquad \text{and} \qquad
    c(H,G) p^{-e(H)} = c(H,G')\pm\eps N^{v(H)}\,.
    \]
\end{theorem}

Observe that if $H$ has a vertex of degree at least $2$, then it contains three vertices spanning at least two edges and so $m_2(H)\ge1$. Thus the restriction $m_2(H)\ge1$ excludes matchings: we can handle matchings but with a slightly worse $p$, which we prefer to exclude for simplicity here.

By the argument above, the lower bound on $p$ is asymptotically optimal. As far as the failure probability in Theorem~\ref{CountingThmForGraphs} is concerned, for any graph~$H$ which is not a forest, it is easy to see $\eta_{N,p}\gg\exp(-pN^2/C)$. As discussed, the failure probability cannot be smaller than $\eta_{N,p}$. Hence, this is also optimal up to constants.

This result generalises the corresponding result of Conlon and Gowers~\cite{ConGow} in that it works for general graphs~$H$ and not merely strictly balanced graphs, and in that it provides an optimal failure probability. Neither the results of Schacht~\cite{schacht2016extremal} nor the hypergraph container method~\cite{Balogh2015,Saxton2015} provides a transference principle, but as explained above, they allow for Tur\'an-type results in $G(N,p)$. Restricted to this scenario, our result is more general in that it gives a lower bound on $c(H,G)$ that is $1-\eps$ times the expectation of $c(H,G)$ (compared to $\eps$ times this expectation), and in that it also gives a corresponding upper bound.

\medskip

Theorem~\ref{CountingThmForGraphs} is implied by the more general Theorem~\ref{CountingThmForColouredGraphs}, as we argue below.

\subsection{Additional features of our general transference principle}\label{subsec:additionalfeatures}
Let us now explain which more general features, compared to Theorem~\ref{CountingThmForGraphs}, our main theorem allows for.

\subsubsection{Transference of structure and counting in specific places}
While Theorem~\ref{CountingThmForGraphs} is enough to prove sparse Tur\'an-type results for graphs, it is not strong enough to prove the sparse counting lemma we are aiming for (Theorem~\ref{thm:graphcount}). Two things are missing: the possibility to transfer structure and the possibility to count in specific places only. Let us illustrate this with the help of the example of $H=C_4$, the four-vertex cycle.

We first explain what we mean by transference of structure.
Consider the example that~$G$ is a subgraph of $G(N,p)$ which consists of four disjoint vertex sets $V_1,V_2,V_3,V_4$ and sparse pseudorandom bipartite graphs on each $(V_i,V_{i+1})$ cyclically. 
Assume we would like to use our transference principle in conjunction with the dense counting lemma to count the number of $C_4$s in~$G$ which go around $V_1,\dots,V_4$.
For this, we need not just that the dense model~$G'$ has the right total number of edges, as guaranteed by Theorem~\ref{CountingThmForGraphs}, but also that they lie in four pseudorandom bipartite graphs $(V_i,V_{i+1})$ cyclically with the correct densities (which is the setup needed by the dense counting lemma).
The solution here is to use what we call \emph{structure functions}, which are functions $\sigma:E(K_N)\to[0,1]$ that can be thought of as weighting the edges. We then ask that our dense model~$G'$ has the correct number of $\sigma$-weighted edges for all $\sigma$ in a carefully chosen collection of functions, and this turns out to guarantee the necessary transference of structure. In our specific example, the collection would include one structure function $\sigma_{U,V}$ for each choice of disjoint vertex sets $U,V\subset[N]$ with $|U|=|V|\ge\tilde\eps N$ for some suitably chosen constant $\tilde\eps$, where $\sigma_{U,V}(uv)=1$ if $u\in U$ and $v\in V$ and $\sigma_{U,V}(uv)=0$ otherwise. This allows us to transfer that $(V_i,V_{i+1})$ are pseudorandom bipartite graphs in the sense that they form $\eps$-regular pairs; that is, they form pairs in which edge densities between subsets of size at least $\tilde\eps N$ are close to the density of $(V_i,V_{i+1})$. Since the number of possible choices for $U,V$ is exponential, the resulting collection will contain an exponential number of structure functions.

We now explain what we mean by counting in specific places, using the same example of counting copies of $H=C_4$ in~$G$ which go around $V_1,\dots,V_4$. For this it is not enough to count the total number of $4$-vertex cycles in $G'$ and ask for this to match the appropriately scaled number in $G$, because this would also count $C_4$s which (for instance) lie only between $V_1$ and $V_2$ and we did not want to count these. The solution here is, similarly, to use what we call \emph{subcounts}, which are functions $\omega$ from the collection of $H$-copies in $K_N$ to $[0,1]$. In our example, we use a subcount $\omega$ which puts weight $1$ on the $C_4$ copies we do want to count, and $0$ on all others. We then ask that our dense model~$G'$ has the correct number not only of all $C_4$-copies, but also of $\omega$-weighted copies. In fact, we will again be able to specify a collection of subcounts and ask that the dense model respects all of them simultaneously.

For proving our sparse graph counting lemma (Theorem~\ref{thm:graphcount}),
it turns out that we need an exponentially large collection both of structure functions and of subcounts.
This is supported by our general transference principle (Theorem~\ref{thm:GenSimp}), which allows collections of structure functions and subcounts of sizes up to $\exp(pN^2/C)$.

\subsubsection{Colours}\label{subsec:colours}
For typical applications in Ramsey theory we want to be able to also find monochromatic copies of~$H$ in a graph~$G$ whose edges are $r$-coloured.
While Theorem~\ref{CountingThmForGraphs} applies to this situation, in that it produces for each colour $i\in[r]$ a dense model of the edges of colour $i$ in $G(N,p)$, the dense models need not be compatible, in the sense that there may be some edges of $K_N$ which appear in none of the dense models and others which appear in multiple models. 
But for transferring a dense Ramsey-type result to the sparse setting, we need compatible dense models.
The following theorem achieves this.

\begin{theorem}[simplified transference for graphs with colours]\label{CountingThmForColouredGraphs}
  Let \( H \) be a graph with $m_2(H)\ge1$, let $r\ge2$,  and let \( \eps > 0 \).  Let \(
  \eta_{N,p} \) be the probability that the number of $H$-copies in \(G(N,p) \)
  exceeds \( (1+\tfrac\eps3) p^{e(H)} N^{v(H)} \).  Then there exists a
  constant \( C > 0 \) such that for \( p \ge C N^{-1/m_2(H)} \) the random graph
  $\Gamma = G(N,p)$ satisfies the following with probability at least \( 1 -
  \eta_{N,p} -\exp(-pN^2/C)\).
  For every $r$-colouring $E(\Gamma)=G_1\dcup\dots \dcup G_r$ of $\Gamma$, there
  exists an $r$-colouring $E(K_N)=G'_1\dcup\dots\dcup G'_r$ of $K_N$ such that
  for each $i\in[r]$
    \[
    e(G_i) p^{-1} = e(G'_i) \pm \eps N^2
    \qquad \text{and} \qquad
    c(H,G_i) p^{-e(H)} = c(H,G'_i)\pm\eps N^{v(H)}\,.
    \]
\end{theorem}

This coloured counting is orthogonal to the weighted counting mentioned in our previous subsection that is taken into account via subcounts and structure functions: Our general transference principle (Theorem~\ref{thm:GenSimp}) allows for both simultaneously. The proof of Theorem~\ref{CountingThmForColouredGraphs} is provided in Section~\ref{sec:GTP}.

Theorem~\ref{CountingThmForGraphs} is a corollary of the $r=2$ case of
Theorem~\ref{CountingThmForColouredGraphs}. To see this, given $G$ for
Theorem~\ref{CountingThmForGraphs} let $G_1=G$ and $G_2=\Gamma-G$ in
Theorem~\ref{CountingThmForColouredGraphs}.

\medskip

Additionally, some problems in Ramsey theory discuss not just monochromatic copies of $H$, but copies with a specified colouring. Again, our general transference principle gives counting of $H$-copies with a specified colouring.

\subsubsection{Other discrete structures}\label{subsct:otherdiscretestructures}
So far, we have only discussed special versions of our result for graphs. But, as already indicated earlier in the introduction, it is also interesting to consider transference for other discrete structures, such as, for example, uniform hypergraphs or arithmetic progressions.
Our general transference principle handles all of these in a unified way, using a formalism of Conlon and Gowers~\cite{ConGow} (which is very similar to the formalism used by Schacht~\cite{schacht2016extremal} and for hypergraph containers~\cite{Balogh2015,Saxton2015}, but differs in that it uses ordered hypergraphs).

Before turning to other structures, let us explain this formalism, once more using the example of counting copies of $H=C_4$ in subgraphs of $K_N$. We translate this into a problem of counting $4$-uniform edges in induced subhypergraphs of a specific \emph{ordered hypergraph}~$S$ on $\binom{N}{2}$ nodes. The nodes of $S$ are the edges of $K_N$, and four nodes are an ordered edge if they form a copy of $C_4$ in a cyclic order in $K_N$. Working with ordered copies is, for example, useful when we want to count coloured copies with a given colour pattern.
In this translation, a subgraph of $K_N$ corresponds to a subset of $V(S)$, and we are interested in how many edges of~$S$ the subset induces. In particular, a random graph $G(N,p)$ corresponds to a $p$-random subset $X$ of $V(S)$, a subgraph of a random graph is some $Y\subset X$, and a dense model for~$Y$ is a subset $Z$ of $V(S)$. A coloured version corresponds to partitioning $X$ into colour classes, and in this setting a dense model will be a partition of $V(S)$. Structure functions then are functions from $V(S)$ to $[0,1]$, and subcounts are functions from $E(S)$ to $[0,1]$. 

This formalism extends to many other discrete structures.
For example, arithmetic problems such as the one addressed by Szemer\'edi's theorem can be formulated in this language. Here, the ordered hypergraph~$S$ would have vertices $[n]$ and an ordered edge of size~$k$ for each $k$-term arithmetic progression in~$[n]$. Problems on $k$-uniform hypergraphs, on the other hand, can be translated analogously to the graph case.

In this language, what controls the value of $p$ such that we can prove a transference principle for $p$-random subsets of $V(S)$ is a condition on the $\kappa$-degrees of $S$, where the \emph{$\kappa$-degree} of a set of~$\kappa$ distinct vertices is the number of edges they are contained in jointly (in a suitable ordered fashion; see Definition~\ref{def:degrees} for the details). This turns out to be the same condition that is required in~\cite{Balogh2015,Saxton2015} to prove a container statement, and (as shown in those papers, and repeated in the proofs of our applications) the condition gives the optimal value of $p$ for all the problems we mentioned.

\subsubsection{Counting multiple structures}

One final feature our transference principle allows for is the counting of several structures simultaneously.
For example, one could be interested in a
dense model of a sparse subgraph, in which counts of two graphs~$H_1$ and~$H_2$ are both transferred. This is achieved in Theorem~\ref{thm:mainmulti}.

\subsection{Applications}
We do not attempt to give a list here of all the possible applications of our results. Instead, this subsection aims to give an idea of how our results lead to extensions of important previous work in some main areas of interest where the transference principle or the sparse regularity method proved useful. We expect that many other interesting applications will be discovered by others.

Since our results allow a transference of results from the dense to the sparse setting, but even in the dense setting many challenging extremal combinatorics questions remain open (such as the famous Tur\'an problem for the $4$-vertex clique in $3$-uniform hypergraphs), it is not surprising that we cannot resolve the sparse analogues of such problems. Instead, what our results typically allow us to show is that if the answer to such a problem was found in the dense setting, then it can be transferred with our tools to the sparse setting.

\subsubsection{Tur\'an-type results}
As described, a sparse random Tur\'an theorem was first proved for strictly balanced graphs and hypergraphs by Conlon and Gowers~\cite{ConGow}, and in general by Schacht~\cite{schacht2016extremal}. In both cases, what was actually proved was somewhat stronger, in that it also addressed so-called supersaturation properties. 
For dense $k$-uniform hypergraphs~$H$, \emph{supersaturation} means that there is a density $\pi(H)$ such that for all $\eps>0$ and sufficiently large $N$, there is a $k$-uniform $N$-vertex hypergraph with $\big(\pi(H)-\eps\big)\binom{N}{k}$ edges that does not contain $H$, but any $k$-uniform $N$-vertex hypergraph with $\big(\pi(H)+\eps\big)\binom{N}{k}$ edges contains not just one copy of $H$ but $\Omega(N^{v(H)})$ copies. This was proved in an influential paper of Erd\H{o}s and Simonovits~\cite{ErdosSimonovits}.

In order to formulate the sparse random Tur\'an theorems proved by Conlon and Gowers, and by Schacht, we need some further definitions. The \emph{binomial random $k$-uniform hypergraph} 
$G^{(k)}(N,p)$ is a hypergraph on vertex set $[N]$ obtained by including each set of~$k$ vertices as an edge independently with probability~$p$. 
The \emph{maximum $k$-density} of a $k$-uniform hypergraph $H$ is
\[m_k(H):=\max_{H'\subset H,v(H')\ge k+1}\,\,\frac{e(H')-1}{v(H')-k}\,,\]
and $H$ is \emph{strictly balanced} if the unique maximiser is $H$ itself.

As for graphs, if $p\ll N^{-1/m_k(H)}$ then there are typically $H$-free subhypergraphs of $G^{(k)}(N,p)$ which contain all but a tiny fraction of the edges of $G^{(k)}(N,p)$. Hence no sparse random Tur\'an theorem is possible in this range of~$p$. 
For $p\gg N^{-1/m_k(H)}$, however, a sparse random Tur\'an theorem with supersaturation does hold, as proved for strictly balanced graphs by
Conlon and Gowers~\cite{ConGow} and in general by Schacht~\cite{schacht2016extremal}.

\begin{theorem}[Conlon and Gowers~\cite{ConGow}, Schacht~\cite{schacht2016extremal}]
If $p\gg N^{-1/m_k(H)}$, then for all $\eps>0$ there is $\delta>0$ such that with probability tending to~$1$ as $N\to\infty$ any subgraph of $G^{(k)}(N,p)$ with at least $\big(\pi(H)+\eps\big)p\binom{N}{k}$ edges contains at least $\delta p^{e(H)}N^{v(H)}$ copies of $H$.
\end{theorem}

However, one can ask for a more detailed counting statement. For a given $\rho>0$ and a fixed $k$-uniform hypergraph $H$, let $g_H(\rho, N)$ be the minimal number, and $g'_H(\rho,N)$ be the maximal number, of copies of $H$ which are contained in any $N$-vertex $k$-uniform hypergraph $G$ with $\lfloor\rho\binom{N}{k}\rfloor$ edges. It can easily be shown that
the limits \[g_H(\rho):=\lim_{N\to\infty}g_H(\rho,N)N^{-v(H)}\qquad\text{and}\qquad 
g'_H(\rho):=\lim_{N\to\infty}g'_H(\rho,N)N^{-v(H)}\] exist and are uniformly continuous functions in~$\rho$, though in general these functions are not well understood. For graphs, the Clique Density Theorem of Reiher~\cite{reiher} (which in turn builds on earlier work of Razborov~\cite{Razborov:K3} and Nikiforov~\cite{Nikiforov:K4}) describes $g_{K_k}(\rho)$ for all $k$. For bipartite $H$, a famous conjecture of Sidorenko~\cite{Sidorenko}, which is known for some but not all bipartite $H$, predicts $g_H(\rho)=\rho^{e(H)}$. The function $g'_H(\rho)$ is somewhat better understood, see for instance Gerbner, Nagy and Patk\'os~\cite{GNP}. 

Conlon, Gowers, Samotij and Schacht \cite[Theorem 1.10]{CGSS} proved that results of this type can be transferred to the sparse random setting when~$H$ is strictly balanced (formally, they proved only the $g_H(\rho)$ part, but their methods readily give the upper bound $g'_H(\rho)$ too). We prove the analogue for general graphs~$H$. We note that our result also improves on~\cite[Theorem 1.10]{CGSS} in that we obtain the asymptotically optimal failure probability.

\begin{theorem}[sparse random $H$-density]\label{thm:Hdensity}
  For any $k$-uniform hypergraph $H$ with at least $2$ edges and $m_k(H)>\tfrac{1}{k}$, and any
  $\eps>0$, let \( \eta_{N,p} \) be the probability that the number of
  copies of~$H$ in \(G^{(k)}(N,p) \) exceeds \( (1+\tfrac\eps{9k!}) p^{e(H)}
  N^{v(H)} \).  Then there exists a constant $C$ such that for $p\geq
  CN^{-1/m_k(H)}$ the random hypergraph $\Gamma=G^{(k)}(N,p)$ has the following
  property with probability at least \( 1 - \eta_{N,p} -\exp(-pN^k/C)\). Every
  subgraph \( G \subseteq \Gamma \) with $\lfloor \rho p\binom{N}{k}\rfloor$
  edges contains between $(g_H(\rho)-\eps)p^{e(H)}N^{v(H)}$ and
  $(g'_H(\rho)+\eps)p^{e(H)}N^{v(H)}$ copies of $H$.
\end{theorem}

The restriction $m_k(H)>\tfrac{1}{k}$ rules out only the case that $H$ is a matching: as we will discus in Section~\ref{sec:GTP}, we can handle this case but the lower bound on $p$ is a polylogarithmic factor larger. This theorem is an immediate consequence of the hypergraph analogue of Theorem~\ref{CountingThmForGraphs}, the hypergraph counting lemma (Theorem~\ref{thm:hypercount}). The details of the proof of Theorem~\ref{thm:Hdensity} can be found in Section~\ref{sec:quick}.
Here, we briefly sketch how the graph case of Theorem~\ref{thm:Hdensity} follows from Theorem~\ref{CountingThmForGraphs} to give the reader a basic idea of how this result can be applied.

\begin{proof}[Proof sketch of the graph case of Theorem~\ref{thm:Hdensity}]
  Given $G\subset\Gamma$, we let $G'$ be its dense model as provided by
  Theorem~\ref{CountingThmForGraphs} (with input a sufficiently small $\eps'$
  depending on $H$ and $\eps$), which has close to $\rho\binom{N}{2}$ edges by
  the definition of a dense model. Therefore, by definition and uniform continuity of
  $g_H$ and $g'_H$, the graph $G'$ has between
  $(g_H(\rho)-\eps/2)\binom{N}{v(H)}$ and
  $(g'_H(\rho)+\eps/2)\binom{N}{v(H)}$ copies of $H$, and again by the definition of
  a dense model we obtain the required bounds on the number of copies of $H$ in
  $G$.
\end{proof}

In addition, various well-known questions in extremal combinatorics concern the
counting of multiple subgraphs simultaneously. For example, the Erd\H{o}s
pentagon conjecture, now a theorem of Grzesik~\cite{Grzesik} and Hatami,
Hladk\'y, Kr\'a\v{l}, Norin, and Razborov~\cite{HHKNR}, gives the maximum number
of copies of $C_5$ in a triangle-free graph on $N$ vertices.

\begin{theorem}\label{thm:dense-pentagon}
  For any triangle free graph~$G$ on~$N$ vertices we have $c(C_5,G)\le\frac{2}{625}N^5$.
\end{theorem}

The constant $\tfrac{2}{625}$ in this result is best possible since it is
attained for the balanced blow-up of $C_5$, i.e., five vertex sets of size
$\frac{N}{5}$ each, with each completely adjacent to the next in cyclic order
and no other edges.
It is natural to ask whether a sparse random analogue of this theorem exists,
and the obvious route to proving it would be to ask for a dense model which
preserves simultaneously the number of triangles and of $C_5$s.

\begin{theorem}[sparse pentagon theorem]\label{thm:pentagon}
 For all $\eps>0$ there exists $C>0$ such that, if $p\ge CN^{-1/2}$, with probability tending to $1$ as $N\to\infty$, the random graph $\Gamma=G(N,p)$ has the following property. Any subgraph $G$ of $\Gamma$ which contains no copy of $K_3$ has $c(C_5,G)\le\big(\tfrac{2}{625}+\eps\big)p^5N^5$.
\end{theorem}

The proof of
Theorem~\ref{thm:pentagon} can be found in Section~\ref{sec:pentagon}; in addition to applying our general transference principle it also requires the so-called triangle removal lemma.

More generally, replacing~$K_3$ and~$C_5$ by other graphs in
Theorem~\ref{thm:pentagon} leads to questions involving the so-called generalised
Tur\'an numbers. Results in the dense regime for this type of questions can be
transferred analogously to the sparse setting using our transference
principle. Much about generalised
Tur\'an numbers remains open, but see for example~\cite{GerbnerPalmer} for a survey on what is known.

\subsubsection{Ramsey-type results}
Ramsey's theorem~\cite{ramsey1930problem} states that for any given $k,r\ge2$ and any $k$-uniform hypergraph $H$, if $N$ is sufficiently large then any $r$-colouring of $E(K^{(k)}_N)$ contains a monochromatic copy of $H$. 
Starting with the influential work of Goodman~\cite{Goodman} and a conjecture of Erd\H{o}s~\cite{Erdos:multiplicity} famously refuted by Thomason~\cite{Thomason:disproof}, a counting-version of Ramsey's theorem, called the \emph{Ramsey multiplicity} problem, received much attention. The goal here is to determine, given a hypergraph~$H$ and a number~$r$, the largest constant $c_{H,r}$ such that for all $\eps>0$, if $N$ is sufficiently large then every $r$-colouring of $E(K^{(k)}_N)$ contains at least $(c_{H,r}-\eps)N^{v(H)}$ monochromatic copies of $H$.

Determining $c_{H,r}$ is a difficult problem in general (see for example~\cite{FoxWig} for some history and more recent advances).
We prove that, regardless, sparse analogues of Ramsey multiplicity statements exist.

\begin{theorem}[sparse Ramsey multiplicity]
\label{thm:multiplicity}
Given $r,k\ge 2$, $\eps>0$ and a $k$-uniform hypergraph $H$ with at least two edges and $m_k(H)>\tfrac{1}{k}$, there exists a positive constant $C$ such that if $p\ge CN^{-1/m_k(H)}$ then with probability 
at least $1-\exp(-pN^k/C)$, the random hypergraph $\Gamma=G^{(k)}(N,p)$ has the following property. Every $r$-edge-colouring of $E(\Gamma)$ contains at least $(c_{H,r}-\eps)p^{e(H)}N^{v(H)}$ monochromatic copies of $H$.
\end{theorem}

The proof of this theorem is provided in Section~\ref{sec:quick}.
As before, a version of this result for strictly balanced hypergraphs follows from Conlon and Gowers~\cite{ConGow} (though they did not state this explicitly). We remark that the probability bound stated in Theorem~\ref{thm:multiplicity} does not readily follow from Theorem~\ref{CountingThmForColouredGraphs} (we need to use
Theorem~\ref{thm:GenSimp} instead).

\subsubsection{Szemer\'edi-type results}
Analogously to the Ramsey multiplicity problem, one can formulate a variant of Szemer\'edi's theorem~\cite{szemeredi1975sets} with multiplicity. Let $f(k,\delta)$ 
be the largest constant such that for all $\eps>0$, if~$n$ is sufficiently large, then any subset of $[n]$ of size at least $\delta n$ contains $\big(f(k,\delta)-\eps\big)n^2$ arithmetic progressions of length~$k$. As far as we are aware, not much is known about the behaviour of $f(k,\delta)$. But, again, our methods allow for a transference to the sparse setting.
The \emph{binomial random set} $[n]_p$ is the subset~$X$ of $[n]$ obtained by including each element independently with probability~$p$.

\begin{theorem}[sparse Szemer\'edi multiplicity]
\label{thm:sparse-szem}
 Given $k\ge 3$ and $\eps>0$, there exists a positive constant $C$ such that if $p\ge Cn^{-1/(k-1)}$ then with probability at least
 $1-\exp(-pn/C)$, the binomial random set $X=[n]_p$ has the following property. For all $\delta>0$, all subsets of $X$ of size at least $\delta p n$ contain at least $(f(k,\delta)-\eps)p^{k}n^2$ arithmetic progressions of length $k$.
\end{theorem}

We remark that this result (apart from the failure probability) does also already follow from the work of Conlon and Gowers~\cite{ConGow}, but it is somewhat harder to deduce it in their setting: A technical restriction in their approach makes it necessary to consider arithmetic progressions in $\mathbb{Z}_q$ for a prime $q$, and it then has to be proven that this translates to a counting statement for $[n]$. By contrast, our machinery applies directly to $[n]$. In addition, we obtain the optimal failure probability $1-\exp(-pn/C)$. We provide a proof in Section~\ref{sec:quick}.

\medskip

As a second application concerning arithmetic progressions, we provide a new short proof of a recent theorem of Alvarado, Kohayakawa, Morris, Mota and Ortega~\cite{AKMMO}. The purpose of this application is to demonstrate how we can use our results to handle multicoloured structures.

A classical result of van der Waerden~\cite{vanderWaerden} states that for every choice of numbers~$r$ and~$k$, if~$n$ is sufficiently large, then any $r$-colouring of $[n]$ contains a monochromatic arithmetic progression of length~$k$. A well-known result of Erd\H{o}s and Graham~\cite{ErdGra} establishes a so-called \emph{canonical} version of this theorem. They proved that for every choice of~$k$, if~$n$ is sufficiently large then any colouring (without any restriction on the number of colours) of $[n]$ contains an arithmetic progression that is either monochromatic or \emph{rainbow} (that is, where each colour appears no more than once). The following theorem transfers this canonical van der Waerden theorem to the sparse setting.

\begin{theorem}[sparse canonical van der Waerden~\cite{AKMMO}]\label{thm:VdW}
 Let $k\ge3$ be an integer. There exists $C^\ast >0$ such that for any $p\ge C^\ast n^{-1/(k-1)}$ and $X=[n]_p$ the following holds
 with probability tending to~$1$ as $n\to\infty$. For any colouring of $X$, there exists a $k$-term arithmetic progression in $X$ that is either monochromatic or rainbow.
\end{theorem}

The proof of Theorem~\ref{thm:VdW} is provided in Section~\ref{sec:vdw}.

\subsubsection{Sparse counting lemmas}

As indicated earlier, an important application of our transference principle is
to prove sparse counting lemmas.  Results of this type have many interesting
applications; we refer to~\cite{CGSS} for a list of examples.

In Section~\ref{sec:hypercount} we provide a powerful general counting lemma for
hypergraphs (Theorem~\ref{thm:hypercount}). Due to the intricacies involved in
hypergraph regularity, the additional definitions and notation required to state
this hypergraph result are substantial. This is why we restrict ourselves to the
graph case in this introduction.
The following sparse counting lemma for graphs
generalises the main results of Conlon, Gowers, Samotij and Schacht~\cite{CGSS}
in that it removes the requirement for the graph~$H$ we want to count to be
strictly balanced. 

Given $p\in(0,1]$ and a graph $G$, the \emph{$p$-density} of $(U,V)$, where $U$ and $V$
  are disjoint vertex sets in $G$, is defined to be
  $d_p(U,V):=p^{-1}e(U,V)|U|^{-1}|V|^{-1}$. We say $(U,V)$ is an
  \emph{$(\eps,p)$-regular pair} in $G$ if for all $U'\subset U$ and $V'\subset
  V$ with $|U'|\ge\eps|U|$ and $|V'|\ge|V|$, we have
  $d_p(U',V')=d_p(U,V)\pm\eps$.

\begin{theorem}[sparse graph counting lemma]\label{thm:graphcount}
  Given any graph $H$ on vertex set $[v(H)]$ and $\delta,d_2>0$, there exists
  $\eps>0$ such that for any $d_1>0$ there exist $\eps^*>0$ and $C^\ast$ with the following
  property. Let \( \eta_{N,p} \) be the probability that the number of copies
  of~$H$ in \(G(N,p) \) exceeds \( (1+\tfrac13\eps^*) p^{e(H)} N^{v(H)} \).  Then for $p\ge\max(C^\ast N^{-1},C^\ast N^{-1/m_2(H)})$
  the random graph $\Gamma=G(N,p)$ has the following property with probability
  at least \( 1 - \eta_{N,p} -\exp\big(-pN^2/(2C^\ast)\big)\).

  For every
  $G\subset\Gamma$, if $V_1,\dots,V_{v(H)}\subset V(G)$ are pairwise disjoint
  sets each of size $d_1N$, and if $(V_i,V_j)$ is $(\eps,p)$-regular with
  $d_p(V_i,V_j)\ge d_2$ for each $ij\in E(H)$, then the number of copies of $H$
  in $G$ with $i\in V_i$ for each $i\in V(H)$ is
 \[(1\pm\delta)(d_1N)^{v(H)}\prod_{ij\in E(H)}p\cdot d_p(V_i,V_j)\,.\]    
\end{theorem}

As explained earlier, the power of Theorem~\ref{thm:graphcount} is unlocked in conjunction
with the sparse regularity lemma~\cite{KohSparse,KRSparse}, which allows to obtain the setup required by the sparse counting lemma.
The lower bound on~$p$ in Theorem~\ref{thm:graphcount} is optimal up to the value of $C^\ast$. The first term in this lower bound merely requires that the average degree of $G(N,p)$ is at least a large constant, which is already required for regular pairs to exist.

An example application for Theorem~\ref{thm:graphcount} can already be found
in~\cite{PartUniv}.  We remark that it is mentioned by Conlon, Gowers, Samotij and Schacht~\cite{CGSS} that a
version of their result can be
obtained for graphs that are not strictly balanced from their methods, but only at the cost of a polylogarithmic loss in the lower
bound on~$p$. The details of this were not provided.

\subsection{Structure of the remainder of the paper}
In Section~\ref{sec:GTP} we formally state our transference principle results;
our main technical theorem, from which the other results provided in that
section follow, is Theorem~\ref{thm:main}. In Section~\ref{sec:functional} we
set up the functional language which we use to state a functional version of our
transference principle (Theorem~\ref{thm:maintechint}), implying the results
from Section~\ref{sec:GTP}.  In Section~\ref{sec:overview} we provide a
high-level outline of our proofs and compare our approach to that of Conlon and
Gowers, which is the starting point for this paper.  In Section~\ref{sec:tools} a more detailed proof overview is provided,
along with additional notation and tools, a sequence of reductions of our
functional transference principle to statements concerning certain norm bounds,
as well as the main lemmas required for proving these. In
Sections~\ref{sec:intmodels} to~\ref{sec:delete} we then prove these lemmas and
a theorem required for one of these reductions. In
Section~\ref{sec:multitransfer} we provide an extension of
Theorem~\ref{thm:main} that allows the counting of multiple substructures. In
Section~\ref{sec:applproofs} we prove the applications stated in the
introduction, and in Section~\ref{sec:hypercount} we formulate and prove our
sparse counting lemma for hypergraphs. We finish with some concluding remarks in
Section~\ref{sec:concl}.

\section{The transference principle for ordered hypergraphs}\label{sec:GTP}

The aim of this section is to state our general transference principle results
in the general language of ordered hypergraphs:
Theorem~\ref{thm:GenSimp}, Theorem~\ref{thm:lowerGenSimp}, and
Theorem~\ref{thm:main}.  These results support colours, but they do not
support the counting of multiple structures. We shall formulate and prove a
transference principle for such multiple counts in
Section~\ref{sec:multitransfer}. In addition, we shall prove that we can
concentrate on proving Theorem~\ref{thm:main} in the remainder of the paper,
because it implies Theorem~\ref{thm:GenSimp} and
Theorem~\ref{thm:lowerGenSimp}. We will also show that Theorem~\ref{thm:GenSimp}
implies our simplified transference principle for graphs with colours,
Theorem~\ref{CountingThmForColouredGraphs}.

\medskip

As explained in the introduction, our results are phrased in the language of
uniform ordered hypergraphs.  Given positive integers \(n, k\geq 2\), a
\emph{\(k\)-uniform ordered hypergraph} \(S\) of size \(n\) is a \(k\)-uniform
hypergraph on \([n]\) with an order associated to each of its edges. That is,
each edge of $S$ is an ordered sequence of length \(k\) of distinct elements of
\([n]\). As explained before, for the example of counting $C_4$-copies in~$K_N$,
one would set $n=\binom{N}{2}$ and include an ordered hyperedge in~$S$ for each
ordered quadruple of edges which forms a $C_4$ in $K_N$.

Our transference result then concerns random subsets $[n]_p$ of $[n]$. For which
probabilities~$p$ we obtain our  transference depends on the so-called
codegrees present in~$S$.

\begin{definition}[codegree, maximum $\kappa$-degree]
\label{def:degrees}
Given a $k$-uniform ordered hypergraph \(S\) on $[n]$, and a sequence \(\mathbf{x}\) of length \(k\) of elements of \([n]\cup\{\ast\}\), the \emph{codegree} \(\deg_S(\mathbf{x})\) of $\mathbf{x}$ is the number of edges of \(S\) which agree with \(\mathbf{x}\) at all positions which do not equal \(\ast\).
For a positive integer \(\kappa\le k\), the \emph{maximum $\kappa$-degree} \(\Delta_\kappa(S)\) of~$S$ is the maximum value of \(\deg_S(\mathbf{x})\) over all sequences \(\mathbf{x}\) with exactly \(\kappa\) entries not equal to \(\ast\).
\end{definition}

That is, in the definition of $\deg_S(\mathbf{x})$ the entries equalling
\(\ast\) are allowed to vary, while the others are fixed to the value they have
in \(\mathbf{x}\). Our transference principle will require that the maximum
$\kappa$-degrees of~$S$ are sufficiently small (depending on $p$ and $n$).

Next, we translate the concepts of structure functions and subcounts
mentioned in the introduction to the setting of ordered hypergraphs.

\begin{definition}[structure function, subcount]
Let $S$ be a $k$-uniform ordered hypergraph on vertex set $[n]$. We call a function \(\sigma:[n]\to[0,1]\) over the set of vertices a \emph{structure function}, and a function \(\omega:S\to[0,1]\) over the set of edges a \emph{subcount}.
\end{definition}

To state our transference principle we need the following special functions.  We
denote with \(\mathbf{1}\) any function that takes value \(1\) on its domain
(whatever that domain might be). In particular, when $\one$ is a structure
function, then $\one:[n]\to \RR$ and when $\one$ is a subcount, then
$\one:S\to\RR$.  In contrast, we use \(\ind\) to denote the \emph{indicator
function} of an event, such as \(\ind(y\in Y)\).

\medskip

We would next like to translate the concept of dense models to the setting of
ordered hypergraphs. In order to support colours, we work with partitions: Given
an $[r]$-colouring of $X=[n]_p$, which we write as $Y_1\dcup\dots\dcup Y_r=X$, our
result provides a dense model $Y'_1\dcup\dots\dcup Y'_r=[n]$. 

\begin{definition}[$\eps$-good dense model]\label{def:densemodpart}
  Let $k,r\ge2$, $\eps>0$, $n\in\mathbb{N}$ and $p\in(0,1]$. Let~$S$ be a $k$-uniform ordered hypergraph on $[n]$, let $X\subset [n]$, and let $\Sigma$
  be a set of structure functions and $\Omega$ a set of subcounts.
  We say that a partition $Y'_1\dcup\dots\dcup Y'_r=[n]$ is an \emph{$\eps$-good dense model} of a partition $Y_1\dcup\dots \dcup Y_r=X$ if the following hold.
  \begin{enumerate}[label=\itmarab{DM}]
  \item\label{main:itm:sim} For every $i\in[r]$ and \(\sigma\in\Sigma\),
     \[\sum_{y\in [n]}p^{-1}\ind(y\in Y_i)\sigma(y)=\sum_{y\in [n]}\ind(y\in Y'_i)\sigma(y)\pm\eps n\,.\]
   \item\label{main:itm:sub} For every $i_1,\dots,i_k\in[r]$ and \(\omega\in\Omega\) 
     \[\sum_{s\in S}p^{-k}\ind\big(s\in Y_{i_1}\times\dots\times Y_{i_k}\big)\omega(s)=\sum_{s\in  S}\ind\big(s\in Y'_{i_1}\times\dots\times Y'_{i_k}\big)\omega(s)\pm\eps e(S)\,.\]
  \end{enumerate}
\end{definition}

The setup considering $Y\subset X$ (without colours) can be handled by this
formalism via setting $r=2$, $Y_1=Y$ and $Y_2=X\setminus Y$, and letting
$Y'=Y'_1$ be the $\eps$-good dense model for $Y$. In our example of counting
$C_4$, the set $X=[n]_p$ corresponds to a sample of $G(N,p)$, the set $Y$ is the
set of edges of our subgraph~$G$ of $G(N,p)$, and $Y'$ forms the edges of a
subgraph of $K_N$ which is the dense model of $Y$. Then~\ref{main:itm:sim} with
$\sigma=\one$ states that the number of edges in~$G$ scaled by $p^{-1}$
approximates the number of edges in the dense model; and~\ref{main:itm:sub} with
$i_1=i_2=i_3=i_4=1$ and $\omega=\one$ states that the number of $C_4$ in~$G$ scaled
by $p^{-4}$ is approximated by the number of $C_4$ in the dense model.  In the
coloured setting, \ref{main:itm:sub} allows us to count copies of graphs with a
particular colour pattern. Here, it also becomes crucial that~$S$ is an ordered
hypergraph, as this allows us to specify which edge of a copy receives which
colour.

More generally, as we shall prove formally for completeness at the end of this
section, Theorem~\ref{CountingThmForGraphs} and
Theorem~\ref{CountingThmForColouredGraphs} are both implied by the
following transference principle for ordered hypergraphs.

\begin{theorem}[transference for ordered hypergraphs]\label{thm:GenSimp}
Given $k,r\ge2$, $c\ge1$ and $\eps>0$, there exists $C>0$ such that the following holds for all sufficiently large $n$ and any $0<p\le 1$.
Let $S$ be a $k$-uniform ordered hypergraph on $[n]$, and let $\Sigma$ be a set of structure functions on $[n]$ and $\Omega$ a set of subcounts on $S$ such that
\begin{enumerate}[label=\rom]
 \item\label{itm:GS:i} $p\ge (\log^C n)n^{-1}$, 
 \item\label{itm:GS:iv} $\Delta_\kappa(S)\le cC^{1-\kappa}p^{\kappa-1}e(S)n^{-1}$ for each $\kappa\in[k]$,
 \item\label{itm:GS:ii} $|\Sigma|,|\Omega|\le \exp(C^{-1}pn)$.
\end{enumerate}
Then with probability at least
\[1-\Prob\Big(\big|S\cap [n]_p^k\big|\ge\big(1+\tfrac13\eps\big)p^ke(S)\Big)-\exp\Big(-\frac{pn}{C}\Big)\]
the random set \(X=[n]_p\) has the following property.
Every partition $Y_1\dcup\dots\dcup Y_r=X$ has an $\eps$-good dense model $Y'_1\dcup\dots\dcup Y'_r=[n]$.
\end{theorem}

In this result, conditions~\ref{itm:GS:i} and~\ref{itm:GS:iv} provide a
lower bound on the probabilities~$p$ for which we obtain transference,
where~\ref{itm:GS:i} is a technical condition we need in the proof. In the
applications we consider here, \ref{itm:GS:iv} always implies~\ref{itm:GS:i}.
It is easy to check that in our example of counting $H=C_4$ we need $p=\Omega(n^{-1/3})=\Omega(N^{-2/3})$ for~\ref{itm:GS:iv}, which readily
implies~\ref{itm:GS:i}. In the graph setting, this would only fail when~$H$ forms a matching, which for this reason we exclude for example in Theorem~\ref{CountingThmForGraphs} by requiring $m_2(H)\ge 1$. As we explained in the introduction, \ref{itm:GS:ii}
allows us to transfer more complex properties, such as counting transversal
$C_4$-copies in $\eps$-regular partitions.

The failure probability in Theorem~\ref{thm:GenSimp}
is close to optimal for the following reason. If the number of edges in the random set $S\cap[n]_p^k$ exceeds $(1+\eps)p^k e(S)$, then for $Y_1=[n]_p$ even $Y'_1=[n]$ does not contain enough edges of~$S$ to satisfy~\ref{main:itm:sub} with $\omega=\one$. But when~$S$ for example represents $H$-copies in the complete graph (or any other of the applications we present in this paper), then
$\Prob\big(|S\cap [n]_p^k|\ge(1+\eps)p^ke(S)\big)$ and $\Prob\big(|S\cap [n]_p^k|\ge(1+\tfrac13\eps)p^ke(S)\big)$ will both be exponentially small, with the factor $\frac13$ only affecting the constant in the exponent but not its order of magnitude. In addition, these probabilities will be larger than $\exp(-\frac{pn}{C})$. In other words, the term $\Prob\big(|S\cap [n]_p^k|\ge(1+\tfrac13\eps)p^ke(S)\big)$ only appears because $X=[n]_p$ might contain too many edges of $S$, and consequently its subsets might also. If we are only interested in lower bounds on the number of edges, however, we can eliminate this term, as our next result shows.

\begin{definition}[$\eps$-good lower dense model]\label{def:lowerdensemodpart}
  Let $k,r\ge2$, $\eps>0$, $n\in\mathbb{N}$ and $p\in(0,1]$. Let~$S$ be a $k$-uniform ordered hypergraph on $[n]$, let $X\subset [n]$, and let $\Sigma$
  be a set of structure functions and $\Omega$ a set of subcounts.
  We say that a partition $Y'_1\dcup\dots\dcup Y'_r=[n]$ is an \emph{$\eps$-good lower dense model} of a partition $Y_1\dcup\dots \dcup Y_r=X$ if~\ref{main:itm:sim} and the following hold.
  \begin{enumerate}[label=\itmarabp{DM},start=2]
  \item\label{main:itm:sub'}
    For every $i_1,\dots,i_k\in[r]$ and \(\omega\in\Omega\) 
    \[\sum_{s\in S}p^{-k}\ind\big(s\in Y_{i_1}\times\dots\times Y_{i_k}\big)\omega(s)\ge\sum_{s\in  S}\ind\big(s\in Y'_{i_1}\times\dots\times Y'_{i_k}\big)\omega(s)-\eps e(S)\,.\]
  \end{enumerate}
  \end{definition}

The following theorem differs from Theorem~\ref{thm:GenSimp} only in that it
provides an improved failure probability but only guarantees lower dense models.

\begin{theorem}[lower transference for ordered hypergraphs]\label{thm:lowerGenSimp}
Given $k,r\ge2$, $c\ge1$ and $\eps>0$, there exists $C>0$ such that the following holds for all sufficiently large $n$ and any $0<p\le 1$.
Let $S$ be a $k$-uniform ordered hypergraph on $[n]$, and let $\Sigma$ be a set of structure functions on $[n]$ and $\Omega$ a set of subcounts on $S$ such that
\begin{enumerate}[label=\rom]
 \item\label{itm:lGS:i} $p\ge (\log^C n)n^{-1}$, 
 \item\label{itm:lGS:iv} $\Delta_\kappa(S)\le cC^{1-\kappa}p^{\kappa-1}e(S)n^{-1}$ for each $\kappa\in[k]$,
 \item\label{itm:lGS:ii} $|\Sigma|,|\Omega|\le \exp(C^{-1}pn)$.
\end{enumerate}
Then with probability at least
\[1-\exp\Big(-\frac{pn}{C}\Big)\]
the random set \(X=[n]_p\) has the following property.
Every partition $Y_1\dcup\dots\dcup Y_r=X$ has an $\eps$-good lower dense model $Y'_1\dcup\dots\dcup Y'_r=[n]$.
\end{theorem}

As we shall show, the next result implies both Theorem~\ref{thm:GenSimp} and Theorem~\ref{thm:lowerGenSimp}. It shows that we can obtain both the improved failure probability and an $\eps$-good dense model at the cost of possibly having to delete a few vertices from $[n]_p$ before asking for a dense model. We use the following definition.

\begin{definition*}[\(\eps\)-deletion]
  Let $n\in\mathbb{N}$ and $p\in(0,1]$. For $X\subset[n]$ and \(\eps>0\), we
  say that $\tilde{X}\subset X$ is an \emph{\(\eps\)-deletion} of \(X\) if
  $|\tilde{X}|\ge(1-\eps)pn$.
\end{definition*}

With this, we can state the main technical theorem of this section.

\begin{theorem}[transference for ordered hypergraphs with $\eps$-deletion]\label{thm:main}
Given $k,r\ge2$, $c\ge1$ and $\eps>0$, there exists $C>0$ such that the following holds for all sufficiently large $n$ and any $0<p\le 1$.
Let $S$ be a $k$-uniform ordered hypergraph on $[n]$, and let $\Sigma$ be a set of structure functions on $[n]$ and $\Omega$ a set of subcounts on $S$ such that
\begin{enumerate}[label=\rom]
 \item $p\ge (\log^C n)n^{-1}$, 
 \item $\Delta_\kappa(S)\le cC^{1-\kappa}p^{\kappa-1}e(S)n^{-1}$ for each $\kappa\in[k]$,
 \item $|\Sigma|,|\Omega|\le \exp(C^{-1}pn)$.
\end{enumerate}
Then with probability at least
\[1-\exp\Big(-\frac{pn}{C}\Big)\]
the random set \(X=[n]_p\) has an $\eps$-deletion $\tilde{X}$ with the following properties.
\begin{itemize}
\item The $r$-partition $\tilde{X}\dcup\emptyset\dcup\dots\dcup\emptyset$ has the $\eps$-good dense model $[n]\dcup\emptyset\dcup\dots\dcup\emptyset$.
\item Every partition $Y_1\dcup\dots\dcup Y_r=\tilde{X}$ has an $\eps$-good dense model $Y'_1\dcup\dots\dcup Y'_r=[n]$.
\end{itemize}
\end{theorem}

The first property asserted for the $\eps$-deletion $\tilde{X}$ in this theorem merely states that when passing from~$X$ to $\tilde{X}$ we did not disproportionally delete edges from~$S$ with respect to the weightings given by our subcount functions.
Theorem~\ref{thm:main}, in turn, will follow from the main technical theorem of the next section, Theorem~\ref{thm:maintechint}, which is a reformulation of Theorem~\ref{thm:main} (and more) in a functional language. We now show how Theorem~\ref{thm:GenSimp} and Theorem~\ref{thm:lowerGenSimp} are implied by Theorem~\ref{thm:main}. We need the following Chernoff bound (see, e.g.,~\cite[Section 2.3]{Vershynin2018}).

\begin{theorem}[Chernoff bound]\label{thm:chernoff}
 Let \(X_1,\dots{},X_n\) be independent Bernoulli random variables, let \(Y=\sum_{i=1}^nX_i\), and let \(\delta\in(0,1)\). Then we have
 \[\mathbb{P}\big[Y\ge(1+\delta)\mathbb{E}[Y]\big],\mathbb{P}[Y\le(1-\delta)\mathbb{E}[Y]\big]\le\exp\big(-\tfrac{\delta^2}{3}\mathbb{E}[Y]\big)\,.\]
\end{theorem}

\begin{proof}[Proof of Theorem~\ref{thm:GenSimp} and Theorem~\ref{thm:lowerGenSimp}]
Given $k,r\ge2$, $c\ge1$ and $\eps>0$, set $\eps'=\tfrac13\eps$. Let $C\ge100\eps^{-2}$ be such that $\tfrac12C$ is sufficiently large for Theorem~\ref{thm:main} for input $k,r,c,\eps'$.

Suppose $n$ is sufficiently large and $p\in(0,1]$. Let $S, \Sigma, \Omega$ be as in Theorem~\ref{thm:GenSimp} and Theorem~\ref{thm:lowerGenSimp}, and suppose conditions~\ref{itm:GS:i}--\ref{itm:GS:ii} of these theorems hold. Then with probability at least $1-\exp(-2pn/C)$, the likely event of Theorem~\ref{thm:main} occurs for input $k,r,c,\eps'$. Additionally, by the Chernoff bound of Theorem~\ref{thm:chernoff}, with probability at most $\exp(-\eps^2 pn/27)$ the set $X=[n]_p$ has more than $(1+\eps/3)pn$ elements. By the choice of $C$, both events hold with probability at least $1-\exp(-pn/C)$.

Assuming these two likely events occur, we now show that we obtain lower dense models, establishing Theorem~\ref{thm:lowerGenSimp}. Indeed, given a partition $X=Y_1\dcup\dots\dcup Y_r$, let $\tilde{X}$ be the $\eps$-deletion guaranteed by Theorem~\ref{thm:main}. We have $|X\setminus\tilde{X}|\le\tfrac23\eps pn$. Let $\tilde{Y}_i:=Y_i\cap\tilde{X}$ for each $i\in[r]$, and let $Y'_1\dcup\dots\dcup Y'_r$ be the $\eps'$-good dense model of $\tilde{Y}_1\dcup\dots\dcup\tilde{Y}_r$. We claim $Y'_1\dcup\dots\dcup Y'_r$ is the required $\eps$-good lower dense model of $Y_1\dcup\dots\dcup Y_r$. To see this, observe that~\ref{main:itm:sub'} holds for $\tilde{Y}_1,\dots,\tilde{Y}_r$, and hence also for the supersets $Y_1,\dots,Y_r$. For proving that~\ref{main:itm:sim} holds, choose $i\in[r]$ and $\sigma\in\Sigma$. We have
\[\sum_{y\in Y_i}p^{-1}\sigma(y)=\sum_{y\in \tilde{Y}_i}p^{-1}\sigma(y)\pm\tfrac23\eps n=\sum_{y\in Y'_i}\sigma(y)\pm\eps n\,,\]
where the first equality is since $|Y_i\setminus\tilde{Y}_i|\le|X\setminus\tilde{X}|\le\tfrac23\eps pn$, and the second since $Y'_i$ is an $\eps'$-good dense model of $\tilde{Y}_i$.

For Theorem~\ref{thm:GenSimp} we need a little more work. Consider the event that~$X^k$ contains at most $(1+\eps')p^ke(S)$ edges of $S$. The failure probability of this event is by definition $\Prob\big(\big|S\cap [n]_p^k\big|\ge(1+\tfrac13\eps)p^ke(S)\big)$, hence with probability at least $1-\Prob\big(\big|S\cap [n]_p^k\big|\ge(1+\tfrac13\eps)p^ke(S)\big)-\exp(-pn/C)$ this event as well as the two good events from above hold, which we assume from now on. 
Let the partition $X=Y_1\dcup\dots \dcup Y_r$ and its $\eps$-deletion $\tilde{X}$ with parts $\tilde{Y}_i:=Y_i\cap\tilde{X}$ for $i\in[r]$ and
$\eps'$-good dense model $Y'_1\dcup\dots\dcup Y'_r$ be as before. We have already argued above that~\ref{main:itm:sim} and the lower bound of~\ref{main:itm:sub} hold. It remains to show that the upper bound in~\ref{main:itm:sub} is also true.

Indeed, since $[n]\dcup\emptyset\dcup\dots\dcup\emptyset$ is the dense model of
$\tilde{X}\dcup\emptyset\dcup\dots\dcup\emptyset$ by
Theorem~\ref{thm:main}, we have
\[\sum_{s\in S}p^{-k}\ind(s\in\tilde{X}^k)=(1\pm\eps')e(S)\,.\] It follows that
\[\sum_{s\in S}p^{-k}\ind(s\in X^k\setminus\tilde{X}^k)\le 2\eps'e(S)\,.\]
Now observe that for any given $i_1,\dots,i_k\in[r]$, we have
\begin{align*}
\sum_{s\in S}p^{-k}\omega(s)&\ind(s\in Y_{i_1}\times\dots\times Y_{i_k})\\
&=\sum_{s\in S}p^{-k}\omega(s)\big(\ind(s\in \tilde{Y}_{i_1}\times\dots\times \tilde{Y}_{i_k})+\ind(s\in Y_{i_1}\times\dots\times Y_{i_k}\setminus \tilde{Y}_{i_1}\times\dots\times \tilde{Y}_{i_k})\big)\\
&\le\sum_{s\in S}\omega(s)\ind(s\in Y'_{i_1}\times\dots\times Y'_{i_k})+\eps'e(S)+\sum_{s\in S}p^{-k}\ind(s\in X^k\setminus\tilde{X}^k)\\
&\le \sum_{s\in S}\omega(s)\ind(s\in Y'_{i_1}\times\dots\times Y'_{i_k})+\eps e(S)\,,
\end{align*}
where the first inequality uses that $Y'_i$ is the required dense model of $\tilde{Y}_i$ for each $i$, that $\omega(s)\in[0,1]$ and that if $s\in Y_{i_1}\times\dots\times Y_{i_k}\setminus \tilde{Y}_{i_1}\times\dots\times \tilde{Y}_{i_k}$, then it contains an element which is in $X$ but not $\tilde{X}$, so it is in $X^k\setminus\tilde{X}^k$. This gives the required upper bound.
\end{proof}

We complete this section by showing that Theorem~\ref{thm:GenSimp} implies
Theorem~\ref{CountingThmForColouredGraphs}. In fact, we prove the following
$k$-uniform hypergraph generalisation of Theorem~\ref{CountingThmForColouredGraphs} for
all $k\ge2$, again excluding matchings by insisting that $m_k(H)>\tfrac{1}{k}$.

\begin{theorem}[simplified transference for hypergraphs with colours]\label{thm:colHGtrans}
  Let $k\ge 2$ and let \( H \) be a $k$-uniform hypergraph with
  $m_k(H)>1/k$. Let $r\ge2$, and let \( \eps > 0 \).  Let \( \eta_{N,p} \) be
  the probability that the number of $H$-copies in \(G^{(k)}(N,p) \) exceeds \(
  (1+\tfrac\eps3) p^{e(H)} N^{v(H)} \).  There exists a constant \( C > 0 \)
  such that for \( p\ge C N^{-1/m_k(H)} \) the random hypergraph $\Gamma =
  G^{(k)}(N,p)$ satisfies the following with probability at least \( 1 -
  \eta_{N,p} -\exp(-pN^k/C)\).
  For every $r$-colouring $E(\Gamma)=G_1\dcup\dots \dcup G_r$ of $\Gamma$, there
  exists an $r$-colouring $E(K^{(k)}_N)=G'_1\dcup\dots\dcup G'_r$ of $K^{(k)}_N$ such that
  for each $i\in[r]$
    \[
    e(G_i) p^{-1} = e(G'_i) \pm \eps N^k
    \qquad \text{and} \qquad
    c(H,G_i) p^{-e(H)} = c(H,G'_i)\pm\eps N^{v(H)}\,.
    \]
\end{theorem}

\begin{proof}
 Given $r,k\ge2$, a $k$-uniform hypergraph $H$ with $m_k(H)>\tfrac{1}{k}$, and $\eps>0$, we let $c=2v(H)!$. Let $C'$ be returned by Theorem~\ref{thm:GenSimp} for input $e(H),r,c,\eps$, and let $C=10\cdot 2^kk!C'$.

 Given $N$, let $n=\binom{N}{k}$, and identify $[n]$ with the edges of the $N$-vertex complete $k$-uniform hypergraph $K^{(k)}_N$. Let $S$ be the ordered $e(H)$-uniform hypergraph on $[n]$ whose edges consist of any ordered tuple of elements of $[n]$ corresponding to a copy of $H$ in $K^{(k)}_N$. Let $\Sigma=\{\one\}$ and $\Omega=\{\one\}$. 

 Suppose that $p\ge CN^{-1/m_k(H)}$. Since $m_k(H)>\tfrac1k$, it follows that $H$ is not a matching, and there is a vertex in at least two edges of $H$. This gives immediately $2k-1$ vertices of $H$ which span at least two edges, so $1/m_k(H)\le k-1$. Thus $pn=p\binom{N}{k}\ge \tfrac1{2k!}CN>C'(\log^{C'} n)$ for sufficiently large $N$, verifying condition~\ref{itm:GS:i} of Theorem~\ref{thm:GenSimp}. Condition~\ref{itm:GS:ii} also immediately follows.

 It remains to check condition~\ref{itm:GS:iv}. Indeed, by construction, we have $e(S)=(1+o(1))N^{v(H)}e(H)!$, where the error term accounts for the fact that there are slightly fewer than $N^{v(H)}$ injective maps from $V(H)$ to $[N]$.
 Given \(1\le \kappa\le e(H)\), let \(\mathbf{x}\) be a sequence of length \(e(H)\) from \([n]\cup\{\ast\}\) with exactly \(\kappa\) entries not equal to \(\ast\). For \(\kappa=1\), by symmetry we have \(\deg_S(\mathbf{x})=\tfrac{e(S)}{n}\), which is as required. We now assume \(\kappa\ge2\). 
 Let \(W\subset[N]\) be the vertices of \(K^{(k)}_N\) which are contained in some edge in \(\mathbf{x}\). By definition, if \(\mathbf{x}\) has two identical non-\(\ast\) entries, then \(\deg_S(\mathbf{x})=0\), so we can assume that \(\mathbf{x}\) has at least two distinct non-\(\ast\) entries, and hence \(|W|\ge k+1\). By definition of \(m_k(H)\), we have
 \[\frac{\kappa-1}{|W|-k}\le m_k(H)\qquad\text{and hence}\qquad |W|\ge\frac{\kappa-1}{m_k(H)}+k\,.\]
 To obtain a member of \(S\) which agrees with \(\mathbf{x}\) at the non-\(\ast\) coordinates, we can at most pick a further \(v(H)-|W|\) vertices in $[N]$, one of the at most \(v(H)!\) maps from the vertices of \(H\) to the picked vertices together with \(W\), and one of the at most $e(H)!$ orderings of the remaining edges. Thus, we have
 \begin{align*}
  \deg_S(\mathbf{x})&\le  N^{v(H)-|W|}v(H)!e(H)!
  \le 2v(H)!e(S)N^{-|W|}
  \le 2v(H)!e(S)N^{-\tfrac{\kappa-1}{m_k(H)}-k} \\
  &= 2v(H)!\tfrac{e(S)}{N^k}\big(N^{-1/m_k(H)}\big)^{\kappa-1}
  \le 2v(H)!\tfrac{e(S)}{n}(p/C)^{\kappa-1}\,,
 \end{align*}
which is the bound required for~\ref{itm:GS:iv}.

 The failure probability for Theorem~\ref{thm:GenSimp} is by choice of $C$ smaller than that claimed in Theorem~\ref{thm:colHGtrans}, as required. Suppose that $\Gamma$, which by our identification is both $[n]_p$ and $G^{(k)}(N,p)$, satisfies the likely event of Theorem~\ref{thm:GenSimp}. Given an $r$-colouring $G_1\dcup\dots\dcup G_r$ of $\Gamma$, viewing $\Gamma$ as a subset of $[n]$, Theorem~\ref{thm:GenSimp} returns an $r$-colouring $G'_1\dcup\dots\dcup G'_r$ of $[n]$, which we view as an $r$-colouring of $E(K^{(k)}_N)$.
 Using~\ref{main:itm:sim} with $\sigma=\one\in\Sigma$, we conclude for each $i\in[r]$ that
 \[p^{-1}e(G_i)=e(G'_i)\pm\eps n=e(G'_i)\pm\eps N^k\,,\] which is one of the two required conclusions of Theorem~\ref{thm:colHGtrans}.
 Using~\ref{main:itm:sub} with $\omega=\one\in\Omega$, and $i_1=\dots=i_{e(H)}=i\in[r]$, we obtain
 \[p^{-e(H)}e(H)! c(H,G_i)=e(H)! c(H,G'_i)\pm \eps e(S)=e(H)! c(H,G'_i)\pm\eps e(H)! N^{v(H)}\,,\]
 where the factor $e(H)!$ appears since each copy of $H$ corresponds to $e(H)!$ edges in $S$ (all the possible orderings), and we have $e(S)\le N^{v(H)}e(H)!$ since a copy of $H$ in $K^{(k)}_N$ corresponds to an injective map from $H$ into $V(K^{(k)}_N)$, of which there are at most $N^{v(H)}$, and any ordering of the edges of $H$ then gives us an edge of $S$. This is the second required conclusion of Theorem~\ref{thm:colHGtrans}.
\end{proof}

In this proof we used $m_k(H)>\tfrac{1}{k}$ to verify that $pn>C'(\log^{C'}n)$. This implies that we can also handle the case that $m_k(H)=\tfrac{1}{k}$ (when $H$ is a matching) at the cost of requiring an extra polylog factor in the lower bound on $p$. This corresponds to random hypergraphs with
a polylogarithmic number of edges in expectation. We feel that this regime is uninterestingly sparse.

\section{The transference principle in functional notation}\label{sec:functional}

Our aim in this section is to reformulate Theorem~\ref{thm:main} in functional notation. This approach is similar (but not identical) to the approach of Conlon and Gowers~\cite{ConGow}. 

We start with a translation of the dense set $[n]$, the sparse set $X=[n]_p$ and its $\eps$-deletion $\tilde X\subset X$ from Theorem~\ref{thm:main} to the functional language.
We will always write $\mu$ for the function $\mu:[n]\to\RR$ defined by \[\mu(x)=p^{-1}\ind(x\in X)\,.\] Similarly, we write $\tmu$ for the function \[\tmu(x)=p^{-1}\ind(x\in\tilde{X})\,.\] We will think of $\mu$ and $\tmu$ as scaled indicator functions of the sparse sets $X$ and $\tilde{X}$ respectively, while the indicator function of the dense set $[n]$ is $\one$ (and we do not scale it).

Next, we translate the concept of taking subsets to the functional setting. 

\begin{definition}[bounded, extreme]
\label{def:extreme}
    Given \(h:[n]\to\mathbb{R}_{\ge0}\), we say \(f:[n]\to\mathbb{R}_{\ge0}\) is \emph{\(h\)-bounded} if \(0\le f(x)\le h(x)\) for all \(x\in [n]\). More generally, given a collection \(P\) of functions from \([n]\) to \(\mathbb{R}_{\ge0}\), we say that \(f:[n]\to\mathbb{R}_{\ge0}\) is \emph{\(P\)-bounded} if there exists \(h\in P\) such that \(f\) is \(h\)-bounded. Suppose that \(f\) is \(P\)-bounded. We say that \(f\) is \emph{$h$-extreme} if for every \(x\in [n]\) we have either \(f(x)=0\) or \(f(x)=h(x)\), and $f$ is $P$-extreme if there exists $h\in P$ such that $f$ is $h$-extreme.
\end{definition}

In this language, saying a function is 
$\one$-extreme (respectively, $\mu$- or $\tmu$-extreme) corresponds exactly to saying that the set represented by the function is a subset of $[n]$ (respectively, $X$ or $\tilde{X}$), while
$\one$-boundedness (respectively, $\mu$- or $\tmu$-boundedness) is a generalisation of this. In particular, the fact that $Y\subset\tilde{X}$ corresponds to the function $y\to p^{-1}\ind(y\in Y)$ being $\tmu$-extreme. 

Next, we want to formulate what it means for $Y'_1\dcup\dots\dcup Y'_r=[n]$ to be a dense model for $Y_1\dcup\dots\dcup Y_r=\tilde{X}$ in functional language. We start with~\ref{main:itm:sim}.
For a structure function $\sigma\in\Sigma$, the quantity we want to control is
\[\sum_{y\in[n]}p^{-1}\ind(y\in Y_i)\sigma(y)\,.\]
Observe that this looks like an inner product. Specifically, for any $f,g:[n]\to\RR$, we define the inner product
\begin{equation}\label{definitioninnerproduct}
  \langle f,g\rangle:=\frac{1}{n}\sum_{x\in[n]}f(x)g(x)\,.
\end{equation}
With this definition, if $f_i$ is the function $x\to p^{-1}\ind(x\in Y_i)$ and $g_i$ is the $\one$-extreme function $x\to\ind(x\in Y'_i)$, then~\ref{main:itm:sim} becomes $\langle f_i ,\sigma\rangle=\langle g_i,\sigma\rangle\pm\eps$, or equivalently,
\begin{equation*}\tag{{\it DM}\,1*}\label{eq:DM1}
  \big|\langle f_i-g_i,\sigma\rangle\big|\le\eps \qquad \text{for each $i\in[r]$ and $\sigma\in\Sigma$}\,,
\end{equation*}  
that is, all $f_i-g_i$ have small inner product with all structure functions in $\Sigma$.
We will stick to using the letter $f$ for `sparse' functions, usually $\tmu$-bounded, and $g$ for `dense' functions, usually $\one$-bounded, in what follows.

Our aim now is to show that a similar statement about inner products with $f_i-g_i$ gives~\ref{main:itm:sub}. This requires a few more definitions.

\begin{definition}[$S_i(x)$, convolution]
  Let~$S$ be a \(k\)-uniform ordered hypergraph on \([n]\), let \(x\in[n]\) and
  \(i\in [k]\). Let \(f_1,\dots{},f_{i-1},f_{i+1},\dots,f_k\colon[n]\to \RR\), and let \(\omega\colon S\to [0,1]\) be a subcount.
  We write \(S_i(x)\) for the subset of \(S\) consisting of all edges whose \(i\)-th entry is \(x\).

  The \emph{convolution}
  \(\ast_{i,S,\omega}(f_1,\dots{},f_{i-1},f_{i+1},\dots{},f_k)\) is the function
  from $[n]$ to $\RR$ defined as follows. For \(x\in [n]\),
  \[\ast_{i,S,\omega}(f_1,\dots{},f_{i-1},f_{i+1},\dots{},f_k)(x):=\frac{n}{e(S)}\sum_{s\in S_i(x)}\omega(s)\prod_{j\neq i}f_j(s_j)\,.\]
  When $S$ is clear from the context (which is almost always the case) we omit
  it. Furthermore, we will simply write $\ast_{i,\omega}(f_1,\dots,f_k)$ and let
  it be implicit that $f_i$ is not one of the arguments.
\end{definition}

To illustrate the relevance of this definition, consider as before the case that each $f_j$ is $\ind(x\in Y'_j)$. Then $\ast_{i,\one}(f_1,\dots,f_k)(x)$ counts (with a scaling factor $\tfrac{n}{e(S)}$) the number of edges of $S$ whose $i$-th member is $x$, and whose $j$-th member, for each $j\neq i$, is in $Y'_j$. For a general subcount $\omega$ replacing $\one$, we count these edges with weight according to $\omega$.

Note that the convolution operator is multilinear, i.e., it is linear in each of the $f_j$.
We now connect this definition to~\ref{main:itm:sub}. Filling in the definition of the inner product, we get
\begin{equation}\label{eq:dotProductForMainThm}
  \big\langle f_{i},\ast_{i,\omega}(f_1,\dots{},f_k)\big\rangle
  = \frac{1}{n}\sum_{x\in[n]} f_i(x) \cdot \frac{n}{e(S)}\sum_{s\in S_i(x)}\omega(s)\prod_{j\neq i}f_j(s_j)
  = \frac{1}{e(S)}\sum_{s\in S} \omega(s)\prod_{j=1}^k f_j(s_j)\,.    
\end{equation}
Observe that the right-hand side of~\eqref{eq:dotProductForMainThm} does not contain $i$, which means that
\[\big\langle f_{i},\ast_{i,\omega}(f_1,\dots{},f_k)\big\rangle=\big\langle f_{i'},\ast_{i',\omega}(f_1,\dots{},f_k)\big\rangle\qquad\text{for each $i,i'\in[k]$}\,.\]
Observe further that
when $f_j$ is $\ind(x\in Y'_j)$, then~\eqref{eq:dotProductForMainThm} implies that
~\ref{main:itm:sub} asserts
\[\big\langle f_i,\ast_{i,\omega}(f_{j_1},\dots,f_{j_k})\big\rangle=\big\langle g_i,\ast_{i,\omega}(g_{j_1},\dots,g_{j_k})\big\rangle\pm\eps \qquad\text{for each $j_1,\dots,j_k\in[r]$ and $\omega\in\Omega$}\,,\]
where the parameter~$i$ is irrelevant by the previous observation.
As we will show at the end of this section (in the proof of Theorem~\ref{thm:main}), we obtain this assertion if we can show that 
\begin{equation*}\label{eq:dotProductToBound}\tag{{\it DM}\,2*}
   \big|\big\langle f_i-g_i,\ast_{i,\omega}(f_{j_1},\dots{},f_{j_{i-1}},g_{j_{i+1}},\dots{},g_{j_k})\big\rangle\big|\le\eps k^{-1} \qquad\text{for each $j_1,\dots,j_k\in[r]$ and $\omega\in\Omega$}\,.
\end{equation*} 
Observe that again, this requires that $f_i-g_i$ has a small inner product with a collection of functions.

\medskip

In our translation of Theorem~\ref{thm:main} to the functional setting, it turns out that for the proof it will actually be handy to control the inner product of more than just these functions.  Our next aim is to specify precisely which functions we want to have small inner product with $f_i-g_i$. These will form a much larger polytope of functions. Here by \emph{polytope} we mean the convex hull of a finite set of points in a finite-dimensional vector space. Given a polytope \(\Phi\), the \emph{vertex set} of \(\Phi\) is the unique minimal set \(V\) of points whose convex hull equals \(\Phi\).
The functions in this larger polytope will be called test functions for the following reason. We want that $f_i-g_i$ is anti-correlated with each of these functions, that is, $f_i-g_i$ has small components in the direction of each of these functions. In other words, the functions in this polytope will be used to test \emph{anti-correlation}. The polytope of functions we will use is defined as follows. For our translation of Theorem~\ref{thm:main} to the functional setting we shall use this definition with $P=\{\tmu,\one\}$.
 
\begin{definition}[test functions, test polytope]\label{def:polytope}
  Let $S$ be a \(k\)-uniform ordered hypergraph on \([n]\), let $\Sigma$ be a set of structure functions and let $\Omega$ be a set of subcounts. Let \(P\) be a set of functions from $[n]$ to $\RR_{\ge0}$.
A function \(\phi\) is a \emph{\(P\)-test function} if
\begin{enumerate}[label=\rom]
\item $\phi\in\Sigma$, or
\item $\phi=\ast_{i,\omega}(f_1,\dots,f_k)$ for some \(P\)-bounded \(f_1,\dots{},f_k\), some \(i\in [k]\), and some \(\omega\in \Omega\).
\end{enumerate}
The \emph{test polytope} \(\Phi(P, S,\Sigma, \Omega)\subset\mathbb{R}^{[n]}\) is the convex hull of all \(P\)-test functions \(\phi\) and their negatives \(-\phi\). When $S$, $\Sigma$ and $\Omega$ are clear from the context, we simply write $\Phi(P)$. When we specify the elements of $P$, we drop the set brackets.
\end{definition}

For example, recalling that the set of $\one$-bounded functions is $F=\{f\colon 0\le f(x)\le 1\}$, we see that 
$\Phi(\one)$ is the convex hull of $\Sigma\cup\{\ast_{i,\omega}(f_1,\dots,f_k):i\in[k],\omega\in \Omega,f\in F\}$. Apart from $\Phi(\one)$, we shall be particularly interested in $\Phi(\tmu,\one)$, since it contains all functions we took inner products with in~\eqref{eq:DM1} and~\eqref{eq:dotProductToBound}.
This is why we can now say that our goal is to show that
\begin{equation*}\tag{{\it DM}\,*}\label{eq:DM:product}
  \langle f_i-g_i,\phi\rangle\le\eps k^{-1} \qquad \text{for all $\phi\in\Phi(\tmu,\one)$}\,,
\end{equation*}
  since this implies both~\eqref{eq:DM1} and~\eqref{eq:dotProductToBound}. Note that the absolute value around the inner product that we had earlier is no longer required since the polytope is \emph{centrally symmetric} by definition, which means $\phi\in\Phi$ if and only if $-\phi\in\Phi$.

It is a standard fact that for any centrally symmetric bounded polytope $\Phi\subset\RR^{[n]}$ which is $n$-dimensional
\[\|f\|_\Phi:=\max_{\phi\in\Phi}\langle f,\phi\rangle\]
defines a norm for $f\in\RR^{[n]}$. In order to automatically get a polytope $\Phi(P)$ which is $n$-dimensional, we shall impose the condition that $\Sigma\subset\Phi(P)$ contains all the functions $x\to\ind(x=i)$; since we always allowed~$\Sigma$ to be exponentially large, including these few additional functions is not problematic. Since $\Phi(P)$ is the convex hull of a finite, centrally symmetric collection of vertices, it is also bounded and centrally symmetric. Thus $\|\cdot\|_{\Phi(P)}$ is a norm for any~$P$. We remark that, since $\Phi(P)$ is centrally symmetric, the statement that  $\|f\|_{\Phi(P)}\le b$ for some function~$f$ and bound~$b$ is equivalent to the statement that $|\langle f,\phi\rangle|\le b$ for all $\phi\in\Phi(P)$, and we shall use these two statements interchangeably throughout the paper.

In this formalism, \eqref{eq:DM:product} becomes

\begin{equation*}\tag{{\it DM}}
 \|f_i-g_i\|_{\Phi(\tmu,\one)}\le\eps k^{-1}\,.
\end{equation*}

When this is the case for all~$i$, then we also say that $g_1,\dots,g_r$ is an
\emph{$\eps$-good dense model} of $f_1,\dots,f_r$. The main technical result of
this paper then is the following functional version of Theorem~\ref{thm:main}.

\begin{theorem}[main technical theorem]\label{thm:maintechint}
Given $k,r\ge2$, $c\ge1$ and $\eps>0$, there exists $C>0$ such that the following holds for all sufficiently large $n$ and any $0<p\le 1$.
Let $S$ be a $k$-uniform ordered hypergraph on $[n]$, and let $\Sigma$ be a set of structure functions on $[n]$ and $\Omega$ a set of subcounts on $S$ such that
\begin{enumerate}[label=\rom]
\item $p\ge (\log^C n)n^{-1}$, 
\item $\Delta_\kappa(S)\le cC^{1-\kappa}p^{\kappa-1}e(S)n^{-1}$ for each $\kappa\in[k]$,
\item $|\Sigma|,|\Omega|\le \exp(C^{-1}pn)$,
\item\label{maintechint:iii} $\one\in\Sigma$, $x\to\ind(x=i)\in\Sigma$ for each $i\in[n]$, and $\one\in\Omega$.
\end{enumerate}
Then with probability at least \(1-\exp\big(-\tfrac{pn}{C}\big)\) the random set \(X=[n]_p\) has an $\eps$-deletion $\tilde{X}$
with the following properties.
\begin{itemize}
  \item $\|\tmu-\one\|_{\Phi(\tmu,\one)}\le\eps k^{-1}$.
  \item For any $f_1,\dots,f_r\colon[n]\to\RR_{\ge0}$ with
    $f_1+\dots+f_r=\tmu$, there exist $g_1,\dots,g_r\colon[n]\to\{0,1\}$ with
    $g_1+\dots+g_r=\one$, such that
    \(\|f_i-g_i\|_{\Phi(\tilde{\mu},\mathbf{1})}\le\eps k^{-1}\) for each
    $i\in[r]$.
\end{itemize}
\end{theorem}

After the explanations above, the proof that Theorem~\ref{thm:main} follows from Theorem~\ref{thm:maintechint} is now straightforward. For completeness, we provide the details here.

\begin{proof}[Proof of Theorem~\ref{thm:main}]
  Given the setting of Theorem~\ref{thm:main}, we let $C'$ be given by
  Theorem~\ref{thm:maintechint} for the same parameter values and set the $C$
  promised by Theorem~\ref{thm:main} to $2C'$. We also add $\one$ and
  $x\to\ind(x=i)$ for each $i\in[n]$ to $\Sigma$; we add $\one$ to
  $\Omega$. Since $p\ge (\log^C n)n^{-1}$ by assumption, after this addition we
  still have $|\Sigma|,|\Omega|\le 2\exp(C^{-1}pn)\le\exp\big((C')^{-1}pn\big)$,
  as required for Theorem~\ref{thm:maintechint}.
  
 Suppose $X=[n]_p$ admits an $\eps$-deletion $\tilde{X}$ satisfying the likely event of Theorem~\ref{thm:maintechint}. Let $Y_1\dcup\dots\dcup Y_r$ be any partition of $\tilde{X}$. For each $1\le i\le r$, let the function $f_i$ be defined by $f_i(y)=p^{-1}\ind(y\in Y_i)$. Then $f_1,\dots,f_r$ satisfy the conditions of Theorem~\ref{thm:maintechint}, so by that theorem there exist $g_1,\dots,g_r$ taking values in $\{0,1\}$ with $g_1+\dots+g_r=\one$, such that $\|f_i-g_i\|_{\Phi(\tmu,\one)}\le\eps k^{-1}$ for each $i$.
 
 For each $i\in[r]$, define $Y'_i=\{z\in[n]\,:\,g_i(z)=1\}$. Then $Y'_1,\dots,Y'_r$ form a partition of $[n]$ as required. We claim that these are the required $\eps$-good dense models.
 For the structure statement~\ref{main:itm:sim}, given $\sigma\in\Sigma$ and any $i\in[r]$, recall that $\sigma\in\Phi(\tmu,\one)$, so that we have
 \[\sum_{y\in Y_i}p^{-1}\sigma(y)=n\langle f_i,\sigma\rangle=n\langle g_i,\sigma\rangle\pm\eps k^{-1} n=\sum_{y\in Y'_i}\sigma(y)\pm\eps n\]
 as required.

 For the counting statement~\ref{main:itm:sub}, we will prove more generally
 that for any collection of~$k$ (not necessarily disjoint) subsets of~$\tilde
 X$, whose functional versions have good dense models (which could have been
 obtained from different applications of Theorem~\ref{thm:maintechint}) we
 obtain a good dense model for the subsets. Indeed, let $\tilde Y_1,\dots,\tilde Y_k$ be any subsets of $\tilde{X}$ with corresponding $\tmu$-extreme functions $\tilde f_1,\dots,\tilde f_k$, that is, $\tilde f_i(y)=p^{-1}\ind(y\in \tilde Y_i)$  for each $i\in[k]$. Suppose that we have functions $\tilde g_1,\dots,\tilde g_k\colon[n]\to\{0,1\}$ such that for each $i\in[k]$ we have $\|\tilde f_i-\tilde g_i\|_{\Phi(\tmu,\one)}\le\eps k^{-1}$.  By~\eqref{eq:dotProductForMainThm}, for any $\omega\in\Omega$ we have
 \begin{equation}\label{eq:main:conv}
   \sum_{s\in S}\omega(s)\prod_{i=1}^k\tilde f_i(s_i)=e(S)\big\langle \tilde f_1,\ast_{1,\omega}(\tilde f_2,\dots,\tilde f_k)\big\rangle=e(S)\big\langle \tilde f_j,\ast_{j,\omega}(\tilde f_1,\dots,\tilde f_{j-1},\tilde f_{j+1},\dots,\tilde f_k)\big\rangle
 \end{equation}
 for each $2\le j\le k$, and a similar statement holds replacing any $\tilde f$ functions with $\tilde g$ functions. In particular, for each $j$ the function
 \[\ast_{j,\omega}(\tilde g_1,\dots,\tilde g_{j-1},\tilde f_{j+1},\dots,\tilde f_k)\]
 is in $\Phi(\tmu,\one)$, and hence we have
 \[\big\langle \tilde f_j,\ast_{j,\omega}(\tilde g_1,\dots,\tilde g_{j-1},\tilde f_{j+1},\dots,\tilde f_k)\big\rangle
 =\big\langle \tilde g_j,\ast_{j,\omega}(\tilde g_1,\dots,\tilde g_{j-1},\tilde f_{j+1},\dots,\tilde f_k)\big\rangle\pm\eps k^{-1}\,.\]
 We can then sequentially replace each $\tilde f_i$ with $\tilde g_i$ in the inner product, each time incurring an error bounded by $\eps k^{-1}$. Putting this into~\eqref{eq:main:conv}, we obtain
\begin{equation}\label{eq:techtomain:omega}
 \sum_{s\in S}\omega(s)\prod_{i=1}^k\tilde f_i(s_i)=\sum_{s\in S}\omega(s)\prod_{i=1}^k\tilde g_i(s_i)\pm\eps e(S)\,.
\end{equation}
 For rewriting this in the setting of Theorem~\ref{thm:main},
let ${\tilde Y}'_1,\dots,{\tilde Y}'_k$ be the subsets of $[n]$ corresponding to $\tilde g_1,\dots,\tilde g_k$, that is, $\tilde g_i(y)=\ind(y\in\tilde{Y}'_i)$ for each $i\in[k]$. Then~\eqref{eq:techtomain:omega} says
\[\sum_{s\in S}\omega(s)p^{-k}\ind\big(s\in \tilde Y_1\times\dots\times \tilde Y_k\big)
=\sum_{s\in S}\omega(s)\ind\big(s\in {\tilde Y}'_1\times\dots\times {\tilde Y}'_k\big)\pm\eps e(S)\]
which is the required counting statement.

The additional conclusion $\|\tmu-\one\|_{\Phi(\tmu,\one)}\le\eps k^{-1}$ directly translates, as above, to the required additional conclusion of Theorem~\ref{thm:main} that $[n]\dcup\emptyset\dcup\dots\dcup\emptyset$ is an $\eps$-good dense model of $\tilde{X}\dcup\emptyset\dcup\dots\dcup\emptyset$.
\end{proof}

Much of the formalism introduced above already appears in the work of Conlon and Gowers~\cite{ConGow}. Let us close this section by explaining what is different and what is not. The last proof is an easy generalisation of Conlon and Gowers' Lemma~3.3, and the norm $\|\cdot\|_{\Phi(\tmu,\one)}$ is essentially the norm that they are `tempted' to define, but which they proceed to explain leads to difficulties; in the end, they use a significantly more complicated (and less helpful) norm, which requires them to introduce several more ideas before they can define it. One main contribution in this paper is to resolve these difficulties mentioned by Conlon and Gowers.

Our structure functions are essentially already used in~\cite{ConGow}: Conlon and Gowers explain how to modify their proof to obtain a structural conclusion which corresponds to the $\{0,1\}$-valued case of our structure functions, but in fact the general $[0,1]$-valued case would cause no difficulties for their proof. However, the collection $\Omega$ of subcounts is not found in their work; and they are only able to deal with the rather special case of `strictly balanced' $S$.

\section{Overview of the proof of the main technical theorem}
\label{sec:overview}

Sections~\ref{sec:tools} to~\ref{sec:delete} are dedicated to the proof of our
main technical theorem, Theorem~\ref{thm:maintechint}. In this section we
provide a brief high-level overview of how we proceed in this proof. More
detailed explanations will be provided in the individual sections capturing
different parts of this proof, and all the main lemmas and intermediate results
are collected in the following Subection~\ref{sec:main-lemmas}.

The important part of Theorem~\ref{thm:maintechint} is the second bullet point,
guaranteeing the existence of good dense models. The other bullet point
(which
in Theorem~\ref{thm:main} corresponded to the fact that in the $\eps$-deletion
we did not disproportionally delete edges in places where the subcount functions
tell us we should not)
is a statement which we will anyway prove along the way.
We ignore it in this overview. In addition, we assume for simplicity in this
sketch that the number of colours is $r=2$. This means that we can concentrate
on the task of providing a good dense model~$g$ for one function~$f$ with $0\le
f\le\tmu$, representing the first colour; this will, using the triangle
inequality, automatically give a good dense model for the second colour. For
now, we will also assume that $\tmu=\mu$, that is, that we do not perform any
deletions. We will explain at the end of this overview what needs to be changed in
order to deal with deletions.

Observe that the norm requirement $\|f-g\|_{\Phi(\tmu,\one)}\le\eps k^{-1}$ of
the second bullet point of Theorem~\ref{thm:maintechint} is equivalent (by
definition of the norm) to a finite collection of linear inequalities all being
satisfied, so that the existence of a $\{0,1\}$-valued good dense model~$g$
for~$f$ amounts to~$f$ being close to an integer point in a certain polytope.
Our first step is to use randomised rounding to relax this to the existence of a
$[0,1]$-valued good dense model~$g$ for~$f$, at the cost of needing $\eps$
somewhat smaller (see Theorem~\ref{thm:intmodel}). This, in turn, corresponds
to~$f$ being contained in a related polytope.

The next step is to proceed by contradiction. If~$f$ is not contained in the
related polytope, then there is a hyperplane which witnesses this by the
separating hyperplane theorem. We show that the normal vector of this hyperplane can
be taken to be of the form $\phi$ for some $\phi\in\Phi(\tmu,\one)$, and we
argue that a contradiction is obtained if
$\langle\tmu-\one,\max(0,\phi)\rangle<\eps$, where the maximum is taken
pointwise. This reduces the entire problem to proving the anti-correlation
statement that $\langle\tmu-\one,\max(0,\phi)\rangle<\eps$ holds for all
$\phi\in\Phi(\tmu,\one)$ (see Lemma~\ref{lem:HBreduce}).

Our basic idea then is to use a concentration inequality to argue that this
anti-correlation statement is very likely to hold for each vertex $\phi$ of
$\Phi(\tmu,\one)$ individually, take the union bound over vertices, and conclude
that because the anti-correlation statement holds for all vertices, it will also hold for all other points of 
$\Phi(\tmu,\one)$. This is unfortunately not possible as such for several
reasons, but circumventing these will lead to a proof.

The first reason the argument is invalid is that the function $\max(0,\phi)$ is
not linear, so the extreme inner products may not occur at vertices.  To deal
with this, we observe that $x\to\max(0,x)$, though not linear, is continuous,
so by the Stone-Weierstrass approximation theorem, for any given bounded
interval, there is a polynomial $\pi$ of some constant degree $d$ which closely
approximates $\max(0,x)$ on the given interval. The intuition is that it is then
enough to show $\langle\tmu-\one,\pi(\phi)\rangle$ is small for all
$\phi\in\Phi(\tmu,\one)$, and in turn this follows from a linear optimisation
over the \emph{product polytope} $\Phi(\tmu,\one)^d$, that is, the convex hull
of vectors which are the pointwise product of any $d$ members of
$\Phi(\tmu,\one)$. This intuition is not quite valid, because entries of
$\Phi(\tmu,\one)$ can be unboundedly large, hence outside the interval on which
we are guaranteed a close approximation. However, for a suitable choice of
interval ($[-2c,2c]$, in fact) it is close enough to true: exceptionally large
entries are very rare and we show they make only a tiny contribution to the
inner products we are interested in. (For this step, see Lemma~\ref{lem:lincor}.)

Having dealt with nonlinearity, it is now enough to show that
$\langle\tmu-\one,\phi\rangle$ is small for all $\phi$ in the product polytope
$\Phi(\tmu,\one)^d$. This is a linear problem, so the extrema occur at
vertices. However, two more problems with the `concentration and union bound'
approach arise here. The product polytope turns out to have far too many
vertices for the union bound. Furthermore, we would really like to reveal a
vertex $\phi$ and then use Bernstein's inequality to argue that $\tmu-\one$ is
very unlikely to correlate with it. But these two vectors are not independent:
revealing $\phi$ generally reveals $\tmu$ too.
The fix to both these problems is to perform random splittings of the
functions~$\tmu$ and~$\one$ into constantly many functions of randomly chosen
disjoint support. This will give us a new product polytope $\Phi'$ which turns
out to contain $\Phi(\tmu,\one)^d$, but has far fewer vertices. Furthermore,
revealing one of its vertices only reveals a tiny part of $\tmu$. Hence, we then
proceed to bounding $\langle\tmu-\one,\phi\rangle$ over the new product polytope
(see Lemma~\ref{lem:anticor}) using Bernstein's inequality and the union
bound. In doing this, we still need to be careful to bound the contribution from
the tiny part of $\tmu$ which is dependent on $\phi$, and from any exceptionally
large values in $\phi$ (to which we cannot apply Bernstein's inequality as it
would spoil the required concentration).

In this sketch, when we used the Stone--Weierstrass theorem to arrive at a linear
problem and when we used Bernstein's inequality, we indicated we need to deal
with the contribution of exceptionally large values in $\phi$ to an inner
product such as $\langle\tmu,\phi\rangle$. Let us briefly explain now how we
approach this.  We show that unless the contribution of the large values is
tiny, for large enough even $d'$ the contribution of the large values to
$\langle\tmu,\phi^{d'}\rangle$ will indeed be enormous, and in particular (because we
can ensure both vectors in the inner product are non-negative) the latter inner
product will be enormous. We think of $\langle\tmu,\phi^{d'}\rangle$ as a high
moment of $\langle\tmu,\phi\rangle$. Consequently, we refer to a good upper
bound on $\langle\tmu,\phi^{d'}\rangle$ as a \emph{moment bound}.  Conversely,
having a moment bound implies that the contribution of large entries to
$\langle\tmu,\phi\rangle$ is tiny.  We use the Kim--Vu polynomial concentration
inequality to show that the moment bounds we need hold (see
Lemma~\ref{cor:momentboundcor}).

\medskip

The steps described above give a version of Theorem~\ref{thm:maintechint}
without deletion, and with weak bounds on the failure probability coming from
the Kim--Vu inequality. To obtain optimal bounds and introduce deletion, we use
a version of the R\"odl-Ruci\'nski deletion method~\cite{RR} due to Sp\"ohel,
Steger and Warnke~\cite{SSW} to argue that at the cost of deleting a few
elements of the random set to go from $\mu$ to $\tmu$, we have exponential
bounds on the probability of the required moment bounds holding (see
Lemma~\ref{lem:deletionbounds}). The rest of the proof already works at the
optimal failure probability and is sufficiently robust to survive this deletion.

\subsection{Comparison to the work of Conlon and Gowers}

As explained before, our starting point was the approach of Conlon and
Gowers~\cite{ConGow}. The obvious difference is that their proof uses a more
complicated norm than ours, involving `capped convolutions' whose values are
bounded. There are several points of similarity. Both we and they use randomised
rounding to obtain integer models from fractional ones, though at this step we
do more work and obtain significantly stronger properties. Our use of the
separating hyperplane theorem (which they call Hahn--Banach theorem) to reduce
the problem to an anti-correlation statement for $\tmu-\one$ is broadly
similar to theirs. The step of linearising the problem using Stone--Weierstrass
is again quite similar in both proofs. In fact, for Conlon and Gowers it is
easier since their norm avoids the issue of large values.

However, at this stage the proofs diverge: we have reduced to a genuinely linear
problem, but Conlon and Gowers, due to the `capped convolution', have a problem
which is only approximately linear. While their end goal, as ours, is to get to
a union bound over vertices, the route they take to get there is a good deal
harder work than ours. Our next step of randomised splitting resembles a point
in the construction of the norm for Conlon and Gowers, but they are only able to
use it to obtain enough independence between the polytope vertices and the
random vector, so that two further ideas (which they call `sufficient
randomness' and `normalised restrictions') are needed to reduce the number of
vertices over which a union bound is to be taken. Where we use moment bounds to
deal with exceptionally large entries in the polytope vertices, which we can do in
full generality, Conlon and Gowers need to understand the effect of capping
convolutions. This turns out to be difficult, and is the part of their proof
where they can no longer work in a general setting, but need to restrict to the
special cases (`arithmetic problems with two degrees of freedom' and strictly
balanced hypergraphs) that they proved.

\section{Setup, notation, main lemmas, and some reductions}\label{sec:tools}

Our goal in this section is to set up the notation necessary for our proof of Theorem~\ref{thm:maintechint}, provide a series of reductions of this theorem to simpler goals, and formulate the main lemmas that we need to obtain these reductions as well as for establishing the result of these reductions (Theorem~\ref{thm:correlate}). We will also provide a more detailed outline of the main ideas of our proofs alongside these technical details.

\subsection{Preliminaries}

In this and the following sections, we will always use the following setup
(which is identical to the setup used in Theorem~\ref{thm:maintechint}).

\begin{setting}
\label{mainSetting}
We are given integers \(k,n\geq 2\) and $p\in(0,1]$. Let \(S\) be a \(k\)-uniform ordered hypergraph on \([n]\), and let \(\Sigma\) and \(\Omega\) be sets of structure functions on \([n]\) and subcounts of \(S\), respectively. Let \(\Sigma\) and \(\Omega\) contain the  function $\one$ that takes value \(1\) at all elements of $[n]$ and $S$ respectively, and let \(\Sigma\) contain each of the \(n\) functions \(x\to \ind(x=i)\).

Any convolution we take is with respect to $S$. We denote a sample from $[n]_p$ by~$X$ and an $\eps$-deletion of $X$ by~$\tilde{X}$. We write $\mu$ and $\tmu$ for the functions
\[\mu(y)=p^{-1}\ind(y\in X)\,, \quad\text{and}\quad \tmu(y)=p^{-1}\ind(y\in\tilde{X})\,,\]
and for any collection of functions $f_1,\dots,f_k\colon[n]\to\reals$ we write
$\ast_{i}(f_1,\dots,f_k)$ for the convolution $\ast_{i,\one}(f_1,\dots,f_k)$, where $\one\in\Omega$.
For any given set $P$ of functions from $[n]$ to $\RR_{\ge0}$, the polytope $\Phi(P)$ is as in Definition~\ref{def:polytope}, with respect to $S,\Sigma,\Omega$.
\end{setting}

The assumption that $\Sigma$ contains the functions $x\to\ind(x=i)$ (the standard basis functions) is only needed to ensure that $\|\cdot\|_{\Phi(P)}$ is a norm (as explained in Section~\ref{sec:functional}), and we will not use it otherwise. The assumption that $\Sigma$ and $\Omega$ contain their respective $\one$ functions will be used more seriously in some optimisation arguments in our proof.

In addition to the assumptions in this setup, we shall in our proof work with the following boundedness conditions. These are kept separate from the setup, since we want to be able to specify the parameters of these conditions explicitly, as they change over time.

\begin{definition}[boundedness conditions]
  Suppose we are in Setting~\ref{mainSetting}. We say the \emph{$(c,C,p,n)$-boundedness conditions} hold if
  \begin{enumerate}[label=\itmarab{BD}]
  \item $p\ge (\log^C n)n^{-1}$,
  \item $|\Sigma|,|\Omega|\le\exp(C^{-1}pn)$, and
  \item $\Delta_\kappa(S)\le cC^{1-\kappa}p^{\kappa-1}e(S)n^{-1}$ for each $\kappa\in[k]$.
  \end{enumerate}
\end{definition}

Before we can outline our proof of Theorem~\ref{thm:maintechint} in more detail,
we need some preliminary remarks on notation, some further definitions and some
easy facts concerning convolutions and polytopes.

We will often apply real operations and operators to functions, by which we
always intend to apply them pointwise. For example, the product of two functions
$f,g:[n]\to\RR$, written \(fg\), is the function $x\to f(x)g(x)$.  Also, for a
positive integer \(d\), we use \(f^d\) to denote the pointwise product $x\to
f(x)^d$. When we want to take an inner product, we will always write $\langle
f,g\rangle$ explicitly and mean the inner product as defined
in~\eqref{definitioninnerproduct}.  We also consider the space of functions
$\RR^{[n]}$ as a vector space. The $x\in[n]$ coordinate of the vector $f$ in
this vector space is $f(x)$. Subscripts (which we will use heavily), on the
other hand, will always be indices identifying different vectors in a
collection.

In the proof of Theorem~\ref{thm:maintechint} we shall be interested in bounding
$\|f\|_{\Phi(P)}=\max_{\phi\in\Phi(P)}\langle f,\phi\rangle$ for certain linear functions~$f$.
It is a standard fact of linear optimisation that such a maximum is attained at a vertex of $\Phi(P)$.
We say that $\phi$ is a \emph{non-negative vertex} of $\Phi(P)$ if $\phi\ge 0$ (pointwise), and a \emph{non-positive vertex} of $\Phi(P)$ if $\phi\le 0$.

Observe further that, if \(f_1,\dots{},f_k\) are non-negative functions, bounded by \(h_1,\dots{},h_k\in P\), respectively, then for all \(x\in [n]\) we have
    \[
    0\le \ast_{i,\omega}(f_1,\dots{},f_k)(x)\le \ast_{i,\mathbf{1}}(h_1,\dots{},h_k)(x)\, .
    \]
This justifies the following definition.

\begin{definition}[largest and extreme test functions]\label{def:largest}
  Let $\phi\in\Phi(P)$. We say that $\phi$ is a \emph{\(P\)-largest test function} if
  \[\phi=\ast_{i,\mathbf{1}}(h_1,\dots{},h_k) \qquad \text{for some } h_1,\dots,h_k\in P\,.\]
  We say that $\phi$ is a \emph{$P$-extreme test function} if
  \[\phi\in\Sigma \quad\text{or}\quad \phi=\ast_{i,\omega}(f_1,\dots{},f_k)\qquad \text{for some $\omega\in\Omega$ and some \(P\)-extreme \(f_1,\dots{},f_k\)}\,.\]
\end{definition}

With these definitions in place, we can now collect a couple of straightforward but useful facts about the polytope \(\Phi(P)\) and the corresponding $P$-test functions for convenient reference later.

\begin{fact}\label{Rmk:anti_uniform_functionals}
  Let \(P\) be a set of functions from $[n]$ to $\RR_{\ge0}$.
  \begin{enumerate}[label=\rom]
    \item\label{itm:Phi:1} \(\Phi(P)\) is an $n$-dimensional centrally symmetric polytope in \(\mathbb{R}^{[n]}\).
    \item\label{itm:Phi:2} All \(P\)-largest test functions are
      \(P\)-extreme test functions, but not vice-versa.
    \item\label{itm:Phi:vertices} Every vertex of \(\Phi(P)\) is a \(P\)-extreme test function.
    \item\label{itm:Phi:3} Every vertex of $\Phi(P)$ is either non-negative or non-positive.
    \item\label{itm:Phi:4} For any non-negative vertex~$\phi$ of~$\Phi(P)$, either $0\le \phi\le\one$, or there
      exists a \(P\)-largest test function
      \(\ast_{i,\mathbf{1}}(h_1,\dots{},h_k)\) such that
      \(0\le \phi\le\ast_{i,\mathbf{1}}(h_1,\dots{},h_k)\).
\end{enumerate}
\end{fact}

Here, \ref{itm:Phi:1} follows directly from the definition of $\Phi(P)$.
Definitions~\ref{def:extreme} and~\ref{def:largest} immediately give~\ref{itm:Phi:2}.
Moreover, \ref{itm:Phi:vertices} follows from the fact that
 the convolution operator is multilinear and from Definition~\ref{def:polytope}.
The definition of $\Phi(P)$ also implies that $\Sigma$ and all the convolutions are non-negative functions and contain the non-negative vertices of $\Phi(P)$, while their negatives contain the remaining vertices of $\Phi(P)$. Hence we obtain~\ref{itm:Phi:3} and~\ref{itm:Phi:4}, since the structure functions $\sigma\in\Sigma$ satisfy $0\le\sigma\le\one$.

\medskip

In our optimisation arguments, some functions we will be interested in are not linear. Since, as explained in more detail later, we will approximate these by polynomials and then consider each monomial in such a polynomial separately in the optimisation task, we shall also be interested in what we call a product polytope. (We remark though that this is not the usual definition of a product polytope; in particular the convex hull in the following definition is not redundant.)

\begin{definition}[product polytope]
Given a polytope $\Phi\subset\RR^{[n]}$ and some $\ell\ge2$, the \emph{product polytope} $\Phi^\ell\subset\RR^{[n]}$ is the convex hull of
\[\big\{\phi_1\dots\phi_{\ell}\,:\,\phi_i\in\Phi\text{ for each }i\in[\ell]\big\}\,,\]
where $\phi_1\dots\phi_{\ell}$ denotes a pointwise product.
\end{definition}

The following is a simple consequence of this definition, the definition of the test polytope $\Phi(P)$, and the fact that in our setting
\(\Sigma\) contains the function \(\one\).

\begin{fact}\label{fac:d-product}
  If we are in Setting~\ref{mainSetting}, then for any integer $\ell\ge 1$ and for any set of functions~$P$, we have
  \[\Phi(P) \subset \Phi(P)^\ell\,.\]
\end{fact}

The next lemma identifies the vertices of a product polytope.

\begin{lemma}\label{lem:vtxprod} Let $\Phi$ be a polytope and $V(\Phi)$ its vertices. For any $\ell\ge 2$, the vertices of $\Phi^\ell$ are all of the form $\phi_1\dots\phi_\ell$ where each $\phi_i$ is in $V(\Phi)$.
\end{lemma}
\begin{proof}
  Since the set
  $T:=\{\phi_1\dots\phi_\ell\colon \phi_i\in V(\Phi)\}$
  of vectors is a subset of $\Phi^\ell$ by definition, it is enough to show that any vector in $\Phi^\ell$ can be written as a convex combination of vectors in $T$. This is true if all vectors $\phi_1\dots\phi_\ell$ with $\phi_i\in\Phi$ for all $i\in[\ell]$ can be written in this form.
    Let $\phi_i=\sum_{v\in V(\Phi)}\alpha_{iv}v$ be written as a convex combination for each $i\in[\ell]$. Then expanding out the product, we see that the coefficient of the vector $v_1\dots v_\ell\in T$ in $\phi_1\dots\phi_\ell$ is $\alpha_{1v_1}\dots\alpha_{\ell v_\ell}$. Because $\sum_{v\in V(\Phi)}\alpha_{iv}=1$ and the $\alpha_{iv}$ are non-negative, these coefficients $\alpha_{1v_1}\dots\alpha_{\ell v_\ell}$ are non-negative and sum to $1$ as required.
\end{proof}

We can use this to deduce the following lemma concerning optimisation problems over the product polytope $\Phi(P)^\ell$, where $\Phi(P)$ is our polytope of test functions.

\begin{lemma}\label{lem:max-product}
  Let \(P\) be a set of functions from $[n]$ to $\RR_{\ge0}$ and~$f$ be a non-negative vector. 
  Then
  \[\max_{\phi\in\Phi(P)^\ell}\langle f,\phi\rangle
  = \max\Big\{\langle f,\psi_1\dots\psi_{\ell'}\rangle\colon \ell'\in[\ell], \psi_i\text{ is a $P$-largest test function for all $i\in[\ell']$} \Big\}\,. \]
\end{lemma}

\begin{proof}
  By Lemma~\ref{lem:vtxprod}, $\langle f,\phi\rangle$ is maximised over $\phi\in\Phi(P)^\ell$ at a vector of the form $\phi_1\dots\phi_\ell$ where each $\phi_i$ is a vertex of $\Phi(P)$. Since $f$ is non-negative, we know from the definition of our inner product and Fact~\ref{Rmk:anti_uniform_functionals}\ref{itm:Phi:3} that each $\phi_i$ is a non-negative vertex of $\Phi(P)$.  Moreover, Fact~\ref{Rmk:anti_uniform_functionals}\ref{itm:Phi:4} tells us that either $0\le\phi_i\le\one$ or $0\le\phi_i\le \psi_i$ pointwise, where $\psi_i$ is a $P$-largest test function. It follows that
  \[\max_{\phi\in\Phi(P)^\ell}\langle f,\phi\rangle
  \le \max\Big\{\langle f,\psi_1\dots\psi_{\ell'}\rangle\colon \ell'\in[\ell], \psi_i\text{ is a $P$-largest test function for all $i\in[\ell']$} \Big\}\,. \]
  But since all the $\psi_i$ in the set on the right-hand side are contained in $\Phi(P)$, and hence $\psi_1\dots\psi_{\ell'}$ is feasible for $\Phi(P)^\ell$, we also get
  \[\max_{\phi\in\Phi(P)^\ell}\langle f,\phi\rangle
  \ge \max\Big\{\langle f,\psi_1\dots\psi_{\ell'}\rangle\colon \ell'\in[\ell], \psi_i\text{ is a $P$-largest test function for all $i\in[\ell']$} \Big\}\,. \qedhere\]
\end{proof}

We will see that the conclusion of this lemma is critical for our proof, where we will only have a \emph{polynomial} number of $P$-largest test functions. This will allow us to apply standard concentration inequalities.

\medskip

We now turn to the proof of Theorem~\ref{thm:maintechint}. We start with two reductions.

\subsection{Reduction from \texorpdfstring{$\{0,1\}$}{{0,1}}-valued models to \texorpdfstring{$[0,1]$}{[0,1]}-valued models}\label{sec:reduction-one}

We first reduce Theorem~\ref{thm:maintechint} to the following theorem, which is
identical except that the $\eps$-good dense models it produces take values in
$[0,1]$ not $\{0,1\}$.

\begin{theorem}[main technical theorem: {$[0,1]$-valued variant}]\label{thm:maintech}
Given $k,r\ge2$, $c\ge1$ and $\eps>0$, there exists $C>0$ such that the following holds for all sufficiently large $n$ and any $0<p\le 1$. Let us be in Setting~\ref{mainSetting}, and suppose the $(c,C,p,n)$-boundedness conditions hold.
Then with probability at least \(1-\exp\big(-\tfrac{pn}{C}\big)\) the random set \(X=[n]_p\) has an $\eps$-deletion $\tilde{X}$
with the following properties.
\begin{itemize}
  \item $\|\tmu-\one\|_{\Phi(\tmu,\one)}\le\eps k^{-1}$.
  \item For any $f_1,\dots,f_r\colon[n]\to\RR_{\ge0}$ with
    $f_1+\dots+f_r=\tmu$, there exist $g_1,\dots,g_r\colon[n]\to[0,1]$ with
    $g_1+\dots+g_r=\one$, such that
    \(\|f_i-g_i\|_{\Phi(\tilde{\mu},\mathbf{1})}\le\eps k^{-1}\) for each
    $i\in[r]$.
\end{itemize}
\end{theorem}

We now explain (and then prove) how this theorem implies Theorem~\ref{thm:maintechint}.
We shall proceed in two steps, corresponding to the following two statements. The first converts the $[0,1]$-valued models $g_1,\dots,g_r$ from Theorem~\ref{thm:maintech} into $\{0,1\}$-valued dense models, but using the $\|\cdot\|_{\Phi(\one)}$ norm instead of the $\|\cdot\|_{\Phi(\tmu,\one)}$ norm that
Theorem~\ref{thm:maintech} requires.

\begin{theorem}[{[0,1]-valued models imply \{0,1\}-valued models}]\label{thm:intmodel}
  Given $k,r\ge2$, $c\ge1$ and $\eps>0$ there exists $C>0$ such that the
  following holds for all sufficiently large $n$ and any $0<p\le 1$. Let us be
  in Setting~\ref{mainSetting}, and suppose that the $(c,C,p,n)$-boundedness
  conditions hold.
  
  For any \(g_1,\dots,g_r:[n]\to [0,1]\) with $g_1+\dots+g_r=\one$, there exist \(g^\ast_1,\dots,g^\ast_r:[n]\to \{0,1\}\)
  with $g^\ast_1+\dots+g^\ast_r=\one$ such that {for each $i\in[r]$} we have
  \[
  \|g_i-g^\ast_i\|_{\Phi(\mathbf{1})}\le\eps\,.
  \]
\end{theorem}

The proof of this theorem is provided in Section~\ref{sec:intmodels}. The construction of $g^\ast_1,\dots,g^\ast_r$ in this proof is given by the following simple randomised rounding: for each $x\in[n]$, we choose $i\in[r]$ with probability $g_i(x)$, and set $g^\ast_i(x)=1$, and $g_{i'}(x)=0$ for $i'\neq i$. The analysis of this randomised rounding is rather more complicated.
But we remark that it uses several of the key ideas that we also use for the proof of Lemma~\ref{lem:anticor} which we will explain later in this section.

The second step then is to convert the closeness in $\|\cdot\|_{\Phi(\one)}$ to
closeness in $\|\cdot\|_{\Phi(\tmu,\one)}$. The next lemma implies that this conversion can be performed.

\begin{lemma}[{comparing $\|\cdot\|_{\Phi(\one)}$ and $\|\cdot\|_{\Phi(\tmu,\one)}$}]\label{lem:upgrade}
    Given $k\ge2$, $p\in(0,1]$ and $n$, let us be in Setting~\ref{mainSetting}. Given $\tilde{X}\subset[n]$, suppose that for every $\tmu$-bounded $f'$ there exists $\one$-bounded $g'$ such that $\|f'-g'\|_{\Phi(\tmu,\one)}\le\tfrac{1}{2k}\eps$. Then for every pair of $\one$-bounded functions $g,g^\ast$ with $\|g-g^\ast\|_{\Phi(\one)}\le\tfrac{1}{2k}\eps$, we have $\|g-g^\ast\|_{\Phi(\tmu,\one)}\le\eps$.
\end{lemma}

The proof of this lemma is a fairly easy application of telescoping sums and the properties of the convolution. Hence, we provide it straight away.

\begin{proof}[Proof of Lemma~\ref{lem:upgrade}]
  Let $\phi\in\Phi(\tmu,\one)$ maximise $|\langle g-g^\ast,\phi\rangle|$.
  This maximum is attained either in the maximum or the minimum of $\langle g-g^\ast,\phi\rangle$; in either case, since $\langle g-g^\ast,\phi\rangle$ is linear, $\phi$ is a vertex of $\Phi(\tmu,\one)$. By Fact~\ref{Rmk:anti_uniform_functionals}\ref{itm:Phi:3}, $\phi$ is either non-negative or non-positive. Let us assume $\phi$ is a non-negative vertex; the case when $\phi$ is a non-positive vertex can be proved analogously.

  By Fact~\ref{Rmk:anti_uniform_functionals}\ref{itm:Phi:vertices}, either
  $\phi\in\Sigma\subset\Phi(\one)$ and there is nothing to prove, or $\phi$ is of the
  form \[\ast_{j,\omega}(f_1,\dots,f_{j-1},f_{j+1},\dots,f_{k})\,,\] with each $f_i$ being
  either $\tmu$- or $\one$-bounded. Suppose for simplicity that
  $f_1,\dots,f_\ell$ are exactly the $\tmu$-bounded functions (all other cases
  follow by symmetry). Since $\phi\not\in\Phi(\one)$, we
  have $\ell\ge 1$. For each $i\in[\ell]$, let~$g_i$ be the $\one$-bounded function
  with $\|f_i-g_i\|_{\Phi(\tmu,\one)}\le\tfrac{1}{2k}\eps$ that exists by assumption.
  For $i>\ell$ with $i\neq j$, let $g_i=f_i$. 
    We can write
    \begin{equation*}
      \langle g,\phi\rangle=\big\langle g,\ast_{j,\omega}(f_1,\dots,f_{k})\big\rangle
      =\big\langle f_1,\ast_{1,\omega}(f_2,\dots,f_{j-1},g,f_{j+1},\dots,f_{k})\big\rangle\,.
    \end{equation*}
    By the definition of our norm,
    $\|f_i-g_i\|_{\Phi(\tmu,\one)}\le\tfrac{1}{2k}\eps$ implies for all
    $\phi'\in\Phi(\tmu,\one)$ that $\langle f_i-g_i,\phi'\rangle=\langle
    f_i,\phi'\rangle-\langle g_i,\phi'\rangle\le\tfrac{1}{2k}\eps$. Since
    $\Phi(\tmu,\one)$ is centrally symmetric, we then also have $\langle
    f_i,\phi'\rangle-\langle g_i,\phi'\rangle\ge -\tfrac{1}{2k}\eps$ for all
    $\phi'\in\Phi(\tmu,\one)$. Since
    $\ast_{1,\omega}(f_2,\dots,f_{j-1},g,f_{j+1},\dots,f_{k})\in\Phi(\tmu,\one)$,
    we conclude that
    \begin{equation*}
      \langle g,\phi\rangle=
      \big\langle g_1,\ast_{1,\omega}(f_2,\dots,f_{j-1},g,f_{j+1},\dots,f_{k})\big\rangle\pm \frac{1}{2k}\eps
      =\big\langle g,\ast_{j,\omega}(g_1,f_2,\dots,f_{k})\big\rangle\pm \frac{1}{2k}\eps\,.
    \end{equation*}
    Repeating this argument a further $\ell-1$ times for $i=2,\dots,\ell$, we obtain
    \begin{equation}\label{eq:g-product}
      \langle g,\phi\rangle=\big\langle g,\ast_{j,\omega}(g_1,\dots,g_{k})\big\rangle\pm\frac{\ell}{2k}\eps\,,
    \end{equation}
    and by the same logic also
    \begin{equation}\label{eq:gs-product}
     \langle g^\ast,\phi\rangle=\big\langle g^\ast,\ast_{j,\omega}(g_1,\dots,g_{k})\big\rangle\pm\frac{\ell}{2k}\eps\,.
    \end{equation}
    Now $\ast_{j,\omega}(g_1,\dots,g_{k})\in\Phi(\one)$, so $|\langle
    g-g^\ast,\ast_{j,\omega}(g_1,\dots,g_{k})\rangle|\le\tfrac1{2k}\eps$ by
    assumption. Combining this with~\eqref{eq:g-product} and~\eqref{eq:gs-product}, we get
    \[\langle g^\ast,\phi\rangle=\langle g,\phi\rangle\pm\frac{2\ell}{2k}\eps\pm\frac{1}{2k}\eps\,.\]
    Since $\ell\le k-1$, this proves $|\langle g-g^\ast,\phi\rangle|\le\eps$,
    and hence $\|g-g^\ast\|_{\Phi(\tmu,\one)}\le\eps$ as required.
\end{proof}

The proof of Theorem~\ref{thm:maintechint} from Theorem~\ref{thm:maintech},
Theorem~\ref{thm:intmodel}, and Lemma~\ref{lem:upgrade} is now a quick calculation.

\begin{proof}[Proof of  Theorem~\ref{thm:maintechint}]
Given $k,r\ge2$ and $c\ge1$, $\eps>0$, let $C>0$ be sufficiently large for both Theorem~\ref{thm:maintech} and for Theorem~\ref{thm:intmodel}, both with input $k,r,c,\tfrac{1}{4k^2}\eps$.
Given $p\in(0,1]$, let us be in Setting~\ref{mainSetting}, and suppose that the $(c,C,p,n)$-boundedness conditions are satisfied.

Suppose now that the likely event of Theorem~\ref{thm:maintech} occurs for \(X=[n]_p\). Let \(\tilde{X}\) be the $\eps$-deletion that this guarantees. For any $\tmu$-bounded $f'$, by the likely event of Theorem~\ref{thm:maintech} (with $f_1=f'$, $f_2=\tmu-f'$ and $f_3,\dots,f_r=0$) there exists $\one$-bounded $g'$ such that $\|f'-g'\|_{\Phi(\tmu,\one)}\le\tfrac{1}{4k^2}\eps$. Hence, the assumption of Lemma~\ref{lem:upgrade} with input $\tfrac{1}{2k}\eps$ is satisfied.

The first conclusion of Theorem~\ref{thm:maintechint} that $\|\tmu-\one\|_{\Phi(\tmu,\one)}\le\eps k^{-1}$ is immediate from Theorem~\ref{thm:maintech}.
For the second conclusion, given $0\le f_1,\dots,f_r\le\tmu$ with $f_1+\dots+f_r=\tmu$, observe that by Theorem~\ref{thm:maintech} there are non-negative \(g_1,\dots,g_r\) such that $g_1+\dots+g_r=\one$ and for each $i\in[r]$ we have
\begin{equation*}
 \|f_i-g_i\|_{\Phi(\tilde{\mu},\mathbf{1})}\le\tfrac{1}{4k^2}\eps\le\tfrac{1}{2k}\eps\,.
\end{equation*}
By  Theorem~\ref{thm:intmodel}, there are \(g^\ast_1,\dots,g^\ast_r:[n]\to\{0,1\}\) such that $g^\ast_1+\dots+g^\ast_r=\one$ and for each $i\in[r]$ we have $\|g_i-g_i^\ast\|_{\Phi(\mathbf{1})}\le\tfrac{1}{4k^2}\eps$. Applying Lemma~\ref{lem:upgrade}, with input $\tfrac{1}{2k}\eps$, we obtain $\|g_i-g^\ast_i\|_{\Phi(\tmu,\one)}\le\tfrac{1}{2k}\eps$. Finally, by the triangle inequality we get $\|f_i-g^\ast_i\|_{\Phi(\tmu,\one)}\le\tfrac{1}{k}\eps$ as required.
\end{proof}

\subsection{Reduction from \texorpdfstring{$[0,1]$}{[0,1]}-valued models to a single anti-correlation statement}\label{sec:reduction-two}

We next reduce Theorem~\ref{thm:maintech} to yet a simpler statement, which we shall call the anti-correlation theorem.
To motivate this theorem, observe that Theorem~\ref{thm:maintech} asks us to find, for any given $f_1,\dots,f_r$, functions $g_1,\dots,g_r$ which satisfy an \emph{anti-correlation} statement over $\Phi(\tmu,\one)$: the quantity $\langle f_i-g_i,\phi\rangle$ should be very small for all $\phi\in\Phi(\tmu,\one)$.
We now reduce this to the following related anti-correlation theorem, which is more complicated than the above in that we ask for anti-correlation with a more complicated collection of functions, but simpler in that we only ask that the single function $\tmu-\one$ is anti-correlated with the collection.

\begin{theorem}[anti-correlation theorem]\label{thm:correlate}
Given $k,r\ge2$, $c\ge1$ and $\eps>0$, there exists $C>0$ such that the following holds for all sufficiently large $n$ and any $0<p\le 1$.
Let us be in Setting~\ref{mainSetting}, and suppose the $(c,C,p,n)$-boundedness conditions hold. Then with probability at least \(1-\exp\big(-\tfrac{pn}{C}\big)\) the random set \(X=[n]_p\) admits an \(\eps\)-deletion \(\tilde{X}\) with
the following properties.
\begin{itemize}
  \item $\|\tmu-\one\|_{\Phi(\tmu,\one)}\le\eps$.
  \item $\|\tilde{\mu}\|_{\Phi(\tilde{\mu},\mathbf{1})},\|\one\|_{\Phi(\tilde{\mu},\mathbf{1})}\le 2c$.
  \item
    $\big\langle\tilde{\mu}-\mathbf{1},\max(0,\phi_1,\dots,\phi_{r-1})\big\rangle<\eps$
    for every \(\phi_1,\dots,\phi_{r-1}\in\Phi(\tilde{\mu},\mathbf{1})\).
\end{itemize}
\end{theorem}

Observe that this statement gives us anti-correlation for $\tmu-\one$ with a collection of functions, and it also promises that $\tmu$ and $\one$ individually do not have too large an inner product with members of $\Phi(\tmu,\one)$. 

The following deterministic lemma will allow us to justify the reduction from Theorem~\ref{thm:maintech} to Theorem~\ref{thm:correlate}. For this, we will apply the lemma with $\xi=\tmu$ the scaled indicator of $\tilde{X}$ and $\Phi=\Phi(\tmu,\one)$.

\begin{lemma}[anti-correlation reduction lemma]\label{lem:HBreduce}
  Given $r\ge2$, $c\ge1$ and $\eps>0$, there exists $\delta>0$ such that the following holds. Let $n\in\nats$, let $\xi:[n]\to\RR_{\ge0}$, and let $\Phi\subset\RR^{[n]}$ be a centrally symmetric, $n$-dimensional polytope.
  If
  \begin{enumerate}[label=\itmarab{AC}]
  \item\label{itm:HB:small} $\|\xi-\one\|_{\Phi}\le\delta$,
  \item\label{itm:HB:moment} $\|\xi\|_{\Phi},\|\one\|_{\Phi}\le 2c$, and
  \item\label{itm:HB:max} $\big\langle\xi-\one,\max(0,\phi_1,\dots,\phi_{r-1})\big\rangle\le \delta$ for every $\phi_1,\dots,\phi_{r-1}\in\Phi$,
  \end{enumerate}
  then for every non-negative $f_1,\dots,f_r$ with $f_1+\dots+f_r=\xi$, there
  exist $g_1,\dots,g_r$ taking values in $[0,1]$ with $g_1+\dots+g_r=\one$, such
  that $\|f_i-g_i\|_{\Phi}\le\eps$ for each $i\in[r]$.
\end{lemma}

We prove Lemma~\ref{lem:HBreduce} in Section~\ref{sec:correlate}. The idea of its proof is as follows. We can write the set of functions $f_1,\dots,f_{r-1}$ for which we can obtain dense models $g_1,\dots,g_{r-1}$ as a convex set in $\RR^{[(r-1)n]}$, and then the separating hyperplane theorem tells us that if there is some $f_1,\dots,f_{r-1}$ which does \emph{not} have dense models, then there is a hyperplane separating $f_1,\dots,f_{r-1}$ from the convex set. We then use the anti-correlation assumption to obtain from this a contradiction. Actually, we need to introduce a (small) rescaling of $f_1,\dots,f_{r-1}$ to make this argument work. Our new idea for this is that we use the bounds from~\ref{itm:HB:moment} to argue that the effect of rescaling is tiny and can be absorbed in the error terms. Finally,  to obtain the dense model $g_r$ for $f_r$, we use~\ref{itm:HB:small} and the triangle inequality. In this argument, we treat $f_r$ and its dense model $g_r$ separately from the other $f_i$ and $g_i$ so that we can get $g_1+\dots+g_r=\one$.

\medskip

Theorem~\ref{thm:maintech} now follows directly from Theorem~\ref{thm:correlate} and Lemma~\ref{lem:HBreduce}.

\begin{proof}[Proof of Theorem~\ref{thm:maintech}]
Given $k,r\ge2$ and $c\ge1,\eps>0$, let $\delta>0$ be returned by Lemma~\ref{lem:HBreduce} for input $r,c,\eps k^{-1}$. Without loss of generality, we may assume $\delta\le\eps k^{-1}$. Let $C$ be returned by Theorem~\ref{thm:correlate} for input $k,r,c,\delta$.

Let $n$ be sufficiently large, and suppose $0<p\le 1$. Let us be in Setting~\ref{mainSetting}, and suppose that the $(c,C,p,n)$-boundedness conditions hold.
With probability $1-\exp(-pn/C)$ the likely event of Theorem~\ref{thm:correlate} occurs for the given inputs. Fix $X=[n]_p$ such that this is the case for a $\delta$-deletion $\tilde{X}$ of $X$, which is also an $\eps$-deletion since $\delta\le\eps$.
From the first conclusion of Theorem~\ref{thm:correlate} (with $\eps$ replaced by $\delta$) we obtain
\[\|\tmu-\one\|_{\Phi(\tmu,\one)}\le\delta\le\eps k^{-1}\,,\]
which is
the first conclusion of Theorem~\ref{thm:maintech}.
In addition, from the conclusions of Theorem~\ref{thm:correlate} we also infer that $\xi$ defined by $\xi(x)=\tmu(x)=p^{-1}\ind(x\in\tilde{X})$ and the polytope $\Phi(\tmu,\one)$ satisfy the conditions of Lemma~\ref{lem:HBreduce}. Hence the conclusion of Lemma~\ref{lem:HBreduce} holds with $\eps$ replaced by $\eps k^{-1}$.
This is precisely the second conclusion of Theorem~\ref{thm:maintech}.
\end{proof}

\subsection{Main lemmas and the proof of the anti-correlation theorem}\label{sec:main-lemmas}

We now continue by providing the lemmas which make up the proof of Theorem~\ref{thm:correlate}. Their aim is to make more tractable the anti-correlation statement we need. The first difficulty we would like to address is that the problem of maximising $\langle\tmu-\one,\max(0,\phi_1,\dots,\phi_{r-1})\rangle$ over $\phi_1,\dots,\phi_{r-1}$ is not a linear optimisation problem. However, it turns out that we can bound this maximum by performing a linear optimisation over a product polytope with the following lemma, which we shall apply with $\xi=\tmu$ the scaled indicator function of $\tilde{X}$ and $\Phi=\Phi(\tmu,\one)$.

\begin{lemma}[linear anti-correlation lemma]\label{lem:lincor}
Given $r\ge2$ and $c\ge1$, $\eps>0$ there exist $\eps_\lin>0$ and $d,d'\ge 2$ such that the following holds. Let $n\in\nats$, let $\xi\colon[n]\to\RR_{\ge0}$, and let $\Phi\subset\RR^{[n]}$ be a centrally symmetric polytope containing $\one$ such that
 \begin{enumerate}[label=\itmarab{LC}]
     \item\label{lincor:i} for all \(\phi \in \Phi^d\) we have \(\big|\langle\xi-\mathbf{1},\phi\rangle\big|\le\eps_\lin\), and
     \item\label{lincor:ii} for all $\ell\in[d']$ and all \(\phi\in\Phi^{\ell}\) we have \(\big|\langle \xi, \phi \rangle\big|,\big|\langle \mathbf{1}, \phi \rangle\big|\le 2c^{\ell}\).
 \end{enumerate}
 Then for all \(\phi_1,\dots,\phi_{r-1}\in\Phi\) we have \(\big|\langle\xi-\mathbf{1},\max(0,\phi_1,\dots,\phi_{r-1})\rangle\big|\le\eps\).
\end{lemma}

The proof of Lemma~\ref{lem:lincor}, which is provided in Section~\ref{sec:linear}, uses the Stone--Weierstrass approximation theorem to find a polynomial $\pi=\pi(x_1,\dots,x_{r-1})$ which uniformly approximates $\max(0,x_1,\dots,x_{r-1})$. The point is that to show $\big\langle\xi-\one,\max(0,\phi_1,\dots,\phi_{r-1})\big\rangle$ is guaranteed to be small, it is enough to show that $\langle\xi-\one,\pi(\phi_1,\dots,\phi_{r-1})\rangle$ is guaranteed to be small, and in turn it is enough to consider each of the monomials of $\pi$ separately. Choosing $d$ to be the degree of $\pi$, these monomials are in $\Phi^d$ and so the anti-correlation assumption~\ref{lincor:i} lets us infer the conclusion of the lemma.

There is, however, a complication: the Stone--Weierstrass approximation theorem requires us to specify a compact domain before we are given $\pi$, but the functions $\phi_i$ can have unboundedly large entries. We overcome this difficulty by arguing that exceptionally large values cause only tiny errors, using the moment bound assumption~\ref{lincor:ii}.
We will use this technique repeatedly, so let us briefly explain it now. Suppose that from some $\phi\in\Phi$, the contribution to (for example) $\langle\xi,\phi\rangle$ of entries of $\phi$ which have absolute value bigger than $2c$ is at least $\delta>0$. If $\ell$ is even, all entries of $\phi^\ell$ are non-negative, and hence these entries contribute even more to the $\ell$-th moment $\langle\xi,\phi^\ell\rangle$. More precisely, $\langle\xi,\phi^\ell\rangle\ge\delta(2c)^{\ell-1}$. But $\phi^\ell$ is in the product polytope $\Phi^\ell$, so if $\ell$ is large enough, this gives a contradiction to~\ref{lincor:ii}.

\medskip

In order to apply Lemma~\ref{lem:lincor} we need to establish the anti-correlation assumption~\ref{lincor:i} and the moment bound assumption~\ref{lincor:ii}. The purpose of the final two main lemmas presented in this section is to allow us to do this. We start with the anti-correlation demanded by~\ref{lincor:i}. This concerns a linear optimisation problem, though over the product polytope $\Phi(\tmu,\one)^d$. We know that the maximum value will occur at a vertex of $\Phi(\tmu,\one)^d$, and we can use
Lemma~\ref{lem:vtxprod} to describe the vertices of this polytope. Our approach now is to use the fact that $\tmu-\one$ is close to being a vector of mean-zero random variables, and hope that we can use this to argue that it is very unlikely that $\tmu-\one$ correlates with any given fixed vertex. To make this precise, we would like to use the following inequality of Bernstein (which can be found, for example, in~\cite[Section 2.8]{Vershynin2018}) together with a union bound over vertices.

\begin{lemma}[Bernstein's inequality]\label{lem:bernstein} Let \(Y_1,\dots{},Y_n\) be independent random variables taking values in \([-M,M]\), and let \(Y=Y_1+\dots{}+Y_n\). For \(\lambda\geq 0\) we have
\[\mathbb{P}\big[\big|Y-\mathbb{E} [Y]\big|\ge\lambda\big]\le 2\exp\Bigg(\frac{-\lambda^2/2}{\tfrac{M\lambda}{3}+\sum_i\mathrm{Var}(Y_i)}\Bigg)\,.\]
\end{lemma}

The idea of using Bernstein's inequality as outlined above has several problems. Firstly, $\tmu-\one$ is not a genuinely random vector, but $\mu-\one$ is. Secondly, the vertices of $\Phi(\tmu,\one)$ actually depend on $\tmu$ and hence $\mu$, so once we reveal the vertices there is no randomness left to apply Bernstein's inequality. Thirdly, the vertices of $\Phi(\tmu,\one)$ can have exceptionally large entries, which spoil the concentration guarantee of Bernstein's inequality (even if we could apply it). Finally, $\Phi(\tmu,\one)^d$ has far too many vertices to take the union bound with the probability bound Bernstein's inequality can deliver (even if the concentration were not spoiled, even if we could apply it).

The first and third of these problems are easily handled: $\mu$ is very close to $\tmu$, and we can deal with exceptionally large entries using moment bounds, and use Bernstein's inequality only for the small entries.

The other two problems require more work. For these, we use a `random splitting'. More precisely,
for a large integer $L$, we will pick two uniform random functions $\chi_\mu:[n]\to[L]$ and $\chi_\one:[n]\to\big[\lceil Lp^{-1}\rceil\big]$, and use them to split the functions $\mu$, $\tmu$, and $\one$ as follows.
Here, we split $\tmu$ and $\one$ separately because we are interested in the polytope $\Phi(\tmu,\one)$, in which vertices are $\{\tmu,\one\}$-extreme test functions by Fact~\ref{Rmk:anti_uniform_functionals}; this in turn means these vertices are in~$\Sigma$ or they are convolutions of functions which are either $\one$-extreme or $\tmu$-extreme, and we would like to consider these two cases independently. Further, we split $\tmu$ into~$L$ parts but $\one$ into $Lp^{-1}$ parts, so that in both cases each part has roughly $pn/L$ non-zero entries in expectation.

\begin{definition}[split, split polytope, revealed parts]\label{notationForSplitting}
Given $k,d\ge2$, $L,n\in\nats$, and $p\in(0,1]$, assume we are in Setting~\ref{mainSetting}. 
Assume we are given a function \(\chi_\mathbf{1}:[n]\to \big[\lceil Lp^{-1}\rceil\big]\), which we call the \emph{\(\mathbf{1}\)-split}
and a function \(\chi_\mu:[n]\to[L]\), which we call the \emph{\(\mu\)-split}.
We call each pre-image \(\chi_\mu^{-1}(i)\) with $i\in[L]$ a \emph{part} of the \(\mu\)-split, or also part~$i$ of the split, and
each pre-image \(\chi_\one^{-1}(j)\) with \(j\in[\lceil Lp^{-1}\rceil]\) a \emph{part} of the \(\one\)-split, or also part~$j$ of the split.

We let ${X}_i=\{x\in {X}\colon\chi_\mu(x)=i\}$ and $\tilde{X}_i=\{x\in \tilde{X}\colon\chi_\mu(x)=i\}$, and
now define the following functions. For \(i\in[L]\) and \(j\in[\lceil Lp^{-1}\rceil]\), let
\begin{align*}
\mu_i(x)&=\begin{cases}Lp^{-1}\qquad&\text{if \(x\in X_i\)}\,,\\0&\text{otherwise}\,,\end{cases} \\
\tilde{\mu}_i(x)&=\begin{cases}Lp^{-1}\qquad&\text{if \(x\in \tilde{X}_i\)}\,,\\0&\text{otherwise}\,,\end{cases} \\
\nu_j(x)&=\begin{cases}\lceil Lp^{-1}\rceil\hspace{.45cm}&\text{if \(\chi_\mathbf{1}(x)=j\)}\,,\\0&\text{otherwise}\,.\end{cases}
\end{align*}
We denote by \(\Phi'\) the polytope \(\Phi':=\Phi\big(\tilde{\mu}_1,\dots{},\tilde{\mu}_L,\nu_1,\dots{},\nu_{\lceil Lp^{-1}\rceil}\big)^d\), and call this the \emph{split polytope}.

For any vertex \(\phi\) of \(\Phi'\), let \(Q_{\phi}\subseteq[L]\) be the minimum set such that we can write \(\phi\) as a product of~\(d\) functions which are \(\{\tilde{\mu}_j:j\in Q_\phi\}\cup\{\nu_1,\dots{},\nu_{\lceil Lp^{-1}\rceil}\}\)-test functions. We let
\[Y(\phi) = \{x\in[n]\,:\,\chi_\mu(x)\in Q_{\phi}\}\,,\]
and call this the \emph{revealed part} of \(\phi\).
\end{definition}

The final part of this definition merits some explanation. Recall that Fact~\ref{Rmk:anti_uniform_functionals} and Lemma~\ref{lem:vtxprod} tell us that any vertex of $\Phi'$ can be written as a product of $d$ functions which are $\{\tmu_1,\dots,\tmu_L,\nu_1,\dots{},\nu_{\lceil Lp^{-1}\rceil}\}$-test functions. Since a convolution takes $k-1$ functions as arguments, this implies that at most $d(k-1)$ of these test functions can be used in this product. Hence, we obtain the following fact.

\begin{fact}\label{fact:split}
  Using the notation of Definition~\ref{notationForSplitting},
  $|Q_\phi|\le d(k-1)$
  for any vertex \(\phi\) of \(\Phi'\). In addition, $Q_\phi$ does not depend on $\tmu$, but only on~$\phi$.
\end{fact}

The revealed part~$Y(\phi)$ of a vertex \(\phi\) then contains all the elements of~$[n]$ that would be determined by knowing $\tmu$ on all parts of the $\mu$-split needed to determine \(\phi\). If $\tmu$ was replaced by $\mu$ in all these definitions, this would resolve our problem with lack of randomness as follows. Since each vertex $\phi$ of the polytope $\Phi(\mu_1,\dots{},\mu_L,\nu_1,\dots{},\nu_{\lceil Lp^{-1}\rceil})^d$ is independent of the value of $\mu$ on $[n]\setminus Y(\phi)$ and $Y(\phi)$ is determined by $Q_\phi$, which contains at most $d(k-1)$ entries of $[L]$, if we choose $L$ large enough, then $Y(\phi)$ will be a tiny subset of $[n]$. Unfortunately, the same logic becomes false when $\tmu$ is used (as in the definitions) instead of $\mu$ for the simple reason that we cannot know $\tmu$ on a subset of parts of the $\mu$-split without revealing \emph{all} of $\mu$ (since $\tmu$ is an $\eps$-deletion of $\mu$). We will, nevertheless, show, using suitable moment bounds, that knowing this for $\mu$ is enough to obtain the desired anti-correlation for $\tmu$ on~$\Phi'$. 

Hence, we will pass from $\Phi(\tmu,\one)^d$ to $\Phi'$ for our optimisation. As we will show, this new polytope $\Phi'$ has far fewer vertices than $\Phi(\tmu,\one)^d$, thus also solving the problem concerning union bounds mentioned above. The idea is as follows. A vertex of either polytope, by Lemma~\ref{lem:vtxprod} and Fact~\ref{Rmk:anti_uniform_functionals}, consists of a product of $d$ functions which are either in $\Sigma$ or convolutions \(\ast_{i,\omega}(f_1,\dots{},f_k)\) such that each \(f_i\) is \(P\)-extreme, where $P=\{\tmu,\one\}$ for $\Phi(\tmu,\one)^d$ and $P=\{\tmu_1,\dots{},\tmu_L,\nu_1,\dots{},\nu_{\lceil Lp^{-1}\rceil}\}$ for~$\Phi'$. Hence, we can upper bound the number of vertices by estimating how many functions are $P$-extreme. For $P=\{\tmu,\one\}$, unfortunately there are already $2^n$ functions which are $\one$-extreme, and this is (far) too many for a union bound. But we will show that there are fewer than $2^{2pn/L}$ functions which are $\{\tmu_1,\dots{},\tmu_L,\nu_1,\dots{},\nu_{\lceil Lp^{-1}\rceil}\}$-extreme, and this will be good enough.

Passing from $\Phi(\tmu,\one)^d$ to $\Phi'$ for our optimisation is viable since \(\Phi(\tilde{\mu},\mathbf{1})^d\subseteq \Phi'\) as the following easy lemma shows.

\begin{lemma}[split polytope reduction lemma]\label{lem:splitreduce}
Given $k\ge2$, $d\ge1$, $c\ge1$, integers $n$ and $L$, and $p\in(0,1]$, let us be in Setting~\ref{mainSetting}. Let \(\chi_\mathbf{1}:[n]\to \{1,\dots{},\lceil Lp^{-1}\rceil\}\) and \(\chi_\mu:[n]\to \{1,\dots{},L\}\) be an arbitrary \(\mathbf{1}\)- and \(\mu\)-split respectively, and let \(\tilde{X}\) be an arbitrary subset of \([n]\). Using the notation of Definition~\ref{notationForSplitting}, we have \(\Phi(\tilde{\mu},\mathbf{1})^d\subseteq \Phi'\).
\end{lemma}
\begin{proof}
  By definition of product polytope, showing $\Phi(\tmu,\one)\subset\Phi'':=\Phi(\tmu_1,\dots,\tmu_L,\nu_1,\dots,\nu_{\lceil Lp^{-1}\rceil})$ is enough, and in turn it suffices to show that any vertex \(\phi\) of \(\Phi(\tmu,\one)\) is contained in \(\Phi''\). By Fact~\ref{Rmk:anti_uniform_functionals}\ref{itm:Phi:vertices} $\phi$ is either in \(\Sigma\), or it is a convolution \(\ast_{i,\omega}(f_1,\dots{},f_k)\) for $\omega\in\Omega$ such that each \(f_i\) is \(\{\tmu,\one\}\)-extreme. If $\phi\in\Sigma$, then by definition $\phi\in\Phi''$ and we are done. So suppose
 $\phi=\ast_{i,\omega}(f_1,\dots{},f_k)$.
  We will now rewrite the functions $f_j$ appearing in this convolution.
  Let \(q\) be the number of these functions $f_j$ that are \(\tilde{\mu}\)-extreme; hence \(k-q\) functions $f_j$  are \(\mathbf{1}\)-extreme. Also set
 \[L_j:=\begin{cases}
 L \qquad & \text{if $f_j$ is \(\tmu\)-extreme} \\
 \lceil Lp^{-1}\rceil \qquad & \text{if $f_j$ is \(\one\)-extreme}
   \end{cases}
 \]
 and
 \[
 f_{j,\ell}:=\begin{cases}
 f_{j} \tilde{\mu}_{\ell}/\tmu \qquad & \text{if $f_j$ is \(\tmu\)-extreme} \\
 f_{j}\nu_{\ell}  \qquad & \text{if $f_j$ is \(\one\)-extreme}\,,
 \end{cases}
 \]
 where the fraction \(\tilde{\mu}_{\ell}/\tilde{\mu}\) is to be interpreted pointwise
 and if \(\tilde{\mu}(x)=0\) and hence \(\tilde{\mu}_{\ell}(x)=0\), then we
 define the result to be \(0\).
 Then we write
 \begin{equation}\label{eq:split:mu}
   f_{j}=\frac{1}{L_j}\sum_{\ell\in[L_j]}f_{j,\ell}\,.
 \end{equation}
 Observe that, if $f_{j}$ is $\tmu$-extreme, then each $f_{j,\ell}$ is
 $\tilde{\mu}_{\ell}$-bounded; similarly, if $f_{j}$ is $\one$-extreme, then
 each $f_{j,\ell}$ is $\nu_{\ell}$-bounded.
 
 Recall that
 \[\phi=\ast_{i,\omega}(f_1,\dots{},f_k)(x)=\frac{n}{e(S)}\sum_{s\in S_i(x)}\omega(s)\prod_{j\neq i}f_j(s_j)\]
 is linear in each of the~$f_j$. Therefore,
 substituting~\eqref{eq:split:mu} gives
 \begin{align*}
   \phi
   &=\frac{n}{e(S)} \sum_{s\in S_i(x)}\omega(s)\prod_{j\neq i}\frac{1}{L_j}\sum_{\ell\in[L_j]}f_{j,\ell}(s_j)
   = \frac{n}{e(S)}\cdot \frac{1}{L^{q}\lceil Lp^{-1}\rceil^{k-q}} \sum_{s\in S_i(x)}\omega(s) \prod_{j\neq i}\sum_{\ell\in[L_j]}f_{j,\ell}(s_j) \\
   &=\frac{1}{L^{q}\lceil Lp^{-1}\rceil^{k-q}} \sum_{\ell_1\in[L_1],\dots,\ell_k\in[L_k]}
     \bigg(\frac{n}{e(S)}\sum_{s\in S_i(x)}\omega(s) \prod_{j\neq i} f_{j,\ell_j}(s_j)\bigg) \\
   &=\frac{1}{L^{q}\lceil Lp^{-1}\rceil^{k-q}} \sum_{\ell_1\in[L_1],\dots,\ell_k\in[L_k]}
     \ast_{i,\omega}(f_{1,\ell_1},\dots{},f_{k,\ell_k})\,. 
 \end{align*}
 Since each $f_{j,\ell}$ is $\tilde{\mu}_{\ell}$-bounded or
 $\nu_{\ell}$-bounded, we know that each of the convolutions
 $\ast_{i,\omega}(f_{1,\ell_1},\dots{},f_{k,\ell_k})$ in this sum is in
 $\Phi''$. Since the sum
 has $L^{q}\lceil Lp^{-1}\rceil^{k-q}$ terms, this means that we have
 written~$\phi$ as a convex combination of elements of
 $\Phi''$. This proves
 $\phi\in\Phi''$, as desired.
\end{proof}

Since Lemma~\ref{lem:splitreduce} shows that $\Phi(\tmu,\one)^d\subset\Phi'$,
in our proof of Theorem~\ref{thm:correlate} it will be enough to prove the anti-correlation statement that $\langle\tmu-\one,\phi\rangle$ is small for all $\phi\in\Phi'$. This is captured in the following lemma, which also provides the necessary moment bound statements.

\begin{lemma}[split polytope anti-correlation lemma]\label{lem:anticor}
  Given $k\ge2$, $c\ge 1$, $\eps>0$, and $d,d'\in\nats_{>0}$, there exists \(L_0\) such that, if \(L\geq L_0\), there exists \(C\) such that the following holds for all sufficiently large $n$ and any $0<p\le 1$.
  Let us be in Setting~\ref{mainSetting}, and suppose the $(c,C,p,n)$-boundedness conditions hold. Then with probability at least \(1-\exp\big(-\tfrac{pn}{C}\big)\) the random set \(X=[n]_p\) admits an \(\eps\)-deletion \(\tilde{X}\) with the following properties.

  There exist functions \(\chi_\mu:[n]\to[L]\) and
  \(\chi_\mathbf{1}:[n]\to\big[\lceil Lp^{-1}\rceil\big]\) such that, using the notation
  of Definition~\ref{notationForSplitting},
  \begin{itemize}
  \item \(\|\tmu-\one\|_{\Phi'}\le\eps\), and
  \item \(\big|\langle\tilde{\mu},\phi\rangle\big|,\big|\langle\mathbf{1},\phi\rangle\big|\le 2c^\ell\)
    for each $\ell\in[d']$ and each \(\phi\in\Phi(\tilde{\mu},\mathbf{1})^\ell\).
  \end{itemize}
\end{lemma}

This lemma is proved in Section~\ref{sec:finest}. We explain in the next subsection how we obtain the moment bounds claimed in this lemma.
The idea for obtaining the claimed anti-correlation is the simple approach mentioned earlier: for a fixed vertex $\phi$ of $\Phi'$, we use Bernstein's inequality to argue that the `typical' entries of $\phi$ are unlikely to correlate with $\mu-\one$ on $[n]\setminus Y(\phi)$, and we take a union bound over choices of $\phi$. We need to handle the atypical (large) entries of $\phi$, the (tiny) difference between $\tmu$ and $\mu$, and the correlation on (the tiny set) $Y(\phi)$, all of which we do by using moment bounds.

With this, we are ready to prove Theorem~\ref{thm:correlate}, which now follows from the linear anti-correlation lemma (Lem\-ma~\ref{lem:lincor}),
the split polytope reduction lemma (Lemma~\ref{lem:splitreduce}),
and the split polytope anti-correlation lemma (Lem\-ma~\ref{lem:anticor}).

\begin{proof}[Proof of Theorem~\ref{thm:correlate}]
 Given $k,r\ge2$, $c\ge1$ and \(\eps>0\), let \(\eps_\lin>0\) and \(d,d'\) be returned by Lemma~\ref{lem:lincor} for input \(r,c,\eps\). Without loss of generality, we may assume \(\eps_\lin\le\eps\). We now input $k,c,\eps_\lin,d,d'$ to Lemma~\ref{lem:anticor}, which returns \(L_0\). We choose $L=L_0$ and let $C$ be returned by Lemma~\ref{lem:anticor}.

 Given a sufficiently large $n$ and $p\in(0,1]$, let us be in Setting~\ref{mainSetting} and suppose that the $(c,C,p,n)$-boundedness conditions hold. In particular, the conditions of Lemma~\ref{lem:anticor} are satisfied, so with probability at least \(1-\exp\big(-\tfrac{pn}{C}\big)\), the set \(X=[n]_p\) has an \(\eps_\lin\)-deletion  \(\tilde{X}\), which is also an $\eps$-deletion, such that there exist splits \(\chi_\mu,\chi_\mathbf{1}\) for which the following hold, with the notation of Definition~\ref{notationForSplitting}. For \(\Phi'=\Phi(\tilde{\mu}_1,\dots{},\tilde{\mu}_L,\nu_1,\dots{},\nu_{\lceil Lp^{-1}\rceil})^d\) we have
\begin{equation}\label{eq:correlate:eps}
  \|\tmu-\one\|_{\Phi'}\le\eps_\lin\le\eps\,,
\end{equation}
and in addition for all $\ell\in[d']$ and \(\phi\in\Phi(\tilde{\mu},\mathbf{1})^\ell\) we have
\begin{equation}\label{eq:correlate:twoc}
  |\langle\tilde{\mu},\phi\rangle|,|\langle\mathbf{1},\phi\rangle|\le 2c^\ell\,.
\end{equation}
Suppose \(X\) satisfies the likely event, and fix \(\tilde{X}\) and \(\chi_\mu,\chi_\mathbf{1}\) witnessing this.

We would now like to apply Lemma~\ref{lem:lincor} to $\xi(x)=\tmu(x)=p^{-1}\ind(x\in\tilde{X})$ and the polytope $\Phi=\Phi(\tmu,\one)$.
The assumption~\ref{lincor:ii} of this lemma is satisfied by~\eqref{eq:correlate:twoc}. For the assumption~\ref{lincor:i}, observe that we have
\begin{equation}\label{eq:correlate:Phi}
  \Phi(\tilde{\mu},\mathbf{1})^d\subseteq\Phi'=\Phi(\tilde{\mu}_1,\dots{},\tilde{\mu}_L,\nu_1,\dots{},\nu_{\lceil Lp^{-1}\rceil})^d
\end{equation}
  by
 Lemma~\ref{lem:splitreduce}.
 Therefore, we obtain \(\langle\tilde{\mu}-\mathbf{1},\phi\rangle<\eps_\lin\) for all \(\phi\in\Phi(\tilde{\mu},\mathbf{1})^d\) from~\eqref{eq:correlate:eps}, verifying~\ref{lincor:i}.

 Hence, Lemma~\ref{lem:lincor} tells us that
 \begin{equation*}
   \big|\langle\tmu-\mathbf{1},\max(0,\phi_1,\dots,\phi_{r-1})\rangle\big|\le\eps\,.
 \end{equation*}
 for all \(\phi_1,\dots,\phi_{r-1}\in\Phi(\tmu,\one)\), giving the third
 conclusion of Theorem~\ref{thm:correlate}.  The second conclusion of
 Theorem~\ref{thm:correlate} follows from~\eqref{eq:correlate:twoc} with
 $\ell=1$.  Finally, recall that
 \(\Phi(\tilde{\mu},\mathbf{1})\subseteq\Phi(\tilde{\mu},\mathbf{1})^d\) by
 Fact~\ref{fac:d-product}. Therefore, we have
 $\Phi(\tilde{\mu},\mathbf{1})\subset\Phi'$ by~\eqref{eq:correlate:Phi}, and
 thus the first conclusion of Theorem~\ref{thm:correlate} follows
 from~\eqref{eq:correlate:eps}.
\end{proof}

\subsection{Moment bounds}\label{sec:moment-overview}

In the last part of this section, we state two lemmas providing different types of moment bounds.
The first of these is needed to establish the moment bounds in the conclusion of the
split polytope anti-correlation lemma, Lemma~\ref{lem:anticor}.

\begin{lemma}[moment bounds lemma with deletion]\label{lem:deletionbounds}
Given $k\ge2$, \(\delta>0\), $c\ge1$ and \(d''\in\nats_{>0}\), there exists \(L_0\) such that, if \(L\geq L_0\), then there exists \(C\) such that for sufficiently large $n$ and $p\in(0,1]$ the following holds. Assume we are in Setting~\ref{mainSetting} and that the $(c,C,p,n)$-boundedness conditions hold. Let \(X=[n]_p\) be the binomial random set. Then, with probability at least \(1-3\exp(-\tfrac1{12}\delta^2 pn)\) over the choice of this random set, the following holds.

Using the notation of Definition~\ref{notationForSplitting}, choose a $\one$-split \(\chi_\mathbf{1}\colon [n]\to[\lceil Lp^{-1}\rceil]\) and a $\mu$-split \(\chi_\mu\colon[n]\to[L]\) independently and uniformly at random. Then, with probability at least \(0.9\) over the choice of these two splits, there is a \(\delta\)-deletion \(\tilde{X}\) of \(X\) with the following property.
  
For any \(\ell\in[d'']\), any $\phi\in\Phi(\tilde{\mu}_1,\dots{},\tilde{\mu}_L,\nu_1,\dots{},\nu_{\lceil Lp^{-1}\rceil})^\ell$, and any $i\in[L]$ and $j\in[\lceil Lp^{-1}\rceil]$, we have
\[ \langle\tmu_i,\phi\rangle\le 2 c^\ell\qquad\text{and}\qquad\langle\nu_j,\phi\rangle\le 2 c^\ell\,.\]
\end{lemma}

The proof of this lemma, in Section~\ref{sec:delete}, uses a trick, which is in this form due to Sp\"ohel, Steger and Warnke~\cite{SSW} but dates more generally back to the `deletion method' of R\"odl and Ruci\'nski~\cite{RR}. More precisely, we use the Harris inequality to `upgrade' a high probability statement to the above exponentially small failure probability, at the cost of making some deletions.

Our second moment bound lemma is stronger in that it provides a high probability statement over the choice of the random set~$X$ and the random splits, but weaker in that it does not provide an exponentially small failure probability. It also does not require deletions.

\begin{lemma}[moment bounds lemma without deletion]\label{cor:momentboundcor}
  Given $k\ge2$, $c\ge1$ and \(d\in\nats_{>0}\), there exists \(L_0\) such that, if \(L\geq L_0\), then there exists \(C\) such that for sufficiently large $n$ and $p\in(0,1]$ the following holds. Assume we are in Setting~\ref{mainSetting}, that the $(c,C,p,n)$-boundedness conditions hold, and use the notation of Definition~\ref{notationForSplitting}. Let \(X=[n]_p\) be the binomial random set, choose a $\one$-split \(\chi_\mathbf{1}:[n]\to \{1,\dots{},\lceil Lp^{-1}\rceil\}\) independently and uniformly at random, and a $\mu$-split \(\chi_\mu:[n]\to \{1,\dots{},L\}\) independently and uniformly at random. Then, with high probability over the choice of this random set and these two splits, the following holds.

For any \(\ell \in [d]\), any \(\phi\in \Phi(\mu_1,\dots{},\mu_L,\nu_1,\dots{},\nu_{\lceil Lp^{-1}\rceil})^\ell\), and any $i\in[L]$ and $j\in\big[\lceil Lp^{-1}\rceil\big]$, we have
\[ \langle\mu_i,\phi\rangle\le 2 c^\ell \qquad\text{and}\qquad \langle\nu_{j},\phi\rangle\le 2 c^\ell\,.\]
\end{lemma}

This lemma will be used to prove Lemma~\ref{lem:deletionbounds}. We will, however, also apply it in the proof of Theorem~\ref{thm:intmodel}. In this application we will set $p=1$ and ignore $\chi_\mu$ and $\mu_1,\dots,\mu_L$, obtaining that with high probability, for each $j\in[L]$ and $1\le\ell\le d$ we have $\langle\nu_j,\phi\rangle\le 2c^\ell$ for each $\phi\in\Phi(\nu_1,\dots,\nu_L)^\ell$.

The proof of Lemma~\ref{cor:momentboundcor} proceeds in three steps, provided in Section~\ref{sec:momentbounds}.
First, we argue that we do not need to consider every $\phi\in\Phi(\mu_1,\dots{},\mu_L,\nu_1,\dots{},\nu_{\lceil Lp^{-1}\rceil})^\ell$. Since this is a linear optimisation problem, as usual we only need to consider vertices. But much more is true this time. Because $\mu_i$ and $\nu_{j}$ are non-negative, we can assume that the vertices we need to consider are pointwise maximised. By Lemma~\ref{lem:vtxprod} and Fact~\ref{Rmk:anti_uniform_functionals}, this means the only vertices we need to consider are products of $\{\mu_1,\dots{},\mu_L,\nu_1,\dots{},\nu_{\lceil Lp^{-1}\rceil}\}$-largest functions. We shall show that there are only polynomially many such functions, and hence we will get away with quite weak concentration bounds (which is good, because we do not know how to prove strong ones in this scenario).
The second step is (roughly) to prove the required concentration bound would be true if $\mu_1,\dots{},\mu_L,\nu_1,\dots{},\nu_{\lceil Lp^{-1}\rceil}$ were mutually independent functions (which they are not). This is a fairly straightforward application of the Kim--Vu polynomial concentration inequality.
The final step deals with the dependency of $\mu_1,\dots{},\mu_L,\nu_1,\dots{},\nu_{\lceil Lp^{-1}\rceil}$,
by using a coupling argument
and applies the union bound over the $\{\mu_1,\dots{},\mu_L,\nu_1,\dots{},\nu_{\lceil Lp^{-1}\rceil}\}$-largest vertices.

\medskip

This finishes our more detailed sketch of the ideas needed to establish our main results.

\section{Proof that \texorpdfstring{$[0,1]$}{[0,1]}-valued models imply \texorpdfstring{$\{0,1\}$}{{0,1}}-valued models}\label{sec:intmodels}
    
In this section we prove Theorem~\ref{thm:intmodel}. As discussed, the basic approach is to define \(g^\ast_i\) via \emph{randomised rounding}. That is, independently for each \(x\), we generate \(g^\ast_i(x)\) by choosing \(1\) with probability \(g_i(x)\) and \(0\) otherwise, motivating the following definition.

\begin{definition}[randomised rounding]
  If $g\colon[n]\to[0,1]$ is any function, its \emph{randomised rounding} is the (random) function $g^\ast\colon[n]\to\{0,1\}$ with
  \[g^\ast(x)=\begin{cases}
  1 \qquad & \text{with probability $g(x)$}\,, \\
  0 \qquad & \text{with probability $1-g(x)$}\,, \\
  \end{cases}
  \]
  independently for each $x$.
\end{definition}

We then argue that rounding in this way is likely to lead to the required
closeness in norm. The following lemma states that this holds for a single function.

\begin{lemma}\label{lem:randround}
Given $k\ge2$, $c\ge1$ and $\eps>0$ there exists $C>0$ such that the following holds for all sufficiently large $n$. Let us be in Setting~\ref{mainSetting}, and suppose that the $(c,C,1,n)$-boundedness conditions hold. 

For any \(g\colon[n]\to [0,1]\), with high probability the randomised rounding \(g^\ast\colon[n]\to \{0,1\}\) satisfies \(\|g-g^\ast\|_{\Phi(\mathbf{1})}\le\eps\).
\end{lemma}
 
Before proving this lemma, we show how it implies Theorem~\ref{thm:intmodel}.

\begin{proof}[Proof of Theorem~\ref{thm:intmodel}]
 Given $k,r\ge2$ and $c\ge1$, $\eps>0$, let $C$ be returned by Lemma~\ref{lem:randround}. Given $p\in(0,1]$, let us be in Setting~\ref{mainSetting} and suppose that the $(c,C,p,n)$-boundedness conditions hold, which implies that in particular the $(c,C,1,n)$-boundedness conditions hold.
 Given $g_1,\dots,g_r$ non-negative such that $g_1+\dots+g_r=\one$, we define $g_1^\ast,\dots,g_r^\ast$ as follows. For each $x\in[n]$ independently, choose $i\in[r]$ with probability $g_i(x)$. We set $g^\ast_i(x)=1$ for the chosen $i$, and the other functions at $x$ to $0$. This gives $g_1^\ast+\dots+g^\ast_r=\one$.

 Observe that for each $i$, the function $g^\ast_i$ is distributed as the randomised rounding of $g_i$, so by Lemma~\ref{lem:randround}, for each $i$ with high probability we have \(\|g_i-g_i^\ast\|_{\Phi(\mathbf{1})}\le\eps\). Taking the union bound over $r$, with high probability the norm statement holds for all $i\in[r]$ simultaneously. In particular the required functions exist, proving Theorem~\ref{thm:intmodel}.
\end{proof}

We now explain how we prove Lemma~\ref{lem:randround}. The first idea is (much as discussed for the proof of~\ref{lincor:i} in Section~\ref{sec:reduction-two}) to fix a vertex $\phi$ of $\Phi(\one)$, argue that $\langle g-g^\ast,\phi\rangle$ is small since the entries of $g-g^\ast$ are mean zero random variables, and use the union bound.
The (only) problem is that $\Phi(\one)$ has too many vertices for this union bound. We use random splitting (as in Definition~\ref{notationForSplitting}) to get around this problem. Because we have only the function $\one$ to split, we only need one random split $\chi_\one$, which we define as in Definition~\ref{notationForSplitting} with $p=1$. We repeat the definition here for convenience.

\begin{definition*}[random split]\label{def:random-split}
  Let \(L\) be a positive integer, and let \(\chi_\one:[n]\to[L]\) be chosen uniformly at random.
  For \(i\in [L]\) we then denote by \(\nu_i\) the function
  \[\nu_i(x)=\begin{cases} L\qquad&\text{if \(\chi_\one(x)=i\)}\,,\\0&\text{otherwise}\,.\end{cases}\]
  We have \(\mathbf{1}=\tfrac{1}{L}\sum_{i=1}^L\nu_i\), hence we call $\nu_1,\dots,\nu_L$ a \emph{random split} of \(\one\).
\end{definition*}

For a carefully chosen $d$, we apply Lemma~\ref{cor:momentboundcor} with $p=1$ to fix a $\one$-split $\chi_\one$ so that for each $\ell\in[d]$ we have $\langle\nu_i,\phi\rangle\le 2c^\ell$ for all $\phi\in\Phi(\nu_1,\dots,\nu_L)^\ell$. The polytope $\Phi(\nu_1,\dots,\nu_L)$ contains the polytope $\Phi(\one)$ we are interested in.
For $\phi\in\Phi(\nu_1,\dots,\nu_L)$, we define (for this section only)
\[\phi^\mathrm{small}(x)=\phi(x)\mathbbm{1}\big(|\phi(x)|\le 2c\big)\quad\text{and}\quad\phi^\mathrm{big}=\phi-\phi^\mathrm{small}\,.\]
Hence, we can write \[\langle g-g^\ast,\phi\rangle=\langle g-g^\ast,\phi^\mathrm{small}\rangle+\langle g-g^\ast,\phi^\mathrm{big}\rangle\,.\] 
Now $\langle g-g^\ast,\phi^\smll\rangle$ is a sum of mean zero random variables taking values in $[-2c,2c]$, so Bernstein's inequality tells us it is very likely to be small. We can then take a union bound over choices of vertices $\phi\in\Phi(\nu_1,\dots,\nu_L)$. For the second term in the sum above, we write
\[\big|\langle g-g^\ast,\phi^\bg\rangle\big|\le\big|\langle g,\phi^\bg\rangle\big|+\big|\langle g^\ast,\phi^\bg\rangle\big|\]
and then use moment bounds to argue that these two inner products are small.

\begin{proof}[Proof of Lemma~\ref{lem:randround}]
  Given $k\ge2$, $c\ge1$ and \(\eps>0\), let \(d\geq 4\) be an even integer such that
  \[2\cdot(2c)^{2-d}\cdot 2c^d\leq \frac{\eps}{2}\,.\]
  Let~$L_0$ be returned by Lemma~\ref{cor:momentboundcor} for input $k$, $d$ and $c$.
  Set
  \[L=\max\{L_0, 100(k-1)2\eps^{-2}c^2,2k\}\,,\]
  and let~$C'$ be returned by Lemma~\ref{cor:momentboundcor} for this~$L$.
  Set
  \[C=\max\{C',100\eps^{-2}c^2\}\,.\]
  Suppose $n$ is sufficiently large, that we are in Setting~\ref{mainSetting} and that the $(c,C,1,n)$-boundedness conditions hold.  
  Given $g:[n]\to[0,1]$, let \(g^\ast\) be the randomised rounding of \(g\).

By Lemma~\ref{cor:momentboundcor}, with high probability there exists a $\one$-split $\chi_\one$ corresponding to a random split \(\nu_1,\dots{},\nu_L\) of $\one$ such that
\begin{equation}\label{eq:intbounds:moment}
\langle \nu_j,\phi\rangle\le 2c^\ell
\end{equation}
for each $j\in[L]$, each $1\le\ell\le d$, and each $\phi\in\Phi(\nu_1,\dots,\nu_L)^\ell$. Fix $\nu_1,\dots,\nu_L$ such that this likely event occurs. In particular, since $\one\in\Sigma\subset\Phi(\nu_1,\dots,\nu_L)$, we have $\langle\nu_j,\one\rangle\le 2$, and thus
\[\nu_j \text{ has at most $2n/L$ non-zero entries}\]
for each $j\in[L]$.

For a vertex \(\phi\) of \(\Phi(\nu_1,\dots{},\nu_L)\), we write \(\phi^\mathrm{small}(x)=\phi(x)\ind\big(|\phi(x)|\le 2c\big)\) and \(\phi^\mathrm{big}=\phi-\phi^\mathrm{small}\). We first prove that with high probability, for each vertex \(\phi\) of \(\Phi(\nu_1,\dots{},\nu_L)\) we have \(\langle g-g^\ast,\phi^\mathrm{small}\rangle\leq \frac{\eps}{2}\). We then show that deterministically a similar statement holds for \(\phi^\mathrm{big}\).

For \(\phi^\mathrm{small}\), we do this by union-bounding, over the choice of \(\phi\), the probability of the event $\mathcal{E}$ that \[\big|\langle g-g^\ast,\phi^\mathrm{small}\rangle\big|\le\frac12\eps  \qquad\text{for all vertices $\phi$ of }\Phi(\nu_1,\dots{},\nu_L)\,.\]
To apply the union bound, we start by arguing that the number~$N$ of vertices of \(\Phi(\nu_1,\dots{},\nu_L)\) satisfies
 \[N\le|\Sigma|+k\cdot |\Omega|\cdot L^{k-1}2^{(k-1)2n/L}\,.\]
 Indeed, by Fact~\ref{Rmk:anti_uniform_functionals} every vertex of \(\Phi(\nu_1,\dots{},\nu_L)\) is either in \(\Sigma\) or a \(\nu_1,\dots{},\nu_L\)-extreme test function that is a convolution \(\ast_{i,\omega}(f_1,\dots{},f_{k})\). For the latter there are \(k\) choices for~\(i\); there are \(\vass{\Omega}\) choices for \(\omega\); and each of the \((k-1)\) functions \(f_j\) is determined by the selection of a function $\nu_i$ with $i=i(j)\in[L]$ as well as a subset of the at most \(2n/L\) non-zero entries of the selected~$\nu_i$.
 It follows from the $(c,C,1,n)$-boundedness conditions that the number of vertices of \(\Phi(\nu_1,\dots{},\nu_L)\) is
 \begin{equation*}\begin{split}
     N&\le \exp(C^{-1}n)+k\cdot\exp(C^{-1}n)\cdot L^{k-1}2^{(k-1)2n/L}\le 2k\cdot\exp(C^{-1}n)\cdot L^{k-1}2^{(k-1)2n/L} \\
     &\le L^k\cdot\exp\Big((100\eps^{-2}c^2)^{-1}n\Big)\cdot\exp\Big(\frac{\eps^2 n}{100c^2}\Big)
     \le\exp\Big(\frac{\eps^2 n}{40c^2}\Big)
     \,,
 \end{split}\end{equation*}
 where we use the definition of~$C$ and~$L$ in the penultimate inequality.
 
 We now observe that, considering \(g^\ast\) as a random variable with \(\mathbb{E}\big[g^\ast(x)\big]=g(x)\), we have that
 \[\langle g-g^\ast,\phi^\mathrm{small}\rangle
    = \sum_{x\in[n]} \frac1n \big(g(x)-g^\ast(x)\big)\phi(x)\ind\big(|\phi(x)|\le 2c\big)
 \]
 is a sum of \(n\) random variables with mean~$0$, each taking values in $[-2cn^{-1},2cn^{-1}]$, and so with variance at most \(4c^2n^{-2}\). Applying Bernstein's inequality (Lemma~\ref{lem:bernstein}) with $\lambda=\eps/2$, we have
 \[\mathbb{P}\Big[\langle g-g^\ast,\phi^\mathrm{small}\rangle\ge\frac\eps2\Big]\le \exp\Big(-\frac{\eps^2/4}{\tfrac{1}{3}\cdot2cn^{-1}\cdot\tfrac12\eps+n\cdot 4c^2n^{-2}}\Big)\le\exp\Big(-\frac{\eps^2n}{32c^2}\Big)\,.\]
 Taking the union bound over the~$N$ vertices of \(\Phi(\nu_1,\dots{},\nu_L)\), we obtain that with high probability the event $\mathcal{E}$ holds. Fix such a \(g^\ast\).

 Fix a \(\phi\) in \(\Phi(\nu_1,\dots{},\nu_L)\). We now prove that \(\langle g-g^\ast,\phi^\mathrm{big}\rangle\leq \frac{\eps}{2}\), deterministically. Indeed, using the triangle inequality, that \(g\) and \(g^\ast\) are non-negative,
 that all entries of \(|\phi^\mathrm{big}|\) are either zero or at least \(2c\), and that $\phi^d$ is non-negative because~$d$ is even,
 we obtain
 \begin{align*}
  \big|\langle g-g^\ast,\phi^\mathrm{big}\rangle\big|&\le\big|\langle g,\phi^\mathrm{big}\rangle\big|+\big|\langle g^\ast,\phi^\mathrm{big}\rangle\big|
  \le\big\langle g,|\phi^\mathrm{big}|\big\rangle+\big\langle g^\ast,|\phi^\mathrm{big}|\big\rangle\le 2\big\langle \mathbf{1},|\phi^\mathrm{big}|\big\rangle\\
  &\le2\cdot\big\langle\mathbf{1},(\phi^\mathrm{big})^2\big\rangle\le 2\cdot (2c)^{2-d}\big\langle\mathbf{1},(\phi^\mathrm{big})^{d}\big\rangle
  \le 2\cdot(2c)^{2-d}\langle \one,\phi^d\rangle\,.
  \end{align*}
Since $\one=\tfrac{1}{L}\sum_{j=1}^L\nu_j$, and since $\phi^d\in\Phi(\nu_1,\dots,\nu_L)^d$, by~\eqref{eq:intbounds:moment} we have $\langle\one,\phi^d\rangle\le 2c^d$, and so
  \[\big|\langle g-g^\ast,\phi^\mathrm{big}\rangle\big|\le 2\cdot(2c)^{2-d}\cdot 2c^d\leq \frac{\eps}{2}\,,\]
  where the final inequality is by choice of $d$.

 Putting these two estimates together, we have for every vertex \(\phi\) of \(\Phi(\nu_1,\dots{},\nu_L)\) the bound
 \[\langle g-g^\ast,\phi\rangle\le\frac12\eps+\frac12\eps=\eps\,.\]
 Since linear functions over a polytope are maximised at vertices, we conclude the same bound holds for every \(\phi\in\Phi(\nu_1,\dots{},\nu_L)\).

 From Lemma~\ref{lem:splitreduce} applied with $p=1$, $\tmu=\mu=\one$, and $\chi_\mu=\chi_\one$, we conclude that \(\Phi(\mathbf{1})\subset\Phi(\nu_1,\dots{},\nu_L)\). Hence $\langle g-g^\ast,\phi\rangle\le\eps$ for each $\phi\in\Phi(\mathbf{1})$, completing the proof of the lemma.
 \end{proof}
 
\section{Proof of the anti-correlation reduction lemma}\label{sec:correlate}

In this section we prove Lemma~\ref{lem:HBreduce}. We need the following classical theorem, which states that for any convex closed set containing zero and any point outside the set there is a hyperplane separating the set from the point.

\begin{theorem}[separating hyperplane theorem]\label{thm:Hahn-Banach}
Let \(K\) be a closed convex set in \(\mathbb{R}^{[n]}\) that contains the zero vector and let \(f\) be a vector that
does not belong to \(K\). Then there is  \(\psi\in\mathbb{R}^{[n]}\) such that \(\langle f,\psi\rangle>1\) and \(\langle g,\psi\rangle\leq 1\) for every \(g\in K\).
\end{theorem}

The concept of the
\emph{dual norm} $\|\cdot\|_{\Phi}^\ast$ of our norm $\|\cdot\|_{\Phi}$ also turns out useful for our proof: for $\psi\in\RR^{[n]}$, we define
\[\|\psi\|_{\Phi}^\ast:=\max_{\|h\|_\Phi\le 1}\langle h,\psi\rangle\,.\]
We shall use the following easy fact about the dual norm.

\begin{fact}\label{fact:dual}
  Let~$\Phi$ be a centrally symmetric $n$-dimensional polytope in $\RR^{[n]}$.
  If $\|\psi\|_\Phi^\ast\le1$, then $\psi\in\Phi$.
\end{fact}

To see this is true, observe that if $\psi\not\in\Phi$ then by
Theorem~\ref{thm:Hahn-Banach} there is some $h'$ such that
$\langle\psi,h'\rangle>1$ but $\langle \phi,h'\rangle\le 1$ for all
$\phi\in\Phi$. By definition of $\|\cdot\|_{\Phi}$, this says $\|h'\|_\Phi\le 1$. But then, by
definition of $\|\cdot\|_{\Phi}^\ast$, we have $\|\psi\|_\Phi^\ast\ge\langle h',\psi\rangle>1$.

\begin{proof}[Proof of Lemma~\ref{lem:HBreduce}]
  Given $r\ge2$ and $c\ge1,\eps>0$, assume without loss of generality that $\eps\le 1/10$ and let
  \[\delta=\frac{\eps^2}{100cr^2}\,,  \qquad\text{and}\qquad \gamma=\frac{\eps}{5cr}
  \,.\]
Given $n$, suppose $\xi\colon[n]\to\RR_{\ge0}$ and $\Phi\subset\RR^{[n]}$ satisfy
\ref{itm:HB:small}, \ref{itm:HB:moment}, and \ref{itm:HB:max}.
Now let non-negative $f_1,\dots,f_r$ be given, with $f_1+\dots+f_{r}=\xi$.
We will first show that for the first $r-1$ of these functions, which clearly satisfy
\[f_1+\dots+f_{r-1}\le\xi\,,\]
the following claim is true. We will then argue that this implies our lemma.

\begin{claim}
  There are non-negative functions $g_1,\dots,g_{r-1}$
  with $g_1+\dots+g_{r-1}\le\one$ and functions $h_1,\dots,h_{r-1}$ with $\|h_i\|_{\Phi}\le\eps/(10r)$ for each $i\in[r-1]$, such that
  $f_i=(1+\gamma)(g_i+h_i)$ for each $i\in[r-1]$.
\end{claim}
\begin{claimproof}
  Suppose for a contradiction that this is not true. The set \(K\subset\RR^{[(r-1)n]}\) of $(r-1)$-tuples of
  functions of the form \((g_1+h_1,\dots,g_{r-1}+h_{r-1})\), where each \(g_i\)
  is non-negative and satisfies $g_1+\dots+g_{r-1}\le\one$, and each $h_i$ satisfies
  \(\|h_i\|_\Phi\leq\eps/(10r)\), is a closed convex set
  containing the zero vector, since $g_i=h_i=0$ is feasible.  By Theorem~\ref{thm:Hahn-Banach}, because
  \(\tfrac{1}{1+\gamma}(f_1,\dots,f_{r-1})\) is not in \(K\) by assumption,
  there are
  \(\psi_1,\dots,\psi_{r-1}\in\mathbb{R}^{[n]}\) such that
  \begin{align}
     \label{eq:sep1} \Big\langle \frac{1}{1+\gamma}(f_1,\dots,f_{r-1}),(\psi_1,\dots,\psi_{r-1})\Big\rangle&>1\,,\quad\text{but} \\
     \label{eq:sep2} \big\langle (g_1+h_1,\dots,g_{r-1}+h_{r-1}),(\psi_1,\dots,\psi_{r-1})\big\rangle&\le 1
  \end{align}
  for all $(g_1+h_1,\dots,g_{r-1}+h_{r-1})\in K$.
  From~\eqref{eq:sep1}, we get
  \begin{equation}\label{eq:HB:sum}
    \frac{1}{r-1}\sum_{i=1}^{r-1}\langle f_i,\psi_i\rangle>1+\gamma\,.
  \end{equation}
  Further, for each $i\in[r-1]$ and each $h$ with $\|h\|_\Phi\le\eps/(10r)$,
  taking $h_1=h$ and $h_2=\dots=h_{r-1}=g_1=\dots=g_{r-1}=0$, we conclude
  from~\eqref{eq:sep2} that
  \[\big\langle(g_1+h_1,\dots,g_{r-1}+h_{r-1}),(\psi_1,\dots,\psi_{r-1})\big\rangle=\frac{1}{r-1}\langle
  h,\psi_i\rangle\le 1\,.\]
  It follows that for each $h'$ with $\|h'\|_\Phi\le 1$ we have $\langle
  h',\psi_i\rangle\le(r-1)(10r)/\eps\le 10r^2/\eps$, and thus
  the dual norm of $\psi_i$ satisfies $\|\psi_i\|_\Phi^\ast\le 10r^2/\eps$.
  Therefore, by Fact~\ref{fact:dual}, $\phi_i=\tfrac{\eps}{10r^2}\psi_i$ is in $\Phi$. Since we have
  $\langle\xi-\one,\max(0,\phi_1,\dots,\phi_{r-1})\rangle<\delta$ by~\ref{itm:HB:max}, we conclude
  that
  \begin{equation}\label{eq:HB:max}
    \langle\xi-\one,\max(0,\psi_1,\dots,\psi_{r-1})\rangle\le 10r^2\delta/\eps\,.
  \end{equation}  
  
  Now, we want to appeal to~\eqref{eq:sep2} again, this time with $h_1=\dots=h_{r-1}=0$.
  Observe that we can choose $g_i$ with $g_i(x)\in\{0,1\}$ and $\sum_{i=1}^{r-1}g_i(x)\le 1$ for each~$x$
  such that $\sum_{i=1}^{r-1}g_i\psi_i=\max(0,\psi_1,\dots,\psi_{r-1})$. From~\eqref{eq:sep2} we get
  \begin{equation}\label{eq:HB:one}
    \langle(g_1,\dots,g_{r-1}),(\psi_1,\dots,\psi_{r-1})\rangle=\tfrac{1}{r-1}\langle\one,\max(0,\psi_1,\dots,\psi_{r-1})\rangle\le 1\,.
  \end{equation}
  Letting $\psi_+=\max(0,\psi_1,\dots,\psi_{r-1})$ and putting these pieces
  together we can now write
  \begin{equation}\label{anticorrelationContradiction}\begin{split}
    1+\gamma&\lByRef{eq:HB:sum}\frac{1}{r-1}\sum_{i=1}^{r-1}\langle f_i,\psi_i\rangle\le \frac{1}{r-1}\sum_{i=1}^{r-1}\langle f_i,\psi_+\rangle
    = \frac{1}{r-1}\Big\langle \sum_{i=1}^{r-1}f_i,\psi_+\Big\rangle\le\frac{1}{r-1}\langle\xi,\psi_+\rangle
    \\&=\frac{1}{r-1}\langle \one,\psi_+\rangle+ \frac{1}{r-1}\langle\xi-\one,\psi_+\rangle
    \leBy{\eqref{eq:HB:max},\eqref{eq:HB:one}} 20r\delta\eps^{-1} + 1\,,
    \end{split}
  \end{equation}
  where the
  second inequality uses $\psi_i\le\psi_+$ and $f_i\ge0$ and the third uses
  \(f_1+\dots+f_{r-1}\le\xi\) and $\psi_+\ge0$. We infer from
  \eqref{anticorrelationContradiction} that
  \(\gamma <20r\delta\eps^{-1}=\eps/(5cr)\), contradicting the choice of~$\gamma$.
\end{claimproof}

Now, we use the functions $g_1,\dots,g_{r-1}$ and $h_1,\dots,h_{r-1}$ from this
claim, and let $g_r=\one-g_1-\dots-g_{r-1}$. We claim that these functions
satisfy the conclusion of our lemma. Indeed, clearly $g_1+\dots+g_r=\one$. It
remains to establish closeness in norm to the $f_i$. For this, we begin with
$f_r$ and $g_r$. Since $f_i=(1+\gamma)(g_i+h_i)$ for each $i\in[r-1]$, we have
\begin{align*}
    f_r-g_r&=\xi-f_1-\dots-f_{r-1}-\one+g_1+\dots+g_{r-1}\\
    &=\xi-\one-\gamma(g_1+\dots+g_{r-1})-(1+\gamma)(h_1+\dots+h_{r-1})=:h_r\,.
\end{align*}
By the triangle inequality, using that
$\|g_1+\dots+g_{r-1}\|_\Phi\le\|\one\|_\Phi\le 2c$ by~\ref{itm:HB:moment}, and
that $\|\xi-\one\|_\Phi\le\delta$ by~\ref{itm:HB:small}, we have
\begin{equation*}\begin{split}
  \|h_r\|_\Phi&\le
  \|\xi-\one\|_\Phi+\gamma\|g_1+\dots+g_{r-1}\|_\Phi+2\big(\|h_1\|_\Phi+\dots+\|h_{r-1}\|_\Phi\big) \\
  &\le
  \delta+\gamma\cdot 2c+2(r-1)\eps/(10r)\le\eps\,,
\end{split}\end{equation*}
where the final inequality is by choice of~$\gamma$ and~$\delta$. This establishes $\|f_r-g_r\|_\Phi\le\eps$ as required.
Similarly, for each $i\in[r-1]$ we have $f_i-g_i=\gamma g_i+(1+\gamma)h_i$, and applying the triangle inequality and the fact that $0\le g_i\le\one$ and hence $\|g_i\|_\Phi\le\|\one\|_\Phi\le 2c$ by~\ref{itm:HB:moment}, we get
\[\|f_i-g_i\|_\Phi\le
\gamma\|g_i\|_\Phi+(1+\gamma)\|h_i\|_\Phi
\le \gamma\cdot 2c+2\eps/(10r)\le\eps\,,\]
by choice of~$\gamma$.
\end{proof}

\section{Proof of the linear anti-correlation lemma}\label{sec:linear}

In this section we prove Lemma~\ref{lem:lincor}, which reduces the non-linear anti-correlation that we need to prove for Theorem~\ref{thm:correlate} to a linear anti-correlation statement over a product polytope and a moment bound. As discussed in Section~\ref{sec:tools}, we use the Stone--Weierstrass approximation theorem (see, e.g., \cite[Theorem~7.26]{Rudin}).

\begin{theorem}[Stone--Weierstrass approximation theorem]\label{Weierstrass}
Given $t\ge1$, let $f$ be a continuous real-valued function on $[a,b]^t$. For every \(\delta>0\) there exists a polynomial \(\pi\) with real coefficients such that for every \(\mathbf{x}\in [a,b]^t\) we have \(\vass{\pi(\mathbf{x})-f(\mathbf{x})}\leq \delta\).
\end{theorem}
As discussed, we apply this with the domain $[-2c,2c]^{k-1}$, and we need to deal with exceptionally large values using moment bounds. 
To this end, in this section only, for any function \(\phi\) on \([n]\), write \(\phi^{\mathrm{big}}\) for the function which takes the value \(\phi(x)\) on \(x\in [n]\) whenever \(|\phi(x)|>2c\), and \(0\) otherwise, and \(\phi^{\mathrm{small}}=\phi - \phi^{\mathrm{big}}\). That is,

\[
\phi^{\mathrm{small}}(x)= 
\begin{cases}
    \phi(x) & \textup{if \(\phi(x) \in [-2c,2c]\)}\\
    0       & \textup{otherwise}
\end{cases}
\qquad
\phi^{\mathrm{big}}(x)= 
\begin{cases}
    0       & \textup{if \(\phi(x) \in [-2c,2c]\)}\\
    \phi(x) & \textup{otherwise}
\end{cases}\,.
\]

Note that \(\phi^{\mathrm{big}}\) and \(\phi^{\mathrm{small}}\) have disjoint support, so in particular $\phi^\bg\phi^\smll$ is the constant zero function and it follows that $\pi(\phi)=\pi(\phi^\bg)+\pi(\phi^\smll)-\pi(0)$, where $0$ is the constant zero function.

\begin{proof}[Proof of Lemma~\ref{lem:lincor}]
  Given $r\ge2$, $c\ge1$, and $\eps>0$, we first need to set the required constants $\eps_\lin,d,d'$. For this purpose, 
  consider the function \[(x_1,\dots,x_{r-1})\to\max(0,x_1,\dots,x_{r-1})\] on the domain $[-2c,2c]^{r-1}$. This is a continuous function, so by the Stone--Weierstrass approximation theorem (Theorem~\ref{Weierstrass}) applied with $\delta=\eps/16c$, there is a polynomial \(\pi(x_1,\dots,x_{r-1})\)
  such that for every \((x_1,\dots,x_{r-1})\in [-2c,2c]^{r-1}\) we have
  \[\big|\pi(x_1,\dots,x_{r-1})-\max(0,x_1,\dots,x_{r-1})\big|\le \frac{\eps}{16c}\,.\]
  Let~$d$ be the maximum of the degree of~$\pi$ and~$2$. Further, let~$M$ be the maximum among the absolute values of the coefficients of $\pi$ as well as the number of monomials in $\pi$. We set
  \[\eps_\lin:=\frac14M^{-2}\eps\qquad\text{and}\qquad M_S:=2^{r-1}M\]
  and choose~$d''$ sufficiently large so that
  \[M_S\cdot M\cdot (2c)^d\cdot 2^{2d+1}c^{2d}\cdot 2^{-2d''}\le\frac14\eps\,.\]
  Finally, we set $d':=2dd''$.

Our goal is to upper bound \[\big|\langle \xi - \mathbf{1} , \max(0,\phi_1,\dots,\phi_{r-1})\rangle\big|\] given arbitrary \(\phi_1,\dots,\phi_{r-1}\in\Phi\).
We split this inner product into several parts.
We claim that 
\begin{multline}\label{eq:lincor:split}
\big|\langle \xi - \mathbf{1} , \max(0,\phi_1,\dots,\phi_{r-1})\rangle\big|
\\ \le \big|\langle \xi - \mathbf{1} , \max(0,\phi^\smll_1,\dots,\phi^\smll_{r-1})\rangle\big|
+ \sum_{i=1}^{r-1}\big|\langle\xi,(\phi^\bg_i)^2\rangle\big|+\big|\langle\one,(\phi_i^\bg)^2\rangle\big|\,,
\end{multline}
where we call the first term the \emph{main term}. Indeed,
the sum upper bounds the contributions of those $x\in[n]$ that satisfy $\max(0,\phi_1,\dots,\phi_{r-1})(x)>2c$ and hence are in the support of at least one $\phi^\bg_i$. Using the triangle inequality and the fact that $2c>1$ and $\phi_i^2,\xi,\one\ge0$, we obtain the claimed inequality.

Let us first expand the main term of~\eqref{eq:lincor:split} further. Consider the polynomial
$\pi(\phi^\smll_1+\phi^\bg_1,\dots,\phi^\smll_{r-1}+\phi^\bg_{r-1})$. Expanding this out, we obtain $\pi(\phi^\smll_1,\dots,\phi^\smll_{r-1})+S$, where $\mathcal{S}$ is a collection of monomials, each of which contains at least one term $\phi^\bg_i$ for some $i$. The absolute value of the coefficient of any given monomial is at most $M$, as in $\pi$, and the total number of monomials is at most $M_\mathcal{S}=2^{r-1}M$.
because from each of the at most~$M$ monomials in $\pi(x_1,\dots,x_{r-1})$ we can choose to replace $x_i$ with either $\phi^\bg_i$ or $\phi^\smll_i$, but not both because the pointwise product $\phi^\bg_i\phi^\smll_i$ is the zero function and these terms therefore disappear. By the approximation property of $\pi$, and since the entries of each $\phi_i^\smll$ are in $[-2c,2c]$, we have
\[\max(0,\phi_1^\smll,\dots,\phi_{r-1}^\smll)=\pi(\phi_1^\smll,\dots,\phi_{r-1}^\smll)+h=\pi(\phi_1,\dots,\phi_{r-1})+h-\mathcal{S}\]
where the entries of $h$ have absolute value at most $\tfrac{1}{16c}\eps$.
We therefore get
\begin{multline}\label{eq:lincor:main}
  \big|\langle\xi-\one,\max(0,\phi^\smll_1,\dots,\phi^\smll_{r-1})\rangle\big| \\
  \le\big|\langle\xi-\one,\pi(\phi_1,\dots,\phi_{r-1})\rangle\big|+\big|\langle\xi,h\rangle\big|+\big|\langle\one,h\rangle\big|+\big|\langle\xi,\mathcal{S}\rangle\big|+\big|\langle\one,\mathcal{S}\rangle\big|\,.
\end{multline}
We now consider these terms separately. For bounding the first term, we split $\pi$ into its at most $M$ monomials $a_j\hat\phi_j$, where $a_j$ is the coefficient and $\hat\phi_j$ is a product of at most~$d$ functions~$\phi_i$. Since $\one\in\Phi$ we know that $\hat\phi_j$ is in~$\Phi^d$ for all~$j$.
Therefore, using assumption~\ref{lincor:i}, that there are~at most~$M$ monomials in~$\pi$, and that $|a|\le M$, we obtain
\[|\langle\xi-\one,\pi(\phi_1,\dots,\phi_{r-1})\rangle|\le M^2\eps_\lin\,.\]
For the two terms concerning inner products with~$h$ on the right hand side of~\eqref{eq:lincor:main}, observe that the worst case is $h=\tfrac1{16c}\eps\one$ and that since $\one\in\Phi$ we have $\big|\langle\xi,\one\rangle\big|\le 2c$ by~\ref{lincor:ii}. Hence,
\[\big|\langle\xi,h\rangle\big|,\big|\langle\one,h\rangle\big|\le \frac1{16c}\eps\cdot 2c=\frac18\eps\,.\]
For the other two terms of~\eqref{eq:lincor:main}, we split $\mathcal{S}$ into its at most~$M_\mathcal{S}$ monomials. Consider one such monomial $a\hat\phi_j=a\phi_{j,1}\dots\phi_{j,\ell}$, where $\ell=\ell(j)\in[d]$, each $\phi_{j,\ell}$ is either $\phi_i^\smll$ or $\phi_i^\bg$ for some $i$, and~$a$ is the coefficient, with $|a|\le M$. Without loss of generality, we can suppose the first $\ell'=\ell'(j)$ of these terms are the $\bg$ terms and the rest are the $\smll$ terms; by the definition of~$\mathcal{S}$ we have $\ell'\ge1$. Then
\[|\phi_{j,1}\dots\phi_{j,\ell}|\le (2c)^{\ell-\ell'}(\phi_{j,1}\dots \phi_{j,\ell'})^2\]
pointwise, since each entry of any $\phi_{j,i}$ with $i>\ell'$ has absolute value at most $2c$, while the entries of those~$\phi_{j,i}$ with $i\le\ell'$ are either $0$ or have absolute value at least $2c\ge1$.
The last observation also implies that each entry of $(\phi_{j,1}\dots \phi_{j,\ell'})^2$ is either zero or at least $(2c)^{2\ell'}$. Note further that, by the definition of the product polytope, $\psi=(\phi_{j,1}\dots \phi_{j,\ell'})^2$ is obtained from a non-negative $\phi\in\Phi^{2\ell'}$ by replacing all entries smaller than $(2c)^{2\ell'}$ by zero. Hence, we shall use the following claim to bound the contribution of such a term.

\begin{claim}\label{cl:lincor}
  If $\ell'\in[d]$ and $\psi$ is obtained from a non-negative $\phi\in\Phi^{2\ell'}$ by replacing all entries smaller than $(2c)^{2\ell'}$ by zero, then
  \[\langle\xi,\psi\rangle,\langle\one,\psi\rangle\le 2^{2d+1}c^{2d}\cdot 2^{-2d''}\,.\]
\end{claim}

Before we prove this claim, observe that together with~\eqref{eq:lincor:main}, our arguments above, and our choice of constants it implies
\begin{equation*}\begin{split}
  \big|\langle\xi-\one,\max(0,\phi^\smll_1,\dots,\phi^\smll_{r-1})\rangle\big|
  &\le M^2\eps_\lin + 2\cdot\frac18\eps + M_\mathcal{S}\cdot M\cdot (2c)^d\cdot 2^{2d+1}c^{2d}\cdot 2^{-2d''} \\
  &\le \frac14\eps + \frac14\eps + \frac14\eps = \frac34\eps
  \,,
\end{split}\end{equation*}
and hence from~\eqref{eq:lincor:split} we get
\begin{equation}\label{eq:lincor:remainder}
  \big|\langle \xi - \mathbf{1} , \max(0,\phi_1,\dots,\phi_{r-1})\rangle\big|
  \le \frac34\eps
   + \sum_{i=1}^{r-1}\big|\langle\xi,(\phi^\bg_i)^2\rangle\big|+\big|\langle\one,(\phi_i^\bg)^2\rangle\big|\,.
\end{equation}

\begin{claimproof}[Proof of Claim~\ref{cl:lincor}]
  Since the entries of $\psi$ are either~$0$ or at least $(2c)^{2\ell'}$, we can write
\[\psi\le (2c)^{2\ell'(1-d'')}\psi^{d''}\le(2c)^{2\ell'(1-d'')}\phi^{d''}\,,\]
where the second inequality uses that $\psi$ has only non-negative entries.
Now, recall that $\phi\in\Phi^{2\ell'}$ implies $\phi^{d''}\in\Phi^{2\ell'd''}$
and $2\ell'd''\le 2dd''=d'$.
Hence, since $2dd''=d'$, we can use assumption~\ref{lincor:ii} to conclude
\[\langle\xi,\psi\rangle\le(2c)^{2\ell'(1-d'')}\langle\xi,\phi^{d''}\rangle
\le(2c)^{2\ell'(1-d'')}2c^{2\ell'd''}
\le 2^{2\ell'(1-d'')+1}c^{2\ell'}\le2^{2d+1}c^{2d}2^{-2d''}\,,\]
where the last inequality uses $1\le\ell'\le d$.
By the same logic, replacing $\xi$ with $\one$, we obtain the same bound for $\langle\one,\psi\rangle$.
\end{claimproof}
Thus, it remains to bound the final sum in~\eqref{eq:lincor:remainder}.
However, $\psi=(\phi_i^\bg)^2\in\Phi^{2\ell'}$ is also of the form required to apply Claim~\ref{cl:lincor}, and hence we can bound the sum by
$2(r-1)\cdot 2^{2d+1}c^{2d}\cdot 2^{-2d''}\le\frac14\eps$, where the final inequality is by choice of $d''$.
It follows that
\[|\langle\xi-\one,\max(0,\phi_1,\dots,\phi_{r-1})\rangle|\le \eps\,,\]
as desired.
\end{proof}

\section{Proof of the split polytope anti-correlation lemma}\label{sec:finest}

In this section, we prove Lemma~\ref{lem:anticor}. Recall that the main work of this is to show that it is likely that \(\big|\langle\tilde{\mu}-\mathbf{1},\phi\rangle\big|<\eps\) for all \(\phi\in\Phi'=\Phi(\tmu_1,\dots,\tmu_L,\nu_1,\dots,\nu_{\lceil Lp^{-1}\rceil})^d\), where $\Phi'$ is the split polytope for our random split.
The proof of this looks quite similar to the proof of the corresponding statement from Theorem~\ref{thm:intmodel} in  Section~\ref{sec:intmodels}. As before, it is enough to prove anti-correlation for vertices of \(\Phi'\). And as before, we split the anti-correlation into anti-correlation with \(\phi^\mathrm{small}\) and the remaining \(\phi^\mathrm{big}\). We stress, though, that this time the dividing line is that entries of $\phi^\smll$ are bounded by $2c^d$ not $2c$. Thus, in this section only, for \(\phi\in \Phi(\tilde{\mu},\mathbf{1})^d\), we write
\[\phi^\mathrm{small}(x):=\phi(x)\ind\big(|\phi(x)|\le 2c^d\big)\quad\text{and}\quad\phi^\mathrm{big}(x)=\phi(x)-\phi^\mathrm{small}(x)\,.\]

Much as in  Section~\ref{sec:intmodels}, we can show that \(|\langle\tilde{\mu}-\mathbf{1},\phi^\mathrm{big}\rangle|\) is small by applying some moment bounds.
However, proving \(|\langle\tilde{\mu}-\mathbf{1},\phi^\mathrm{small}\rangle|\) is small requires some new ideas. There are two reasons for this: first, the entries of \(\tilde{\mu}\) are not independent random variables, and second, in order to describe a vertex of \(\Phi'\) we first need to reveal some entries of \(\tilde{\mu}\). We are going to deal with these two difficulties as follows.
Since \(\tilde{X}\) is an \(\eps\)-deletion of \(X\), we have that \(\mu\) and \(\tilde{\mu}\) are equal in most components. We show (using moment bounds) that it suffices to prove \(\langle\mu-\mathbf{1},\phi^\mathrm{small}\rangle\) is small. Now the entries of $\mu-\one$ are independent mean zero random variables. Recalling the notion of the revealed part $Y(\phi)$ of $\phi$ from Definition~\ref{notationForSplitting}, we will argue that the value of $\mu-\one$ on $[n]\setminus Y(\phi)$ is independent of $\phi$. Since $Y(\phi)$ is a tiny part of $[n]$ by choice of the number of parts~$L$ in our random split, for bounding $\big\langle\mu-\one,\phi^\smll\ind(Y(\phi))\big\rangle$ we can again use moment bounds.
Accordingly, the next lemma, in which we use Bernstein's inequality and the union bound over choices of~$\phi$ to control the inner products $\big\langle\mu-\one,\phi^\smll\ind([n]\setminus Y(\phi))\big\rangle$, is the critical part of the problem.

\begin{lemma}\label{lem:finbound}
Given $k\ge2$, $d\in\nats_{>0}$, $c\ge1$  and $\delta>0$, there is $L_0$ such that for each $L\ge L_0$ the following holds for $C\ge L$. Given $n\in\nats_{>0}$ and $p\in(0,1]$, assume we are in Setting~\ref{mainSetting} and that the $(c,C,p,n)$-boundedness conditions hold. 

With probability at least \(1-\exp\big(-\tfrac1{100c^{2d}}\delta^2pn\big)\) over the choices of \(X=[n]_p\), the $\mu$-split  \(\chi_\mu:[n]\to \{1,\dots{},L\}\), and the $\one$-split \(\chi_\mathbf{1}:[n]\to \{1,\dots{},\lceil Lp^{-1}\rceil\}\), the following holds. Given any \(\tilde{X}\subseteq X\), using the notation of Definition~\ref{notationForSplitting}, for each vertex \(\phi\) of \(\Phi'\) we have
 \begin{equation}\label{eq:finanticorr}\big|\langle\mu-\mathbf{1},\phi^\mathrm{small}\cdot \mathbbm{1}\big([n]\setminus Y(\phi)\big)\rangle\big|<\delta\,.\end{equation}
\end{lemma}

Something we need to be a little careful about in the above statement is that in order to know any vertices of \(\Phi'\), we need to reveal all of \(\mu\) (because \(\Phi'\) depends on \(\tilde{\mu}\)). We actually show the above bound for all extreme functions in \(\Phi(\mu_1,\dots{},\mu_L,\nu_1,\dots{},\nu_{\lceil Lp^{-1}\rceil})^d\), and deduce the required statement for \(\Phi'\) using the fact that vertices of $\Phi'$ are extreme functions in that polytope.

\begin{proof}
  Given $k,d,c,\delta$, we set $L_0=1000c^{2d}\delta^{-2}dk$. Given $L\ge L_0$ let $C\ge L$, let $n$ be sufficiently large, and let $p\in(0,1]$. Let us be in Setting~\ref{mainSetting}, and suppose the $(c,C,p,n)$-boundedness conditions hold. We denote by $X$, $\chi_\one$ and $\chi_\mu$ the random set and the splits of the lemma statement, but do not reveal them just yet.

    By using the Chernoff bound from Theorem~\ref{thm:chernoff} and by doing a union bound, we get that with probability at least \(1-2Lp^{-1}\exp\big(-\tfrac{1}{300}pn\big)\) each part of the \(\mu\)-split of \([n]\) has size at least \(\frac{9n}{10L}\) and at most \(\frac{11n}{10L}\), and each part of the \(\mathbf{1}\)-split has size at most \(\tfrac{2pn}{L}\). Reveal the \(\mu\)- and \(\mathbf{1}\)-splits \(\chi_\mu\) and \(\chi_\mathbf{1}\) and suppose that this likely event occurs.
Without revealing \(X\), we know from Fact~\ref{fact:split} that every vertex \(\phi\) of \(\Phi'\) has a revealed part \(Y(\phi)\) which is the union of some at most \(d(k-1)\) parts of the \(\mu\)-split. We can therefore prove the lemma by a union bound over the possible choices of at most \(d(k-1)\) parts of the \(\mu\)-split of~\([n]\).

Hence, let \(Q\) be a set of at most \(d(k-1)\) parts in \([L]\), and let \(Y\) be the union of the parts of the \(\mu\)-split that are mapped to \(Q\).
For $q\in Q$, the expected number of elements of \(X\) with \(\mu\)-part \(q\) is between $p\cdot \frac{9n}{10L}$ and $p\cdot\frac{11n}{10L}$.
Thus, by the Chernoff bound from Theorem~\ref{thm:chernoff} and a union bound, we obtain that with probability
\begin{equation}\label{eq:finbound:p1}
  p_1\le |Q|\exp\Big(-\frac{1}{12}\cdot p\cdot \frac{9n}{10L}\Big)\le dk\exp\big(-(20L)^{-1}pn\big)
\end{equation}
there is a \(q\in Q\) such that the number of elements of \(X\) with \(\mu\)-part \(q\) is not at most \(\tfrac{2pn}{L}\).

We condition now on this unlikely event $\cB$ not occurring, and reveal \(X\cap Y\). Define \(P_Q:=\{\mu_j:j\in Q\}\cup\{\nu_1,\dots{},\nu_{\lceil Lp^{-1}\rceil}\}\) and let \(\Psi(Q)\) be the set of $P_Q$-extreme test functions.
We next want to bound the number~$N_{\Psi(Q)}$ of possible choices of functions $\phi$ which form a product of at most \(d\) elements of \(\Psi(Q)\).
For this, we start by bounding \(|\Psi(Q)|\) as follows. First consider that by definition every vertex in \(\Psi(Q)\) is either in \(\Sigma\), or of the form \(\ast_{i,\omega}(f_1,\dots{},f_k)\) where \(f_1,\dots{},f_k\) are \(P_Q\)-extreme. By definition, \(f_j\) is  \(P_Q\)-extreme if there is \(h_j\in P_Q\) such that for every \(x\in [n]\) we have either \(f_j(x)=0\) or \(f_j(x)=h_j(x)\). Therefore, to upper bound \(|\Psi(Q)|\) we observe that every \(P_Q\)-extreme test function is either in \(\Sigma\) or of the form \(\ast_{i,\omega}(f_1,\dots{},f_k)\) obtained as follows. We first select \(\omega\in\Omega\) and a sequence of \(k-1\) bounding functions \(h_1,\dots{},h_k\) from \(P_Q\); we then choose for each bounding function \(h_j\), from the at most \(\tfrac{2pn}{L}\) non-zero entries, the non-zero entries of \(f_j\) (which by definition of `extreme' are equal to the corresponding entries of \(h_j\)). Hence, assuming $\cB$ does not hold, we have
 \[|\Psi(Q)|\le|\Sigma|+|\Omega|(2Lp^{-1})^{k-1}\cdot 2^{2pn(k-1)/L}\le
 \exp(C^{-1}pn)+\exp(C^{-1}pn)\exp(3L^{-1}kpn)\le\exp(4L^{-1}kpn)
 \,,\]
 where we used $(c,C,p,n)$-boundedness and $L\le C$.
 The number of possible choices of functions $\phi$ that are a
 product of at most \(d\) elements of \(\Psi(Q)\) then is
 \[N_{\Psi(Q)}\le d\exp(4dL^{-1}kpn)\le\exp(5L^{-1}kpn)\,.\]

 We now would like to bound
 $\big\langle\mu-\one,\phi^\mathrm{small}\cdot\mathbbm{1}([n]\setminus Y)\big\rangle$
 by using Bernstein's inequality.
 By definition, the entries of \(\phi^{\mathrm{small}}\cdot\mathbbm{1}\big([n]\setminus Y\big)\) are in \([-2c^d,2c^d]\), and only the entries outside \(Y\) can be non-zero. Observe further that $\big\langle\mu-\one,\phi^\mathrm{small}\cdot\mathbbm{1}([n]\setminus Y)\big\rangle$
 is a sum of independent random variables $Z_x$ with $x\in[n]\setminus Y$ such that
 \[Z_x = \frac{1}{n}\big(\mathbbm{1}(x\in X)p^{-1}-1\big)c_x\,,\]
 where $c_x=\phi^{\mathrm{small}}(x)$ if $x\in[n]\setminus Y$ and $c_x=0$ otherwise, which implies \(c_x\in[-2c^d,2c^d]\).  Since the probability of \(x\in X\) is \(p\), all $Z_x$ have mean zero, and are bounded between \(-2c^dn^{-1}\) and \(2c^dn^{-1}p^{-1}\). We have
 \[\mathrm{Var}(\mathbbm{1}(x\in X)p^{-1}-1)=p(p^{-1}-1)^2+(1-p)(-1)^2=p^{-1}-1\le p^{-1}\,,\]
 so that the variance of each~$Z_x$ is at most \(4c^{2d}n^{-2}p^{-1}\).
 Hence, by Bernstein's inequality (Lemma~\ref{lem:bernstein}), the probability that when we reveal \(X\setminus Y\) we get
 $\big|\big\langle\mu-\one,\phi^\mathrm{small}\cdot \mathbbm{1}\big([n]\setminus Y\big)\big\rangle\big|\ge\delta$
 is
 \[p_2\le 2\exp\Big(-\frac{\delta^2/2}{\frac{1}{3}\cdot2p^{-1}n^{-1}\delta+n\cdot 4c^{2d}n^{-2}p^{-1}}\Big)\le \exp\Big(-\frac1{16c^{2d}}\delta^2pn\Big)\,.\]
   
 Taking a union bound over the choices of \(\phi\) and combining this with~\eqref{eq:finbound:p1}, we obtain that the probability that there exists any product \(\phi\) of at most \(d\) elements of \(\Psi(Q)\) with 
 $\big|\langle\mu-\one,\phi^\mathrm{small}\cdot \mathbbm{1}\big([n]\setminus Y\big)\rangle\big|\ge\delta$
 is at most
 \begin{equation*}\begin{split}
     p_1 + N_{\Psi(Q)}\cdot p_2 \le
     dk\exp\big(-(20L)^{-1}pn\big) + \exp(5L^{-1}kpn) \cdot \exp\Big(-\frac1{30c^{2d}}\delta^2pn\Big)
     \le\exp\Big(-\frac1{40c^{2d}}\delta^2pn\Big)
     \,,
 \end{split}\end{equation*}
 by our choice of~$L_0$.
 Finally, taking a union bound over the at most $2^L$ choices of \(Q\), the probability that there exists \(Q\) and a product \(\phi\) of at most \(d\) elements of \(\Psi(Q)\) such that $\big|\langle\mu-\one,\phi^\mathrm{small}\cdot \mathbbm{1}\big([n]\setminus Y\big)\rangle\big|\ge\delta$
 is at most
 \[2^L\cdot\exp\Big(-\frac1{40c^{2d}}\delta^2pn\Big) \le \exp\Big(-\frac1{50c^{2d}}\delta^2pn\Big)\,.\]

 Suppose now that \(X\) is such that this unlikely event~$\cE$ does not occur. Given any \(\tilde{X}\subset X\), the polytope \(\Phi'\) is then determined. Let \(\phi\) be a vertex of this polytope. From Lemma~\ref{lem:vtxprod} and Fact~\ref{Rmk:anti_uniform_functionals}\ref{itm:Phi:vertices} we know that~\(\phi\) is a product of at most \(d\) extreme $\{\tmu_1,\dots,\tmu_L,\nu_1,\dots,\nu_{\lceil Lp^{-1}\rceil}\}$-test functions. Let $Q_\phi\subset[L]$ be as in Definition~\ref{notationForSplitting}, that is, $Q_\phi$ is a set of size at most $d(k+1)$ such that $\phi$ is a product of~$d$ functions that are
 \(\{\tilde{\mu}_j:j\in Q_\phi\}\cup\{\nu_1,\dots{},\nu_{\lceil Lp^{-1}\rceil}\}\)-test functions. Hence, $\phi$ is a product of at most \(d\) members of \(\Psi(Q_\phi)\), because for each \(j\) the function \(\tilde{\mu}_j\) is pointwise either equal to \(\mu_j\) or equal to zero.
Since~$\cE$ does not occur,
the conclusion of the lemma follows.
\end{proof}

We now prove the split polytope anti-correlation lemma, Lemma~\ref{lem:anticor}. The idea here is straightforward: we assume the likely events of the moment bound lemma, Lemma~\ref{lem:deletionbounds}, and of Lemma~\ref{lem:finbound} hold. In addition we will assume $|X|$ is not too large: the Chernoff bound from Theorem~\ref{thm:chernoff} tells us this is likely. Our inner product $\langle \tmu-\one,\phi\rangle$ is maximised at a vertex of $\Phi'$, so we assume $\phi$ is a vertex. We split the inner product up into the part dealt with by the likely event of Lemma~\ref{lem:finbound} and all the other parts; for the other parts, we always use linearity and the triangle inequality to split further into inner products of the form $\langle f,\phi'\rangle$ where $f$ is non-negative and $\phi'$ is related (in some way) to $\phi$. This will put us in a position to use moment bounds to show all these inner products are small.

\begin{proof}[Proof of Lemma~\ref{lem:anticor}]
Given \(d,d',\eps>0\) and $c\ge 1$, we set \(\delta=\frac{\eps}{12c^d}\) and  \(d''=\max\big(d',d\big(1+2\lceil\log_2\frac{2c^d}{\delta}\rceil\big)\big)\). Let \(L_0\ge 1600c^{2d}dk\delta^{-2}\) be sufficiently large for Lemma~\ref{lem:deletionbounds} with input $k,c,\delta,d''$ and Lemma~\ref{lem:finbound} with input $k,d,c,\delta$. Given $L\ge L_0$, let $C\ge 1000c^{2d}d\delta^{-2}$ be sufficiently large for Lemmas~\ref{lem:finbound} and~\ref{lem:deletionbounds} with the given inputs.
Given sufficiently large $n$ and $p\in(0,1]$, let us be in Setting~\ref{mainSetting} and suppose that the $(c,C,p,n)$-boundedness conditions hold.

  The Chernoff bound from Theorem~\ref{thm:chernoff} tells us that with probability at least \(1-\exp\big(-\frac13pn\big)\) over the choice of \(X=[n]_p\)
  the event $\cE_1$ that~$X$
  has at most \(2pn\) elements occurs. Moreover, Lemma~\ref{lem:finbound} with input~\(\delta\) tells us that with probability at least \(1-\exp\big(-\frac1{100c^{2d}}\delta^2pn\big)\) over the choices of \([n]_p\) and the \(\mu\)- and \(\mathbf{1}\)-splits, for each vertex \(\phi\in\Phi'\) we have
\begin{equation}\label{eq:finbound}
  \big|\langle\mu-\mathbf{1},\phi^\mathrm{small}\cdot \mathbbm{1}\big([n]\setminus Y(\phi)\big)\rangle\big|<\delta\,.
\end{equation}
In particular, with probability at least \(1-\exp\big(-\frac{1}{200c^{2d}}\delta^2pn\big)\) over the choice of~\([n]_p\) the following event $\cE_2$ holds.
With probability at least $0.9$ over the choices of the \(\mu\)- and \(\mathbf{1}\)-splits~\eqref{eq:finbound} holds for each vertex \(\phi\in\Phi'\).
Finally,
Lemma~\ref{lem:deletionbounds}, with the given inputs, tells us that with probability at least \(1-\exp\big(-\frac1{12}\delta^2pn\big)\) over the choice of \(X=[n]_p\), the set \(X\) has the following property~$\cE_3$. There exists a \(\delta\)-deletion \(\tilde{X}\) of \(X\) such that with probability at least \(0.9\) over the choices of the \(\mu\)- and \(\mathbf{1}\)-splits
 for any \(1\le\ell\le d''\), any \(\phi\in\Phi(\{\tilde{\mu}_1,\dots{},\tilde{\mu}_L,\nu_1,\dots{},\nu_{\lceil Lp^{-1}\rceil}\})^\ell\), and any $j\in[L]$ and $j'\in[\lceil Lp^{-1}\rceil]$ we have
\begin{equation}\label{eq:finprobmombound} \langle\tilde{\mu}_j,\phi\rangle,\langle\nu_{j'},\phi\rangle\le 2 c^\ell\,.
\end{equation}

Now, reveal~$X$ and suppose that $\cE_1$, $\cE_2$, and $\cE_3$ hold, which by the union bound has probability at least
\[1-\exp\Big(-\frac13pn\Big)-\exp\Big(-\frac1{200c^{2d}}\delta^2pn\Big)-\exp\Big(-\frac{1}{12}\delta^2pn\Big)\ge 1-\exp\big(-pn/C\big)\,.\]
Next, fix a \(\delta\)-deletion~\(\tilde{X}\) witnessing that~$\cE_3$ holds. The probability that the \(\mu\)- and \(\mathbf{1}\)-splits have the properties specified in~$\cE_2$ and~$\cE_3$ is at least \(0.8\). Hence, we can fix \(\mu\)- and \(\mathbf{1}\)-splits \(\chi_\mu\) and \(\chi_\mathbf{1}\) such that this holds. In the remainder of this proof we will show that these splits satisfy the conclusions of our lemma.

Let \(\tilde{P}=\{\tilde{\mu}_1,\dots{},\tilde{\mu}_L,\nu_1,\dots{},\nu_{\lceil Lp^{-1}\rceil}\}\), and let $\Phi'=\Phi(\tilde P)^d$ be as in Definition~\ref{notationForSplitting}.
We first establish the desired anti-correlation for $\tmu-\one$. We use the following claim.

\begin{claim} For each \(\phi\in\Phi'\) and each \(j\in[L]\), we have
\begin{align}
\label{eq:tmupsi} \langle\tilde{\mu}_j,\phi\rangle&\le 2c^d\,,\\
\label{eq:tmupsibg} \langle\tilde{\mu},|\phi^\mathrm{big}|\rangle&\le \delta\,,\\
\label{eq:onepsibg} \langle\mathbf{1},|\phi^\mathrm{big}|\rangle&\le \delta\,.
\end{align}
\end{claim}
\begin{claimproof}
    Equation~\eqref{eq:tmupsi} is immediate from~\eqref{eq:finprobmombound}, taking $\ell=d$.
    For the remaining two equations, let \(\ell=2\lceil\log_2\tfrac{2c^d}{\delta}\rceil+1\), and note \((1+\ell)d\le d''\). Given $\phi\in\Phi'$, by
    Fact~\ref{fac:d-product}
    we have $\phi^{1+\ell}\in\Phi(\tilde{P})^{d(1+\ell)}$. Note that $1+\ell$ is even, so that $\phi^{1+\ell}$ has non-negative entries. By definition, if for some $x$ we have $\phi^\bg(x)\neq0$, then $|\phi^\bg(x)|>2c^d$. In particular, we have pointwise
    \begin{equation*}
    |\phi^\bg|<(2c^d)^{-\ell}(\phi^\bg)^{1+\ell}\le (2c^d)^{-\ell}\phi^{1+\ell}\,.
    \end{equation*}
    Since $\tmu$ is non-negative, we conclude that
    \[\langle\tmu,|\phi^\bg|\rangle\le (2c^d)^{-\ell}\langle\tmu,\phi^{1+\ell}\rangle=(2c^d)^{-\ell}\frac{1}{L}\sum_{i=1}^L\langle\tmu_i,\phi^{1+\ell}\rangle\le (2c^d)^{-\ell}\cdot 2c^{d(1+\ell)}=2^{1-\ell}c^d\,,\]
    where the final inequality uses~\eqref{eq:finprobmombound} to bound each of the $L$ inner products by $2c^{d(1+\ell)}$. By choice of $\ell$, we have $2^{1-\ell}c^d\le\delta$ as required for~\eqref{eq:tmupsibg}.

    The proof of~\eqref{eq:onepsibg} is almost identical with $\one$ replacing $\tmu$, the only difference being that we split $\one=\frac{1}{\lceil Lp^{-1}\rceil}(\nu_1+\dots+\nu_{\lceil Lp^{-1}\rceil})$; we obtain the same bound.
\end{claimproof} 

Since \(\max_{\phi\in \Phi'}\big|\langle\tilde{\mu}-\mathbf{1},\phi\rangle\big|\) is attained at a vertex of \(\Phi'\), let us fix any vertex~$\phi$ of this polytope and let \(Y=Y(\phi)\), using the notation from Definition~\ref{notationForSplitting}. Our goal is to show that \(\vass{\langle\tilde{\mu}-\mathbf{1}, \phi\rangle}<\eps\). We use linearity of the inner product and the triangle inequality to split this up. Write \(\tilde{\mu}'=\tilde{\mu}\mathbbm{1}([n]\setminus Y)\) and \(\tilde{\mu}''=\tilde{\mu}\mathbbm{1}(Y)\); define \(\mu'\), \(\mu''\) and \(\mathbf{1}'\) and \(\mathbf{1}''\) analogously. We obtain
\[\big|\langle\tilde{\mu}-\mathbf{1}, \phi\rangle\big|\le\big|\langle\tilde{\mu}'-\mathbf{1}', \phi\rangle\big|+\big|\langle\tilde{\mu}''-\mathbf{1}'', \phi\rangle\big|\,.\]
We further split the first term
\[\big|\langle\tilde{\mu}'-\mathbf{1}', \phi\rangle\big|\le\big|\langle\mu'-\mathbf{1}', \phi^\mathrm{small}\rangle\big|+\big|\langle\mu'-\tilde{\mu}', \phi^\mathrm{small}\rangle\big|+\big|\langle\tilde{\mu}'-\mathbf{1}', \phi^\mathrm{big}\rangle\big|\,.\]
The first of these terms is bounded by \(\delta\) by~\eqref{eq:finbound}. Since \(\tilde{\mu}\) and \(\mu\) differ in at most \(\delta p n\) places, \(\mu'-\tilde{\mu}'\) is equal to \(p^{-1}\) in at most \(\delta pn\) places and otherwise equal to zero, while \(|\phi^\mathrm{small}|\) is bounded by \(2c^d\), so the second term is at most \(\tfrac{1}{n}\cdot p^{-1}\cdot\delta p n\cdot 2c^d= 2c^d\delta\). Splitting the third term, we get
\[\big|\langle\tilde{\mu}'-\mathbf{1}', \phi^\mathrm{big}\rangle\big|\le\langle\tilde{\mu}', |\phi^\mathrm{big}|\rangle+\langle\mathbf{1}', |\phi^\mathrm{big}|\rangle\le\langle\tilde{\mu}, |\phi^\mathrm{big}|\rangle+\langle\mathbf{1}, |\phi^\mathrm{big}|\rangle\le 2\delta\,,\]
where we used non-negativity of $\tmu,\tmu'$ and $\one,\one'$, and~\eqref{eq:tmupsibg} and~\eqref{eq:onepsibg}, to justify the inequalities. Putting the parts together we have
\[|\langle\tmu'-\one',\phi\rangle|\le \delta+2c^d\delta+2\delta\le 5c^d\delta\,.\]
We next split
\[\big|\langle\tilde{\mu}''-\mathbf{1}'', \phi\rangle\big|\le\langle\tilde{\mu}'', |\phi|\rangle+\langle\mathbf{1}'', |\phi|\rangle\,.\]
Since $\Phi'$ is centrally symmetric and vertices of $\Phi'$ are either non-negative or non-positive, $|\phi|$ is also in $\Phi'$.
By definition we have $\tmu''=\tfrac{1}{L}\sum_{j\in Q_\phi}\tmu_j$. Applying~\eqref{eq:tmupsi}, this gives
\[\langle\tmu'',|\phi|\rangle=\tfrac{1}{L}\sum_{j\in Q_\phi}\langle\tmu_j,|\phi|\rangle\le \frac{d(k-1)}{L}\cdot 2c^d\,.\]
Finally, we turn to \(\langle\mathbf{1}'',|\phi|\rangle\). Here we split \(\phi=\phi^\mathrm{small}+\phi^\mathrm{big}\), and write
\[\langle\mathbf{1}'',|\phi|\rangle=\langle\mathbf{1}'',|\phi^\mathrm{small}|\rangle+\langle\mathbf{1}'',|\phi^\mathrm{big}|\rangle\le\langle\mathbf{1}'',|\phi^\mathrm{small}|\rangle+\langle\mathbf{1},|\phi^\mathrm{big}|\rangle\,.\]
To deal with the first term of this, observe that \(\mathbf{1}''\) takes the value \(1\) in at most \(\tfrac{2d(k-1)n}{L}\) places, and zero elsewhere, while \(\phi^\mathrm{small}\) is bounded by \(2c^d\), so that the first inner product is at most \(\frac{1}{n}\cdot\frac{2d(k-1)n}{L}\cdot 2c^d=\frac{4c^dd(k-1)}{L}\). The second inner product is by~\eqref{eq:onepsibg} at most \(\delta\).
Putting the pieces together, we have
\[\big|\langle\tilde{\mu}''-\mathbf{1}'', \phi\rangle\big|\le  \frac{d(k-1)}{L}\cdot 2c^d + \frac{4c^dd(k-1)}{L} +\delta=\frac{6c^dd(k-1)}{L} +\delta\,, \]
and hence
\[\big|\langle\tilde{\mu}-\mathbf{1},\phi\rangle\big|\le 5c^d\delta+\frac{6c^dd(k-1)}{L} +\delta\le\eps\,,\]
as required for the anti-correlation statement.

\medskip

We now turn to proving the desired moment bounds. By Lemma~\ref{lem:splitreduce} (applied with~$\ell$ instead of~$d$) we have
\(\Phi(\tilde{\mu},\mathbf{1})^\ell\subseteq\Phi(\tilde{P})^\ell\) for each $1\le\ell\le d'$, so it suffices to prove the required moment bounds hold for all \(\phi\) in the latter polytope. Given $\phi\in\Phi(\tilde{P})^\ell$, since $d''\ge\ell$, by~\eqref{eq:finprobmombound} we have for each $j\in[L]$ and $j'\in[\lceil Lp^{-1}\rceil]$ that
\[\langle\tmu_j,\phi\rangle,\langle\nu_{j'},\phi\rangle\le 2c^\ell\,.\]
Since $\tmu=\tfrac{1}{L}\sum_{j\in[L]}\tmu_j$, we get
\[\langle\tmu,\phi\rangle=\tfrac{1}{L}\sum_{j\in[L]}\langle\tmu_j,\phi\rangle\le 2c^\ell\,.\]
Since the same statement also holds for $-\phi$ instead of $\phi$, this gives the required moment bound for $\tmu$. The same argument, with $\one=\tfrac{1}{\lceil Lp^{-1}\rceil}\sum_{j\in[\lceil Lp^{-1}\rceil]}\nu_j$, gives the required moment bound for $\one$.
\end{proof}

\section{Proof of the moment bounds lemma without deletion}\label{sec:momentbounds}

In this section, we prove the `weaker' moment bound estimate, Lemma~\ref{cor:momentboundcor}.
We shall use the Kim--Vu inequality~\cite{KimVu}, which is a concentration result for a multi-variable polynomial over independent Bernoulli random variables. We state a slightly less general form of the result from~\cite{KimVu}, which is precisely what we need. Let \(F\) be a hypergraph with \(V(F)=\{1,\dots{},n\}\) and edge set \(E(F)\). We say that~$F$ is a \emph{sub-$d$-uniform $w$-weighted} hypergraph if each edge \(e\) of~$F$ has associated weight \(w(e)>0\) and contains at most \(d\) vertices. Moreover, for any \(A\subseteq V(F)\), let \(F_A\) denote the \emph{\(A\)-truncated sub-hypergraph} of \(F\), which is the hypergraph with vertex set \(V(F)\setminus A\) and edge set \(E(F_A)=\{ e'\subseteq V(F_A): e'\cup A\in E(F) \}\). Note that \(w\) extends in a unique way from \(E(F)\) to \(E(F_A)\), therefore we abuse notation and use \(w\) to denote the weight function of either hypergraph.

Suppose now that \(t_1,\dots{},t_n\) are independent random variables, such that for each \(i\in [n]\) there is \(p_i\in [0,1]\) such that \(t_i\) is a \(\{0,1\}\) random variable with \(\mathbb{E}(t_i) = p_i\). We consider the following random variable, which is a polynomial:
\[Y_F = \sum_{e\in E(F)}w(e)\prod_{t_i\in e}t_i\,.
\]
where by convention \(\prod_{t_i\in \emptyset}t_i=1\). Observe that any
polynomial in independent Bernoulli random variables can be written in this
form, and this implicitly defines the weighted hypergraph~$F$. In our
application we shall not make this~$F$ explicit, but instead argue merely using
the polynomial. In this language, $d$ is the maximum number of random variables
that appear in a monomial of the polynomial, and~$n$ is the total number of
random variables used.

For \(i\in \{0, \dots{}, d\}\) let
\[  Y_{F_A} = \sum_{e\in E(F_A)}w(e)\prod_{t_j\in e}t_j\,, \qquad\text{and}\qquad
    \mathbb{E}_i(Y_F)=\max_{A\subset V(F)\colon |A|=i}\mathbb{E}(Y_{F_A})\,.\]
We call these random variables~$Y_{F_A}$ \emph{truncated variables}.
Note that \(\mathbb{E}_0(Y_F)= \mathbb{E}(Y_F)\). Let
\[\mathbb{E}'(Y_F)= \max_{i\in\{0,\dots,d\}}\mathbb{E}_i(Y_F) \qquad\text{and}\qquad \mathbb{E}''(Y_F)= \max_{i\in[d]}\mathbb{E}_i(Y_F)\,.\]
The following version of the Kim--Vu concentration inequality~\cite{KimVu} says that if
$\EE''(Y_F)$ is much smaller than $\EE(Y_F)$, then $Y_F$ is concentrated about
its expectation.

\begin{theorem}[Kim--Vu inequality]\label{thm:Kim-Vu}
  Let \(F\) be a sub-$d$-uniform $w$-weighted hypergraph, and \(\{t_1,\dots,t_n\}\) be independent $\{0,1\}$-random variables.
  For \(\lambda > 1\) and \(a_d:=8^dd!^{1/2}\), we have
  \[P\big[|Y_F-\mathbb{E}(Y_F)|> a_d(\mathbb{E}'(Y_F)\mathbb{E}''(Y_F))^{1/2}\lambda^d\big]= O\Big(\exp\big(-\lambda+(d-1) \log n\big)\Big)\, .\]
\end{theorem}

In the proof of Lemma~\ref{cor:momentboundcor} our goal is to bound the inner products $\langle\mu_i,\phi\rangle$ and $\langle\nu_j,\phi\rangle$ for $\phi\in\Phi(\mu_1,\dots,\mu_L,\nu_1,\dots,\nu_{\lceil Lp^{-1}\rceil})$. We shall reduce this to estimating inner products with functions of the following special type, which we call special products. Recall that $\ast_{i_j}(\cdot)$ is a shorthand for $\ast_{i_j,\one}(\cdot)$.

\begin{definition}[\((q,d)\)-special product]
\label{def:special}
Given $k\ge2$, integers $n$ and $d$, and $q\in(0,1]$, let us be in Setting~\ref{mainSetting}. Let $V_0,V_1,\dots,V_{d(k-1)}$ be independent random samples from $[n]_q$, and define for each $i\in\{0,\dots,d(k-1)\}$ the corresponding scaled indicator function $\rho_i(x)=q^{-1}\ind(x\in V_i)$.
A \emph{\((q,d)\)-special product} is any function of the form $\prod_{j=1}^{d}\ast_{i_j}(f^{(j)}_{1},\dots,f^{(j)}_{k})$ with  $i_j\in[k]$ and $f^{(j)}_{\kappa}\in\{\rho_1,\dots,\rho_{d(k-1)}\}$ for each $j\in[d]$ and $\kappa\in[k]$.
\end{definition}

Observe that in this definition, we use $d(k-1)+1$ independent samples, but the special products are not allowed to use the first of these.
This is because we want to use these special products~$\psi$ in inner products $\langle f,\psi \rangle$, where~$f$ is also a random variable that may (or may not) use this first independent sample.

We will see that for any $(q,d)$-special product $\psi$, the inner product $\langle f,\psi\rangle $ for $f\in\{\rho_0,\dots,\rho_{d(k-1)}\}$ is a polynomial $Y$ in Bernoulli random variables whose monomials correspond to books as in the following definition. Recall that an edge~$s$ of an ordered hypergraph~$S$ is a sequence, which we denote by $s=s_1,\dots,s_k$.

\begin{definition}[book]\label{def:config}
Let~$S$ be a $k$-uniform ordered hypergraph. Let \(d\) be a positive integer, \(x\in [n]\) and \(i_1,\dots{},i_d\in [k]\). A \emph{book} with spine \(x\) and index tuple \(\bfi=(i_1,\dots{},i_d)\) is a  tuple \((s^{(1)},\dots{},s^{(d)})\) of edges of \(S\) such that \(s^{(j)}_{i_j}=x\) for each $j\in[d]$. We call \(t=|\bigcup_i s^{(i)}\setminus\{x\}|\) the \emph{size} of the book.
%
For a positive integer \(t\), and a tuple \(\mathbf{i}=(i_1,\dots{},i_d)\) we
denote by \(\beta(x,\bfi,t)\) the number of books \((s^{(1)},\dots{},s^{(d)})\)
in~$S$ with spine \(x\), index tuple~$\bfi$ and size~$t$.  If we are
given a set $W\subset[n]$, then we write $\beta(W,t)$ for the number of books
in~$S$ (with any spine) which use all vertices of $W$ and contain $t\ge0$ vertices
outside $W$.
\end{definition}

The following book-counting lemma will be used to (accurately) estimate the expectation of a $(q,d)$-special product and its truncated versions.

\begin{lemma}\label{countingBooks}
  Given a $k$-uniform ordered hypergraph $S$ on $[n]$, suppose $c\ge1$ and
  $C,q>0$ are such that $\Delta_\kappa(S)\le cC^{1-\kappa}q^{\kappa-1}e(S)n^{-1}$
  holds for all $\kappa\in[k]$. Let \(x\in [n]\) and let \(d\) be an integer.
  \begin{enumerate}[label=\abc]
  \item\label{lem:books:a} If \(\mathbf{i}=(i_1,\dots{},i_d)\) is any tuple, given any integer $k-1\le t<d(k-1)$ we have
    \[\beta(x,\bfi,t)\le
    (2k)^{3kd} C^{-1}c^d q^{d(k-1)-t}e(S)^dn^{-d}\,, \qquad\text{and}\qquad
    \beta\big(x,\bfi,d(k-1)\big)\leq c^de(S)^dn^{-d}\,.
    \]
  \item\label{lem:books:b} If $\emptyset\neq W\subset[n]$, for any integer $0\le t\le d(k-1)-|W|$ we have
    \[\beta(W,t)\le (2dk)^{3kd^2}C^{-1}c^d q^{d(k-1)-t} e(S)^d n^{-d}\,.\]
  \end{enumerate}
\end{lemma}
\begin{proof}
  We first prove~\ref{lem:books:a}.
  Let us fix \(x\), \(t\), \(d\) and \(\mathbf{i}\). We now describe a process
  that can generate any book with spine \(x\), index tuple
  \((i_1,\dots{},i_d)\), and covering \(t\) vertices besides \(x\). By counting
  the number of choices we make in this process, we can upper bound
  \(\beta(x,\mathbf{i}, t)\).  We start by choosing an edge \(s^{(1)}\) of
  \(S\) whose \(i_1\)-th element is \(x\) and set $m_1=k-1$. We then pick a
  subset~$Y^{(1)}$ including~$x$ among the \(k\) elements of \(s^{(1)}\). Let
  $m_2=k-|Y^{(1)}|$. The elements of~$Y^{(1)}$ will also be contained in
  \(s^{(2)}\): we fix a position of these elements by selecting an injection
  from $Y^{(1)}$ to \([k]\), making sure that \(x\) is assigned position
  \(i_2\), and then select an element \(s^{(2)}\) of \(S\) that contains the
  elements of~$Y^{(1)}$ in these positions. We repeat a similar procedure,
  fixing a set~$Y^{(2)}$ of \(k-m_3\) elements of \(s^{(2)}\) to generate
  \(s^{(3)}\) with \(x\) in position \(i_3\), and so on, until we get
  \(s^{(d)}\). We further assume that the total number $\sum_{i\in[d]}m_i$ of
  vertices these edges cover apart from~$x$ is~$t$.

  In this procedure, we have $\Delta_1(S)$ ways of choosing $s^{(1)}$ containing~$x$, at most $2^{k}$
  ways to choose each set $Y^{(i)}$ once $s^{(i)}$ is chosen, at most $k!$ ways
  to choose each injection from $Y^{(i)}$ to \([k]\), and at most \(\Delta_{k-m_i}(S)\)
  ways to choose each edge $s^{(i+1)}$ once $Y^{(i)}$ is chosen.
  No matter which numbers for $m_2,\dots,m_d$ we happen to have generated, we have
  \begin{equation}\begin{split}\label{eq:books}
    \Delta_1(S)\cdot\prod_{2\le i\le d}\Delta_{k-m_i}(S)
    &\le ce(S)n^{-1} \cdot \prod_{2\le i\le d} \big(cC^{1-k+m_i}q^{k-m_i-1}e(S)n^{-1}\big) \\
    &=c^dC^{d(1-k)+t}q^{d(k-1)-t}e(S)^dn^{-d}\,,
  \end{split}\end{equation}
  because $\sum_{2\le i\le d}m_i=t-(k-1)$. Further, there are at most $k^d$ ways to generate numbers $m_2,\dots,m_d$.
  We conclude that for $t<d(k-1)$ we have
  \begin{equation*}\begin{split}
      \beta(x,\mathbf{i}, t)
      &\le 2^{kd}\cdot (k!)^d\cdot c^dC^{d(1-k)+t}q^{d(k-1)-t}e(S)^dn^{-d}\cdot k^d
      \le (2k)^{3kd} c^d C^{-1} q^{d(k-1)-t}e(S)^dn^{-d}
  \,.
  \end{split}
  \end{equation*}    
  For \(t=d(k-1)\) we necessarily have $Y^{(i)}=\{x\}$ for all~$i$ and hence
  \(m_1=\dots{}=m_d=k-1\). Therefore, using~\eqref{eq:books} we conclude
  \(\beta(x,\mathbf{i}, t)\le c^de(S)^dn^{-d}\). 

  \medskip
  
  We now prove~\ref{lem:books:b}.  To estimate $\beta(W,t)$, we describe how
  to construct a book which uses the vertices $W$ and some $t$ additional
  vertices, and bound the total number of choices we have in this construction.
  We shall pick the edges $s^{(1)},\dots,s^{(d)}$ one by one. But before doing
  so, we pick sets $Q_1,\dots,Q_d\subset W$ and then insist in this process that
  for each $j\in[d]$, the set $Q_j$ is exactly $W\cap
  s^{(j)}\setminus\bigcup_{i\in[j-1]}s^{(i)}$.  If $q_j$ denotes $|Q_j|$, then
  we need to have $q_1+\dots+q_d=|W|$. Since $W\neq\emptyset$, we also know that
  at least one of the values $q_j$ is positive. By possibly permuting the edges
  in the book we are constructing, at the expense of at most an extra factor
  of~$d$ in our count, we can assume that $q_1\ge 1$.
  Also, for each $j\ge2$, since~$Q_j$ will be exactly
  $W\cap s^{(j)}\setminus\bigcup_{i\in[j-1]}s^{(i)}$, the set~$Q_j$ does not
  contain the spine (denoted by $x$) of the book we are constructing. The number of choices for
  $Q_1,\dots,Q_d$ is at most $d^{|W|}$. In addition we choose non-negative
  integers $r_2,\dots,r_d$ such that $d(k-1)+1-|W|-r_2-\dots-r_d=t$. We shall
  insist in our process that for $j\ge 2$ the edge $s^{(j)}$ uses exactly~$r_j$
  vertices from $\bigcup_{i\in[j-1]}s^{(i)}\setminus (W\cup \{x\})$.  There are at most
  $d^{d(k-1)+1-|W|}$ ways to choose the $r_2,\dots,r_d$.

  We can pick any edge of $S$ which contains $Q_1$ as $s^{(1)}$. There are at
  most $k!$ ways to choose in which entries of $s^{(1)}$ the vertices of
  $Q_1$ will be contained in which order. Hence, there are at most $k!\Delta_{q_1}(S)$ choices for
  $s^{(1)}$. We pick from $s^{(1)}$ a spine $x$ for our book, giving a further $k$
  choices.
  Now for each $2\le\ell\le d$ in order, we pick $s^{(\ell)}$ as follows. We
  must use the elements $Q_\ell$ and the spine $x$ (which is not in $Q_\ell$), and
  we pick a further $r_\ell$ elements in $\bigcup_{i\in[j-1]}s^{(i)}\setminus W$, which we
  can do in at most $2^{d(k-1)}$ ways. We choose which entries of $s^{(\ell)}$
  these picked elements will be in at most $k!$ ways. The remaining elements
  must be chosen outside $\bigcup_{i\in[j-1]}s^{(i)}\cup W$, which we can do in at
  most $\Delta_{q_\ell+1+r_\ell}(S)$ ways.
  In this construction, the number of elements of $[n]$ we choose in total outside $W$ is
  \[k-q_1+\sum_{j=2}^d k-q_j-1-r_j=t\,,\]
  as required. Counting choices, we get
  \[\beta(W,t)\le d\cdot d^{|W|}d^{d(k-1)+1-|W|}\cdot k!\Delta_{q_1}(S)\cdot k\cdot \prod_{\ell=2}^d \Big(2^{d(k-1)}k!\Delta_{q_\ell+1+r_\ell}(S)\Big)\,.\]
  Since the $(c,C,q,n)$-boundedness conditions hold, we get
  \begin{align*}
    \beta(W,t)&\le d^{|W|+1}d^{d(k-1)+1-|W|}k!c(qC^{-1})^{q_1-1}e(S)n^{-1}k\prod_{\ell=2}^d\Big( 2^{d(k-1)}k!c(qC^{-1})^{q_\ell+r_\ell}e(S)n^{-1} \Big)\\
    &=d^{d(k-1)+2}k!^dc^d(e(S)n^{-1})^dk2^{d^2(k-1)}(qC^{-1})^{q_1+\dots+q_d+r_2+\dots+r_d-1}\\
    &=d^{d(k-1)+2}k!^dc^d(e(S)n^{-1})^dk2^{d^2(k-1)}(qC^{-1})^{d(k-1)-t} \\
    &\le (2dk)^{3kd^2}C^{-1}c^d q^{d(k-1)-t} e(S)^d n^{-d}\,. \qedhere
  \end{align*}
\end{proof}

We now use this lemma to prove concentration for inner products with $(q,d)$-special products.

\begin{lemma}\label{lem:momentboundlemma}
Given $k\ge2$ and $c\ge 1$, for any $d\ge1$ there exists $C$ such that the following holds for all $\alpha>0$, all sufficiently large $n$ and any $q\in(0,1]$ such that $q\ge n^{-1}\log^{2dk+1}n$. Let $S$ be a $k$-uniform ordered hypergraph on $[n]$, and suppose that $\Delta_\kappa(S)\le cC^{1-\kappa}q^{\kappa-1}e(S)/n$ for each $1\le\kappa\le k$.

We denote by $V_0,V_1,\dots,V_{d(k-1)}$ independent random samples from $[n]_q$. Using the notation from Definition~\ref{def:special}, define the random variables $f\in\{\rho_0,\dots,\rho_{d(k-1)}\}$ and~$\psi$, where~$\psi$ is a $(q,d)$-special product. With probability at least $1-n^{-\alpha}$, we have $\langle f , \psi \rangle \le \tfrac32c^{d}$.
\end{lemma}
\begin{proof}
  Given $k,c,d$, let $C=10 dk\cdot(2k)^{3kd}c^d$ and let~$\alpha$ be given.
  Given $n$ sufficiently large and $q\in(0,1]$ with $q\ge n^{-1}\log^{2dk+1}n$, let $S$ be as in the lemma statement. 
We let $V_0,\dots,V_{d(k-1)}$ be independent samples from $[n]_q$, and $\rho_0,\dots,\rho_{d(k-1)}$ be the corresponding scaled indicator functions.

Let $\psi=\prod_{j=1}^d\ast_{i_j}(f_1^{(j)},\dots{},f_k^{(j)})$ for some \(i_1,\dots{},i_d\in [k]\) and some \(f^{(j)}_{\kappa}\in\{\rho_1,\dots,\rho_{d(k-1)}\}\) for all $j\in[d]$ and $\kappa\in[k]$. Writing the convolutions out explicitly, we have
    \[
    \psi(x)=\Big(\frac{n}{e(S)}\Big)^d\prod_{j=1}^d\,\,\sum_{s^{(j)}\in S_{i_j}(x)}\,\,\prod_{\kappa\neq i_j} f^{(j)}_{\kappa}(s^{(j)}_\kappa)\,,
    \]
    where we recall that $S_i(x)$ is the set of hyperedges of~$S$ whose $i$-th entry is~$x$.
    Letting \[\cS=\big\{(s^{(1)},\dots,s^{(d)})\colon s^{(j)}\in S_{i_j}(x) \text{ for each $j\in[d]$} \big\}\,,\]
    we can rewrite this as
    \[
    \psi(x)=\Big(\frac{n}{e(S)}\Big)^d
    \sum_{(s^{(1)},\dots,s^{(d)})\in\cS}\,\,
    \prod_{j=1}^d\,\,\prod_{\kappa\neq i_j} f^{(j)}_{\kappa}(s^{(j)}_\kappa)\,.
    \]    
    It is helpful to refer to each of the terms in this sum individually. To this end we define functions $g^{(j)}_\kappa:=qf^{(j)}_\kappa$ (for later convenience), and
    \[
    \hat{\psi}(x; s^{(1)},\dots{},s^{(d)}):= \prod_{j=1}^d\,\,\prod_{\kappa\neq i_j} g^{(j)}_{\kappa}(s^{(j)}_\kappa)\,,
    \]
    with which we have
    \[\psi(x) = \Big(\frac{n}{e(S)}\Big)^dq^{-d(k-1)} \sum_{(s^{(1)},\dots,s^{(d)})\in\cS}\,\,\hat{\psi}(x; s^{(1)},\dots{},s^{(d)})\,. \]

    Given $f\in\{\rho_0,\dots,\rho_{d(k-1)}\}$, our goal is to show that with high probability \(\langle f , \psi \rangle \le \frac{3}{2}c^{d}\).
    Again, we define $g:=qf$, which allows us to write
    \[ Y:= \langle f , \psi \rangle = \frac{1}{n}\bigg(\frac{n}{e(S)}\bigg)^dq^{-d(k-1)-1}\sum_{x\in [n]}\,\,\sum_{(s^{(1)},\dots,s^{(d)})\in\cS}g(x)\hat{\psi}(x; s^{(1)},\dots,s^{(d)})\,.\]
    This is a random variable which is a polynomial in the variables
    $q\rho_\ell(y)$, where $0\le\ell\le d(k-1)$ and $y\in[n]$. These are
    \[\tilde n:=n(d(k-1)+1)\]
    independent Bernoulli random variables which take value $1$
    with probability $q$. In each term of this polynomial, at most
    \[\tilde d=d(k-1)+1\]
    different random variables appear.  It is convenient to drop the scaling
    factor and write
    \begin{equation}\label{eq:moment:Y}
    \hat{Y}:= \sum_{x\in [n]}\,\,\sum_{(s^{(1)},\dots,s^{(d)})\in\cS}g(x)\hat{\psi}(x; s^{(1)},\cdots{},s^{(d)})\,,
    \quad\text{giving}\quad Y=\frac{1}{n}\bigg(\frac{n}{e(S)}\bigg)^dq^{-d(k-1)-1}\hat Y    
    \,.
    \end{equation}

    Our plan now is to first calculate the expectation of~$Y$ and then use
    the Kim--Vu inequality (Theorem~\ref{thm:Kim-Vu}) to prove concentration.
    Observe that each term $g(x)\hat{\psi}(x; s^{(1)},\dots,s^{(d)})$ of
    \(\hat{Y}\) corresponds to a tuple \((s^{(1)},\dots{},s^{(d)})\) of \(d\)
    hyperedges of \(S\) such that for each $j\in[d]$ the $i_j$-th vertex
    of~$s^{(j)}$ is~$x$.  In other words, $(s^{(1)},\dots{},s^{(d)})$ is a book
    with spine~$x$ and index tuple $\bfi=(i_1,\dots,i_d)$. Observe further
    that for such a book with size $t=|\bigcup_i s^{(i)}\setminus\{x\}|$, we
    have $k-1\le t\le d(k-1)$, because apart from~$x$ each hyperedge $s^{(j)}$
    contains $k-1$ distinct vertices. In addition, the number of distinct random
    variables $q\rho_\ell(y)$ used in the term $g(x)\hat{\psi}(x;
    s^{(1)},\dots,s^{(d)})$ is at least $t+1$.  With the help of
    Lemma~\ref{countingBooks}\ref{lem:books:a} we can now bound the number $\beta(x,\bfi,t)$ of
    books with spine~$x$, index tuple~$\bfi$, and size~$t$ (recalling that the
    case $t=d(k-1)$ has a different bound to the other cases) to conclude that
    \begin{equation*}\begin{split}
      \EE(\hat{Y}) &\le\sum_{x\in[n]}\sum_{t=k-1}^{d(k-1)}q^{t+1}\beta(x,\bfi,t) \\
      &\le n\bigg(q^{d(k-1)+1}c^de(S)^dn^{-d}+\sum_{t=k-1}^{d(k-1)-1}
      (2k)^{3kd} C^{-1}c^d q^{d(k-1)-t}e(S)^dn^{-d}\cdot q^{t+1}\bigg)\,.
    \end{split}\end{equation*}
    Plugging this into~\eqref{eq:moment:Y}, we obtain
    \begin{equation}\label{eq:momentbound:expY}
      \EE(Y)\le c^d+\sum_{t=k-1}^{d(k-1)-1} (2k)^{3kd} C^{-1}c^d
      \le c^d+ dk\cdot(2k)^{3kd} C^{-1}c^d
      \le \frac54c^d\,,
    \end{equation}
    where the final inequality is by choice of $C$.

    Our next aim, following the notation of Theorem~\ref{thm:Kim-Vu}, is to
    calculate $\EE_i(Y)$ for each $i\ge 1$, where
    $\EE_i(Y)=\max_{|A|=i}\EE(Y_A)$ such that $A$ is a set of our Bernoulli
    random variables and the random variable~$Y_A$ is obtained from $Y$ by
    removing all terms which do not contain all variables from $A$, and removing
    the variables in~$A$ from the product of Bernoulli random variables in the
    remaining terms, with the usual convention that an empty product evaluates
    to $1$.  Again, dropping the scaling factor, we write
    $\hat{Y}_A:=\big(\frac{1}{n}(\frac{n}{e(S)})^dq^{-d(k-1)-1}\big)^{-1}Y_A$.
    
    Consider any term of $\hat{Y}_A$ and let $g(x)\hat{\psi}(x;
    s^{(1)},\dots,s^{(d)})$ be the corresponding term of $\hat Y$.  The term of
    $\hat{Y}_A$ contains at most $d(k-1)+1-|A|$ distinct random variables.  The
    corresponding term of $\hat Y$ must use all random variables
    $q\rho_\ell(y)\in A$. Hence $W\subset\bigcup_i s^{(i)}$ for
    \[W:=\big\{y\in[n]\colon q\rho_\ell(y)\in A\text{ for some }0\le\ell\le
    d(k-1)\big\}\,.\] The elements of $\bigcup_i s^{(i)}\setminus W$, on the other
    hand, correspond to distinct random variables not in $A$. Therefore, the
    term $g(x)\hat{\psi}(x; s^{(1)},\dots,s^{(d)})$ of $\hat Y$ corresponds to a
    book with spine~$x$, containing all elements in~$W$, and~$t$ vertices
    outside~$W$ with $t\le d(k-1)+1-|A|$. Since a term with at least~$t$
    distinct random variables has expectation at most $q^t$, writing
    $\beta(W,t)$ for the number of books on $[n]$ (with any spine)  which use all
    vertices of $W$ and contain $t$ vertices outside $W$, we can bound
    \begin{equation*}
      \EE(\hat{Y}_A)\le\sum_{t=0}^{d(k-1)+1-|A|}\beta(W,t)q^t\,.
    \end{equation*}
    Using Lemma~\ref{countingBooks}\ref{lem:books:b}, we get
    \begin{align*}
      \EE(\hat{Y}_A)&
      \le \sum_{t=0}^{d(k-1)+1-|A|} (2dk)^{3kd^2}C^{-1}c^d q^{d(k-1)-t} e(S)^d n^{-d}\cdot q^t \\
      &\le (2dk)^{4kd^2} c^d e(S)^d n^{-d} q^{d(k-1)}
    \end{align*}
    where the final inequality uses $C\ge1$. It follows that
    \[\EE(Y_A)\le \frac{1}{n}\Big(\frac{n}{e(S)}\Big)^dq^{-d(k-1)-1} \cdot (2dk)^{4kd^2} c^d e(S)^d n^{-d} q^{d(k-1)}
    = (2dk)^{4kd^2} c^d (nq)^{-1}
    \,.\]
    
    Observe that this quantity does not depend on $A$ (we only used that $A$ is
    non-empty and hence $W$ is non-empty), and by our assumption on $q$, it
    tends to zero as $n$ tends to infinity. In particular, using the notation of the Kim--Vu inequality (Theorem~\ref{thm:Kim-Vu}), for all sufficiently
    large $n$, since $c\ge1$ and by~\eqref{eq:momentbound:expY} we have
    \[\EE(Y)=\EE'(Y)\le \frac54c^d \qquad\text{and}\qquad 
    \EE''(Y)\le (2dk)^{4kd^2} c^d (nq)^{-1}\,.\]
    By Theorem~\ref{thm:Kim-Vu},
    applied with $\tilde n=\big(d(k-1)+1\big)n$ instead of~$n$ and $\tilde d=d(k-1)+1$ instead of~$d$, we obtain for
    $\lambda=(\tilde d-1+2\alpha)\log n$ that    
    \[P\Big[|Y-\mathbb{E}(Y)|> 8^{\tilde d}\tilde d!^{1/2}(\mathbb{E}'(Y)\mathbb{E}''(Y))^{1/2}\lambda^{\tilde d}\Big]= O\Big(\exp\big(-\lambda+(\tilde d-1)\log(\tilde n)\big)\Big)=O\big(n^{-2\alpha}\big)\le n^{-\alpha}\,,\]
    where the last inequality is for sufficiently large~$n$.
    To complete the proof, we observe
    \begin{multline*}
    \EE(Y)+8^{\tilde d}(\tilde d)!^{1/2}\big(\EE'(Y)\EE''(Y)\big)^{1/2}\lambda^{\tilde d} \\
    \le 
    \frac54c^d + (8dk)^{dk}\Big(\frac54c^d\cdot (2dk)^{4kd^2} c^d (nq)^{-1}\Big)^{1/2}\big((dk+2\alpha)\log n\big)^{dk}
    \le\frac32c^d
    \end{multline*}
    since $n$ is sufficiently large and since $nq\ge\log^{2dk+1} n$ by our assumption on $q$.
\end{proof}

We are now ready to prove our moment bound lemma, Lemma~\ref{cor:momentboundcor}. We need to prove that with high probability we have $\langle\mu_j,\phi\rangle,\langle\nu_{j'},\phi\rangle\le 2c^\ell$ for all $\phi\in\Phi(\mu_1,\dots,\mu_L,\nu_1,\dots,\nu_{\lceil Lp^{-1}\rceil})^\ell$. The idea for this is as follows. As discussed in Section~\ref{sec:moment-overview}, it is enough to prove the statement only for $\phi$ a largest vertex. There are only polynomially many such, and Lemma~\ref{lem:momentboundlemma} almost gives us the required probability to take a union bound over largest vertices. The gap in the above idea is that the \(\mu_j\) are not independent; and similarly the \(\nu_{j'}\) are not independent, so we cannot apply Lemma~\ref{lem:momentboundlemma} directly. However, any given largest vertex $\phi$ depends on only at most $d(k-1)$ of the functions $\mu_1,\dots,\mu_L,\nu_1,\dots,\nu_{\lceil Lp^{-1}\rceil}$, and importantly this is much less than $L$. We use a coupling with slightly larger independent functions, to which we can apply Lemma~\ref{lem:momentboundlemma}; the factor $\tfrac32$ in the conclusion of Lemma~\ref{lem:momentboundlemma} leaves room for the slight loss that this coupling produces.

\begin{proof}[Proof of Lemma~\ref{cor:momentboundcor}]
  Given $k,c,d$, let \(\alpha=10dk\) and $L_0=100(dk)^4$. Given $L\ge L_0$,
  choose $C$ such that $C/(2L)$ is sufficiently large for
  Lemma~\ref{lem:momentboundlemma} with input $k,c,d$, and suppose in addition
  $C\ge10dk$. Suppose $n$ is sufficiently large and $p\in(0,1]$. Let us be in
  Setting~\ref{mainSetting} and suppose the $(c,C,p,n)$-boundedness conditions
  hold.
  Choose \(0<q<1\) such that
  \begin{equation}\label{eq:moments:q}
    (1-q)^{d(k-1)+1}=1-\big(d(k-1)+1\big)\tfrac{p}{L}\,,
  \end{equation}    
  and \(0<q'<1\) such that
  \[(1-q')^{d(k-1)+1}=1-\big(d(k-1)+1\big)\tfrac{1}{\lceil Lp^{-1}\rceil}\,.\]
  Since $\tfrac{1}{\lceil Lp^{-1}\rceil}\le\tfrac{p}{L}$ we have
  \begin{equation}\label{eq:moments:smaller}
    q'\le q\,. 
  \end{equation}
  Further, note that $\frac{p}{2L}\le q\le \frac{p}{L}$.
  Moreover, by the choice of~$L_0$,
  \begin{equation}\label{eq:moments:two}
    \frac32(\lceil Lp^{-1}\rceil q)^{d(k-1)+1}\le 2\,.
  \end{equation}
  We conclude from $\frac{p}{2L}\le q$, the choice of~$C$ and since the $(c,C,p,n)$-boundedness conditions hold that
  for~$n$ sufficiently large we have $q\ge(\log^{2dk+1}n)n^{-1}$  and that for each $1\le\kappa\le k$
  \[\Delta_\kappa(S)\le cC^{1-\kappa}p^{\kappa-1}\tfrac{e(S)}{n}\le c(C/2L)^{1-\kappa}q^{\kappa-1}\tfrac{e(S)}{n}\,.\]
  This verifies that the conditions of Lemma~\ref{lem:momentboundlemma} are satisfied for~$q$.

     Let \(P=\{\mu_1,\dots{},\mu_L,\nu_1,\dots{},\nu_{\lceil Lp^{-1}\rceil}\}\). These functions are random variables, whose value depends on $X=[n]_p$ and the uniform random choices of $\chi_\mu$ and $\chi_\one$.
Now, choose $\ell\in[d]$ and $f\in P$, and consider the optimisation problem $\max_{\phi\in\Phi(P)^\ell}\langle f,\phi\rangle$.
By Lemma~\ref{lem:max-product}, the maximum is achieved at a function~$\psi$ that is a product of $P$-largest test functions. It follows that either $\psi=\one$ or
$\psi$ is of the form $\psi=\prod_{t=1}^\ell\ast_{i_t}(f_{1,t},\dots,f_{k,t})$, where $i_1,\dots{},i_\ell\in[k]$ and $f_{j,j'}\in P$ for all $j\in[k]\setminus\{i_{j'}\}$ and $j'\in[\ell]$. Hence, in the former case, which we call Case~1, we are interested in a bound on $\langle\one,\mu_i\rangle$ and $\langle\one,\nu_i\rangle$, while in the latter case, which we call Case~2, we are interested in bounding $\langle \prod_{t=1}^\ell\ast_{i_t}(f_{1,t},\dots,f_{k,t}),\mu_i \rangle$ and $\langle \prod_{t=1}^\ell\ast_{i_t}(f_{1,t},\dots,f_{k,t}),\nu_i \rangle$.
Observe that there are at most
\[N^*=d\cdot k^d\cdot(3Lp^{-1})^{kd+1}\]
ways to choose~$\ell$, $f$, and such $\psi$.

To deal with Case~1, we use the Chernoff bound from Theorem~\ref{thm:chernoff} and a union bound to get that with probability at least $1-n^{-1}$ the following event~$\cE_1$ holds: for each $i\in[L]$ the set $X_i$ has size at most $\tfrac32pn/L$, and for each $j\in[\lceil Lp^{-1}\rceil]$ the set $\chi_\one^{-1}(j)$ has size at most $\tfrac32pn/L$. If this event occurs, we get $\langle\one,\mu_i\rangle,\langle\one,\nu_{j}\rangle\le 2$.

To deal with Case~2, more work is needed. Let us fix one choice for $\psi=\prod_{t=1}^\ell\ast_{i_t}(f_{1,t},\dots,f_{k,t})$ and $f\in P$ as described above.
This~$\psi$ and~$f$ are composed of various functions $\mu_t$ and $\nu_{t'}$. Let $T_0\subset[L]$ be the set of indices~$t$ of the functions~$\mu_t$ we are using, and $T'_0\subset[\lceil Lp^{-1}\rceil]$ be the set of indices of the functions~$\nu_{t'}$ we are using. We have $|T_0|+|T_0'|\le \ell(k-1)+1\le d(k-1)+1$. We now choose arbitrary supersets $T_1,T$ of~$T_0$ with $T_0\subset T_1\subset T\subset[L]$, and
arbitrary supersets $T'_1,T'$ of~$T'_0$ with $T'_0\subset T'_1\subset T'\subset[\lceil Lp^{-1}\rceil]$, such that
$|T|= d(k-1)+1$, $|T'|= d(k-1)+1$, and $|T_1|+|T'_1|=\ell(k-1)+1$.
We will use the supersets~$T$ and~$T'$ for the coupling we describe shortly, and the supersets~$T_1$ and~$T'_1$ for the application of Lemma~\ref{lem:momentboundlemma}.

Consider the following random experiment. For each \(t\in T\), we first generate an independent binomial random set \(Z_t=[n]_q\).
We then obtain disjoint sets \(Z'_t\) for \(t\in T\) from this as follows. For each \(x\in\bigcup_{t\in T}Z_t\) independently, pick \(t\) uniformly at random among all the indices~$i$ such that $Z_i$ contains~$x$, and add~$x$ to $Z'_t$.
By definition of \(q\) in~\eqref{eq:moments:q}, for a given \(x\in[n]\) the probability that \(x\in \bigcup_{t\in T}Z_t\) is \(\tfrac{|T|p}{L}\), and the events \(x\in Z'_t\) are disjoint over \(t\in T\). Since \(x\) is equally likely to appear in any given \(Z'_t\) for \(t\in T\), we conclude that the probability of \(x\in Z'_t\) is \(\tfrac{p}{L}\). Observe that the events that \(x\in X=[n]_p\) and \(\chi_\mu(x)=t\), which are also disjoint events over \(t\in T\), also  have probability \(\tfrac{p}{L}\). It follows that the distribution of \((Z'_t)_{t\in T}\) is the same as the distribution of $(X_t)_{t\in T}$, where we recall that $X_t=\{x\in X\colon\chi_\mu(x)=t\}$.
So we can consider the coupling in which the sets $(X_t)_{t\in T}$ are generated according to the random experiment we performed to obtain $(Z'_t)_{t\in T}$.
Let \(\hat{\mu}_t(x)=q^{-1}\mathbbm{1}(x\in Z_t)\), where it is important (for applying Lemma~\ref{lem:momentboundlemma} below) that we use~$Z_t$ instead of~$Z'_t$. In this coupling, because $\mu_t(x)=Lp^{-1}\mathbbm{1}(x\in X_t)$, and $x\in Z'_t$ only if $x\in Z_t$, we have
\begin{equation}\label{eq:momentcor:mu}
  0\le\mu_t(x)\le Lp^{-1}q\hat{\mu}_t(x)\le\lceil Lp^{-1}\rceil q\hat{\mu}_t(x)\,.
\end{equation}

We now perform a similar, independent, random experiment. For each \(t'\in T'\), we generate independently \(W_{t'}=[n]_q\) where \(q\) is as defined above.
We generate \(W''_{t'}\) by sampling the elements of \(W_{t'}\) independently with probability \(\frac{q'}{q}\), which is at most~$1$ by~\eqref{eq:moments:smaller}.
Note that the \(W''_{t'}\) are independent copies of \([n]_{q'}\). Finally,
we again generate disjoint sets \(W'_{t'}\) with \(t'\in T'\) from this as follows. For each \(x\in\bigcup_{t'\in T'}W''_{t'}\) independently, pick \(t'\) uniformly at random among all the indices~$i$ such that $W''_i$ contains~$x$, and add~$x$ to $W'_{t'}$.
As before, the distribution of \((W'_{t'})_{t'\in T'}\) is identical to the distribution of \((\chi^{-1}_\one(t'))_{t'\in T'}\) and we consider the coupling in which the sets \((\chi^{-1}_\one(t'))_{t'\in T'}\) are generated by the random experiment described to generate \((W'_{t'})_{t'\in T'}\). Let \(\hat{\nu}_{t'}(x)=q^{-1}\mathbbm{1}(x\in W_{t'})\). In this coupling, because $\nu_{t'}(x)=\lceil Lp^{-1}\rceil \mathbbm{1}\big(x\in \chi^{-1}_\one(t')\big)$, and $x\in\chi^{-1}_\one(t')$ only if $x\in W_{t'}$, we have
\begin{equation}\label{eq:momentcor:nu}
  0\le\nu_{t'}(x)\le \lceil Lp^{-1}\rceil q\hat{\nu}_{t'}(x)\,.
\end{equation}

Let \(\hat{\psi}\) denote the function obtained from~$\psi$ by replacing each \(\mu_t\) with \(\hat{\mu}_t\) for \(t\in T\), and each \(\nu_t\) with \(\hat{\nu}_t\) for \(t\in T'\). Then~\eqref{eq:momentcor:mu} and~\eqref{eq:momentcor:nu} imply
\[\langle\mu_i,\psi\rangle\le\big(\lceil Lp^{-1}\rceil q\big)^{d(k-1)+1}\langle\hat{\mu}_i,\hat{\psi}\rangle\qquad\text{and}\qquad
\langle\nu_j,\psi\rangle\le\big(\lceil Lp^{-1}\rceil q\big)^{d(k-1)+1}\langle\hat{\nu}_j,\hat{\psi}\rangle\,.\]
Now, \(\hat{\psi}\) is a \((q,\ell)\)-special product defined over the $\ell(k-1)+1$ independent samples of $[n]_q$ given by $(Z_t)_{t\in T_1}$ and $(W_{t'})_{t'\in T'_1}$. So by Lemma~\ref{lem:momentboundlemma} with input $k,c,\ell,\alpha,q$
we get that for fixed $i\in T_1$ and $j\in T'_1$ the events
\[\langle\hat{\mu}_i,\hat{\psi}\rangle >\tfrac32c^{\ell}\qquad \text{and}\qquad\langle\hat{\nu}_j,\hat{\psi}\rangle > \tfrac32c^\ell\]
each have probability at most \(\tfrac{1}{n^{\alpha}}\). Since \(\tfrac32(\lceil Lp^{-1}\rceil q)^{d(k-1)+1}\le 2\) by~\eqref{eq:moments:two}, the same probability bounds hold on the events
\[\langle{\mu}_i,{\psi}\rangle>2c^{\ell}\qquad\text{and}\qquad\langle{\nu}_j,{\psi}\rangle>2c^\ell\,.\]
Let~$\cE_2$ be the event that these bad events do not hold for any~$i$, $j$, $\psi$, and $\ell$.
Taking the union bound over the~$N^*$ possible choices for~$i$, $j$, $\psi$, and $\ell$, we see that the probability of~$\cE_2$ not holding is at most
\[n^{-\alpha}\cdot d\cdot k^d\cdot(3Lp^{-1})^{kd+1}\,,\]
which tends to~$0$ as~$n$ tends to infinity as $p^{-1}\le n$ and by choice of~$\alpha$.

We conclude that the probability that~$\cE_1$ and~$\cE_2$ hold is $1-o(1)$. In this case, the conclusion of our lemma holds.
\end{proof}

\section{Proof of the moment bounds lemma with deletion}\label{sec:delete}

In this section we prove Lemma~\ref{lem:deletionbounds}, which states that a deletion \(\tilde{X}\) satisfying the stated moment bounds exists with exponentially small failure probability. This follows from the Harris inequality~\cite{Harris1960}.  Recall that a subset \(\mathcal{D}\) of the power set \(\mathcal{P}\big([n]\big)\) is called \emph{decreasing} if whenever \(S'\subseteq S\in \mathcal{D}\) we have \(S'\in\mathcal{D}\), and \emph{increasing} if the same statement holds with \(\subseteq\) replaced by \(\supseteq\).

\begin{theorem}[Harris inequality]\label{thm:harris} For any \(p\in[0,1]\) and \(n\), let \(\mathcal{A}\) and \(\mathcal{B}\) be two subsets of \(\mathcal{P}\big([n]\big)\) which are both decreasing. Then
\[\mathbb{P}\big([n]_p\in \mathcal{A}\cap\mathcal{B}\big)\ge\mathbb{P}\big([n]_p\in \mathcal{A}\big)\mathbb{P}\big([n]_p\in \mathcal{B}\big)\,.\]
\end{theorem}

Sp\"ohel, Steger and Warnke~\cite{SSW} deduced the following theorem from the Harris inequality. They state their result for the specific case \([n]=\binom{[N]}{2}\) (i.e.\ for the random graph), but their proof works verbatim in the more general situation. For completeness, we give the details below.

\begin{theorem}[{\cite[Theorem 4]{SSW}}]\label{thm:delete}
 Let \(\mathcal{D}\) be a decreasing subset of \(\mathcal{P}\big([n]\big)\). Given \(\alpha,\delta\in(0,1]\), let \(p\in(0,1]\) be such that \(\mathbb{P}\big([n]_p\in\mathcal{D}\big)\ge\delta\). Then with probability at least \(1-\delta^{-1}\exp\big(-\tfrac13\alpha^2pn\big)\) the random set $X=[n]_p$ has a subset with at least \((1-\alpha)pn\) elements which is in \(\mathcal{D}\).
\end{theorem}
\begin{proof}
 Let \(\mathcal{I}\subset\mathcal{P}([n])\) consist of all sets with at least \((1-\alpha)pn\) elements. Let \(\mathcal{S}\subset\mathcal{P}([n])\) consist of all sets $S$ which have a subset in \(\mathcal{I}\cap\mathcal{D}\). Clearly, $\mathcal S$ is increasing, so \(\overline{\mathcal{S}}=\mathcal{P}([n])\setminus\mathcal{S}\) is decreasing.  By  Theorem~\ref{thm:harris}, we have \(\mathbb{P}\big([n]_p\in\overline{\mathcal{S}}\big)\mathbb{P}\big([n]_p\in\mathcal{D}\big)\le\mathbb{P}\big([n]_p\in\overline{\mathcal{S}}\cap\mathcal{D}\big)\). Rearranging, and observing \(\overline{\mathcal{S}}\cap\mathcal{D}\subseteq\overline{\mathcal{I}}\), we get
 \[\mathbb{P}\big([n]_p\in\overline{\mathcal{S}}\big)\le\frac{\mathbb{P}\big([n]_p\in\overline{\mathcal{S}}\cap\mathcal{D}\big)}{\mathbb{P}\big([n]_p\in\mathcal{D}\big)}\le\delta^{-1}\mathbb{P}\big([n]_p\in\overline{\mathcal{I}}\big)\,.\]
The Chernoff bound from Theorem~\ref{thm:chernoff} applied with~$\alpha$ instead of~$\delta$ now gives \[\mathbb{P}\big([n]_p\in\overline{\mathcal{I}}\big)\le\exp\big(-\tfrac13\alpha^2pn\big)\,,\] which gives the required probability bound.
\end{proof}

We now have the tools we need to prove  Lemma~\ref{lem:deletionbounds}.

\begin{proof}[Proof of Lemma~\ref{lem:deletionbounds}]
Given $k\ge2$, $\delta>0$ and $c\ge1$, and $d''$ a positive integer, let $L_0$ be returned by Lemma~\ref{cor:momentboundcor} for input $k,d'',c$. Given $L\ge L_0$, let $C$ be sufficiently large for Lemma~\ref{cor:momentboundcor}.
Let $n$ be a sufficiently large integer, and $p\in(0,1]$. Let us be in Setting~\ref{mainSetting}, and suppose the $(c,C,p,n)$-boundedness conditions hold.
 
  Let \(\mathcal{D}\subset\mathcal{P}([n])\) be the family of (non-random) sets~$Y$ with the following property. Let \(\mu(y)=p^{-1}\mathbbm{1}(y\in Y)\) be a scaled indicator function of~$Y$. Then uniform random choices of \(\chi_\mu\) and \(\chi_\mathbf{1}\) satisfy with probability at least \(0.9\) that for all \(1\le\ell\le d''\), all $i\in[L]$, all $j\in[\lceil Lp^{-1}\rceil]$, and all $\phi\in \Phi(\mu_1,\dots{},\mu_L,\nu_1,\dots{},\nu_{\lceil Lp^{-1}\rceil})^\ell$ we have
 \[ \langle\mu_i,\phi\rangle\le 2 c^\ell\qquad\text{and}\qquad\langle\nu_{j},\phi\rangle\le 2 c^\ell\,.\]
 Observe that all the left hand sides of these conditions are (for any fixed $\chi_\one,\chi_\mu$) increasing in \(Y\), which means that the left hand sides are not smaller for supersets of~$Y$ than for~$Y$.
This implies that the family \(\mathcal{D}\) is decreasing. Furthermore, Lemma~\ref{cor:momentboundcor} states that \(\mathbb{P}([n]_p\in\mathcal{D})=1-o(1)\ge\tfrac12\).

We now apply  Theorem~\ref{thm:delete} with this \(\mathcal{D}\), with \(\alpha=\tfrac12\delta\), and with \(\mathbb{P}(\mathcal{D})\ge\tfrac12\), to deduce that with probability at least \(1-2\exp\big(-\tfrac1{12}\delta^2pn\big)\) the set $X=[n]_p$ has a subset \(\tilde{X}\) which is in \(\mathcal{D}\) and which has at least \(\big(1-\tfrac12\delta\big)pn\) elements. Additionally, the probability that \([n]_p\) has more than \(\big(1+\tfrac12\delta\big)pn\) elements is by the Chernoff bound from Theorem~\ref{thm:chernoff} at most \(\exp\big(-\tfrac{1}{12}\delta^2pn\big)\). Suppose that \(X\) satisfies both conditions, which occurs with probability at least \(1-3\exp\big(-\tfrac{1}{12}\delta^2pn\big)\) by the union bound. Then \(|X\setminus\tilde{X}|\le\delta p n\) and $\tilde X\in\mathcal{D}$ as required.
\end{proof}

\section{Transference for multiple counts}\label{sec:multitransfer}

The following theorem is an extension of our transference principle (Theorem~\ref{thm:GenSimp}) to a scenario that allows us to count multiple structures. 

\begin{theorem}[transference principle for multiple hypergraphs]\label{thm:mainmulti}
 For every $k,r,t\ge2$, $c\ge1$, and \(\eps,\rho>0\) there exists \(C>0\) such that the following holds. Given a sufficiently large $n$, and $S_1,\dots,S_t$ ordered hypergraphs of uniformities $2\le k_1,\dots,k_t\le k$ on $[n]$, a set $\Sigma$ of structure functions, and for each $i\in[t]$ a set $\Omega_i$ of subcounts on $S_i$, suppose $0<p\le 1$ is such that the following hold for each $i\in[t]$. 
 \begin{enumerate}[label=\rom]
 \item $p\ge n^{-1+\rho}$,
 \item $|\Sigma|,|\Omega_i|\le \exp(C^{-1}pn)$,
 \item\label{itm:mainmulti:Delta} $\Delta_\kappa(S_i)\le cC^{1-\kappa}p^{\kappa-1}e(S_i)n^{-1}$ for each $1\le\kappa\le k_i$, and
\item\label{itm:mainmulti:e} $e(S_i)\ge n(\log n)^{k+k\lceil\rho^{-1}\rceil+2}$.
\end{enumerate}
 Then with probability at least
 \[1-\sum_{i=1}^t\Prob\big(\big|S_i\cap [n]_p^{k_i}\big|\ge(1+\tfrac13\eps)p^{k_i}e(S_i)\big)-\exp\big(-\tfrac{pn}{C}\big)\,,\]
 the binomial random set \(X=[n]_p\) is such that for any partition $Y_1\cup\dots\cup Y_r= X$, there exists a partition $Z_1\cup\dots\cup Z_r$ of $[n]$ which is, for each $i\in[t]$, an $\eps$-good dense model for $Y_1,\dots,Y_r$ with respect to $(S_i,\Sigma,\Omega_i)$.
\end{theorem}

The conditions of this theorem differ slightly from those of Theorem~\ref{thm:GenSimp} in that~$p$ must be polynomially separated from $n^{-1}$, not just polylogarithmically separated, and  that the $S_i$ cannot be too sparse. However, for applications to counting graphs or hypergraphs, these conditions are typically satisfied.

For the proof of Theorem~\ref{thm:mainmulti} we proceed similarly as for the
proof of Theorem~\ref{thm:GenSimp}: The main work here is to prove the
functional deletion version of this statement, which is as follows.

\begin{theorem}[transference for multiple functions, deletion version]\label{thm:mainmultitech}
 For every $k,r,t\ge2$, $c\ge1$, and \(\eps,\rho>0\) there exists \(C>0\) such that the following holds. Given a sufficiently large $n$, and $S_1,\dots,S_t$ ordered hypergraphs of uniformity $k_1,\dots,k_t\le k$ on $[n]$, a set $\Sigma$ of structure functions, and for each $i\in[t]$ a set $\Omega_i$ of subcounts on $S_i$, suppose $0<p\le 1$ is such that each $(S_i,\Sigma,\Omega_i)$ satisfies the $(c,C,p,n)$-boundedness conditions and $p\ge n^{-1+\rho}$. Suppose that $e(S_i)\ge n(\log n)^{k+k\lceil\rho^{-1}\rceil+2}$.
 
 Then with probability at least \(1-\exp\big(-\tfrac{pn}{C}\big)\), the binomial random set \(X=[n]_p\) admits an \(\eps\)-deletion \(\tilde{X}\) such that for any non-negative $f_1,\dots,f_r$ such that $f_1+\dots+f_r=\tmu$, there exist $g_1,\dots,g_r:[n]\to\{0,1\}$ such that $g_1+\dots+g_r=\one$, and such that for each $i\in[t]$ and $j\in[r]$ we have \[\|f_j-g_j\|_{\Phi(\{\tmu,\one\},S_i,\Sigma,\Omega_i)}\le\eps k^{-1}\,.\]
\end{theorem}

The proof of Theorem~\ref{thm:mainmulti} from Theorem~\ref{thm:mainmultitech} is essentially the same as the proof of Theorem~\ref{thm:GenSimp} from Theorem~\ref{thm:maintechint}, and we omit the details.

The first idea for the proof of Theorem~\ref{thm:mainmultitech} is to deduce it from Theorem~\ref{thm:maintechint} by taking $S$ to be the union of the $S_i$ and defining a new collection $\Omega$ of structure functions appropriately. However, this only works if each $S_i$ has the same uniformity and approximately the same number of edges. To begin with, we argue that we can reduce to this setting. Specifically, given $S$ we show how to construct an ordered hypergraph of higher uniformity and a tunable number of edges such that appropriate norm bounds for these higher uniformity hypergraphs imply the desired norm bounds for $S$.

Given $\ell\ge1$ and $1\le m\le n/2$, we define the \emph{$(\ell,m)$-puff} $S^{(\ell,m)}$ of $S$ to be the (random) ordered $(k+\ell)$-uniform hypergraph obtained as follows. For $\ell=1$, for each $s\in S$ independently we pick a uniform random $m$-set $Q(s)$ in $[n]\setminus s$, and let the edges of $S^{(1,m)}$ be of the form $(s,q)$ where $s\in S$ and $q\in Q(s)$. For larger $\ell$, we let recursively $S^{(\ell,m)}=\big(S^{(\ell-1,m)}\big)^{(1,m)}$.

Observe that each $s\in S$ corresponds to exactly $m^\ell$ distinct elements of $S^{(\ell,m)}$. It follows from the definition that, defining $\Omega'$ on $S^{(\ell,m)}$ by including all the functions $\omega'(s')=\omega(s)$ for $\omega\in\Omega$ and each $s'$ obtained from $s$, we have for any $f_1,\dots,f_{k}$ and any $i\in[k]$
\begin{equation}\label{eq:puff}
\ast_{i,S^{(\ell,m)},\omega'}(f_1,\dots,f_k,\one,\dots,\one)=\ast_{i,S,\omega}(f_1,\dots,f_k)\,,
\end{equation}
where $\omega'$ is the subcount obtained from $\omega\in\Omega$.
In particular, this tells us that the following implication holds:
\begin{equation}\label{eq:multimod}
\|f-g\|_{\Phi(\{\tmu,\one\},S^{(\ell,m)},\Sigma,\Omega')}\le\eps\quad\implies\quad\|f-g\|_{\Phi(\{\tmu,\one\},S,\Sigma,\Omega)}\le\eps\,.
\end{equation}

The following lemma states that $S^{(\ell,m)}$ inherits the codegree conditions of~$S$ with high probability.

\begin{lemma}\label{lem:puff}
 Suppose that for some $k\ge2$, $c\ge1$, $C>0$ and $p\in(0,1]$, the $k$-uniform ordered hypergraph $S$ on $[n]$ satisfies $\Delta_i(S)\le cC^{1-i}p^{i-1}e(S)/n$ for each $i\in[k]$. Suppose $e(S)\ge (12(k+\ell)\log n)^{k+\ell} n$. Given $Cp^{-1}\le m\le \tfrac14n$, and $\ell\ge1$, with high probability $S^{(\ell,m)}$ satisfies $\Delta_i\big(S^{(\ell,m)}\big)\le cC^{1-i}p^{i-1}e\big(S^{(\ell,m)}\big)/n$ for each $i\in[k+\ell]$.
\end{lemma}
\begin{proof}
 It suffices to prove the statement for $\ell=1$ and apply this repeatedly for larger $\ell$.
 Given $1\le t\le k+1$, fix a sequence $\mathbf{x}$ with $t$ entries not equal to $\ast$. If $x_{k+1}=\ast$, then $\deg_{S^{(\ell,m)}}(\mathbf{x})=m\cdot\deg_S(\mathbf{x}')$, where $\mathbf{x}'$ is obtained from $\mathbf{x}$ by removing the last entry, which is therefore at most $cC^{1-t}p^{t-1}\tfrac{e(S^{(1,m)})}{n}$ as required.

 Now suppose $x_{k+1}=y\in[n]$. Consider the set $W$ of $s\in S$ which agree with $\mathbf{x}'$ at all positions which do not equal $\ast$. We have $|W|\le cC^{2-t}p^{t-2}\tfrac{e(S)}{n}$ by assumption.

 We first deal with the case $t=k+1$. Since $S$ is non-empty, we have $cC^{1-k}p^{k-1}e(S)/n\ge1$ (otherwise the degree condition would be violated). Since $\mathbf{x}$ contains no $\ast$ entries, at most one element of $S^{(1,m)}$ agrees with $\mathbf{x}$. Since $C^{-1}pm\ge1$ and $e(S^{(1,m)})=m\cdot e(S)$, we have
 \[1\le cC^{1-k}p^{k-1}e(S)/n\le C^{-1}pm\cdot cC^{1-k}p^{k-1}e(S)/n=cC^{-k}p^ke\big(S^{(1,m)}\big)/n\,,\]
 as required. This is the only place where we need the lower bound on $m$.
 
 Suppose now $t\le k$. Each member $s$ of $W$, independently, chooses $y\in Q(s)$ with some probability at most $\tfrac{2m}{n}$. Thus the expected number of $s\in W$ such that $y\in Q(s)$ is at most
 \[\tfrac{2m}{n}|W|\le\tfrac{2m}{n}\cdot cC^{2-t}p^{t-2}\tfrac{e(S)}{n}\le \tfrac12cC^{1-t}p^{t-1}\tfrac{e(S^{(1,m)})}{n}\,,\]
where we used $n\ge 4Cp^{-1}$ for the final inequality. By the Chernoff bound from Theorem~\ref{thm:chernoff}, the probability that more than $cC^{1-t}p^{t-1}\tfrac{e(S^{(1,m)})}{n}$ members $s$ of $W$ have $y\in Q(s)$ is at most $\exp\big(-\tfrac1{6}cC^{1-t}p^{t-1}\tfrac{e(S^{(1,m)})}{n}\big)$. Since $cC^{-k}p^k\tfrac{e(S^{(1,m)})}{n}\ge1$, this exponent is larger than $\tfrac16\big(\tfrac{e(S)}{n}\big)^{1/k}$. Taking the union bound over the at most $n^{k+1}$ choices of $\mathbf{x}$, we see that with high probability $\Delta_t\big(S^{(1,m)}\big)\le cC^{1-t}p^{t-1}e\big(S^{(1,m)}\big)/n$.
\end{proof}

With this lemma in place, we can prove Theorem~\ref{thm:mainmultitech}. In this proof we puff up each $S_i$, assuming that each time the good event of Lemma~\ref{lem:puff} occurs, to obtain ordered hypergraphs $S'_1,\dots,S'_t$ of the same uniformity and approximately the same number of edges, and then let $S$ be their union.

\begin{proof}[Proof of Theorem~\ref{thm:mainmultitech}]
Given $k,r,t,\eps,c,\rho$ as in the theorem statement, set $c'=ct$ and $\eps'=\tfrac1{2t}\eps$. Let $C$ be such that $\tfrac12C$ is sufficiently large for Theorem~\ref{thm:maintechint} with input $c'$ and $\eps'$, and for uniformity $k+k\lceil\rho^{-1}\rceil+1$.

Given $n$ sufficiently large and $S_1,\dots,S_t$ ordered hypergraphs, each of uniformity at most $k$, on $[n]$ with at least $n(\log n)^{k+k\lceil\rho^{-1}\rceil+2}$ members, let $\Sigma$ be a set of structure functions and $\Omega_i$ a set of subcounts on $S_i$ for each $i\in[t]$. Suppose that for some $p\in(0,1]$ each $(S_i,\Sigma,\Omega_i)$ satisfies the $(c,C,p,n)$-boundedness conditions separately, and that $p\ge n^{-1+\rho}$.
  As before, we can safely add $\one$  and $x\to\ind(x=j)$ for each $j\in[n]$ to $\Sigma$,  and $\one$ to $\Omega_i$ for each $i\in[t]$.

Recall that $k_i$ is the uniformity of $S_i$, and by assumption $k_i\le k$. Observe that each $S_i$ has between $n$ (by assumption) and $n^{k_i}$ members. We let, for each $i\in[t]$, the number $\ell_i$ be such that $k_i+\ell_i=\ell:=k+k\lceil\rho^{-1}\rceil+1$. We claim that we can choose $Cp^{-1}\le m_i\le\tfrac14n$ for each $i\in[t]$ such that the $(\ell_i,m_i)$-puffs $S_i^{(\ell_i,m_i)}$ all have the same number of edges up to a factor of at most $2$. To see that this is true, observe that for each $i$ we have
\[e\big(S_i^{(\ell_i,m_i)}\big)=m_i^{\ell_i}e(S_i)\,,\]
and we argue that it is possible to choose each $m_i$ such that this number is within a factor $1.1$ of $n\cdot (\frac18n)^{\ell-k}$. The factor $1.1$ part is obtained by simply solving the equation $x^{\ell_i}e(S_i)=n\cdot (\frac18n)^{\ell-k}$ and rounding the result to an integer to obtain $m_i$, which (since $x\ge Cn^{1-\rho}$ is sufficiently large) alters the value by a factor smaller than $1.1$. So we simply need to argue that the solution to this equation satisfies $Cp^{-1}\le x\le \frac14n$. We do this by checking that, whatever the value of $e(S_i)$, plugging in $x=Cp^{-1}$ gives $x^{\ell_i}e(S_i)\le n\cdot (\frac18n)^{\ell-k}$ and plugging in $x=\frac14n$ gives $x^{\ell_i}e(S_i)\ge n\cdot (\frac18n)^{\ell-k}$. Clearly, the worst case is when $e(S_i)\le n^{k_i}$ is as large as possible for the first inequality, and when $e(S_i)\ge n$ is as small as possible for the second. Checking these values, we have
\begin{align*}
  (Cp^{-1})^{\ell_i} n^{k_i}&= (Cp^{-1})^{\ell-k_i}n^{k_i}\le C^\ell n^{(1-\rho)(\ell-k_i)+k_i}=C^\ell n^{\ell-\rho(k+k\lceil\rho^{-1}\rceil+1-k_i)}\le n\cdot \Big(\frac18n\Big)^{\ell-k}\,, \\
  \Big(\frac14n\Big)^{\ell_i} n^{k_i} &\ge \Big(\frac14n\Big)^{\ell-k_i}n \ge n\cdot \Big(\frac18n\Big)^{\ell-k}\,,
\end{align*}
as desired.
The first inequality uses $p^{-1}\le n^{1-\rho}$ and $C\ge1$, the second that $k_i\le k$ and $n$ is sufficiently large, and the final one that $k_i\le k$. Thus the claimed values for the $m_i$ exist.

With high probability, by Lemma~\ref{lem:puff}, for each $i$ we have $\Delta_j\big(S_i^{(\ell_i,m_i)}\big)\le cC^{1-j}p^{j-1}e\big(S_i^{(\ell_i,m_i)}\big)/n$ for each $j\in[\ell]$. Let for each $i$ the hypergraph $S'_i$ be a realisation of $S_i^{(\ell_i,m_i)}$ such that this likely event holds.

Given $\omega\in\Omega_i$, we let $\omega'(s')=\omega(s)$ for each $s'\in S'_i$ whose first terms are $s$, where $s\in S_i$, and let $\Omega'_i$ be the set of all such $\omega'$. Then $\Omega'_i$ is a set of subcounts on $S'_i$ of size $|\Omega_i|$, so $\big(S_i^{(\ell_i,m_i)},\Sigma,\Omega'_i\big)$ satisfies the $(c,C,p,n)$-boundedness conditions.
Finally, we let $S=S_1\cup\dots\cup S_t$, and define $\Omega=\Omega'_1\cup\dots\cup\Omega'_t$, where we extend the domain of each $\omega\in\Omega'_i$ from $S'_i$ to $S$ by letting it take value zero on each element of $S$ which is not in $S'_i$. We now have an $\ell$-uniform ordered hypergraph on $[n]$ and a set of subcounts defined on it.

For a given $\mathbf{x}$, observe that the number of edges of $S$ which agree with $\mathbf{x}$ at all non-$\ast$ places is at most the sum of the corresponding numbers for the $S'_i$, while the total number of edges of $S$ is at least the maximum of the $e(S'_i)$, which is at least a $\tfrac{1}{t}$-fraction of the sum. The set $\Sigma$ has size at most $\exp(pn/C)$, and $\Omega$ has size at most $t\exp(pn/C)\le\exp(2pn/C)$ since $n$ is sufficiently large and by choice of $p$. Thus $(S,\Sigma,\Omega)$ satisfies the $(tc,\tfrac12C,p,n)$-boundedness conditions.

Suppose the likely event of Theorem~\ref{thm:maintechint}, with parameters $c',\eps',\ell$ as above, occurs for a given $X=[n]_p$. For each $i\in[t]$ and $\phi\in\Phi(\{\tmu,\one\},S_i^{(\ell_i,m_i)},\Sigma,\Omega'_i)$, we have the containment $\tfrac{e(S'_i)}{e(S)}\phi\in\Phi(\{\tmu,\one\},S,\Sigma,\Omega)$. To see that this occurs, it is enough to check the statement holds for the vertices. For elements of $\Sigma$, the statement holds since the constant zero function is in the polytope and since $e(S'_i)\le e(S)$. For the remaining vertices, this holds since the only difference between $\ast_{j,S'_i,\omega}(f_1,\dots,f_\ell)$ and $\ast_{j,S,\omega}(f_1,\dots,f_\ell)$, where $\omega$ in the second convolution is considered to take value zero on elements of $S$ not in $S'_i$, is that the scaling factor is changed by the given fraction. Since $e(S'_i)\ge\tfrac{1}{2t}e(S)$ by construction, we have
\[\|f-g\|_{\Phi(\{\tmu,\one\},S,\Sigma,\Omega)}\le\tfrac{\eps'}{k}\implies\|f-g\|_{\Phi(\{\tmu,\one\},S'_i,\Sigma,\Omega'_i)}\le\tfrac{\eps}{k}\implies \|f-g\|_{\Phi(\{\tmu,\one\},S_i,\Sigma,\Omega_i)}\le\tfrac{\eps}{k}\]
where the second implication is~\eqref{eq:multimod}.
\end{proof}

\section{Proofs of the applications}\label{sec:applproofs}

The remainder of the paper is dedicated to applications of our transference
principle. In this section, we start with the proofs of the Tur\'an-type,
Szemer\'edi-type and Ramsey-type applications formulated in the introduction. In
the next section we turn to counting lemmas.

\subsection{Quick applications}\label{sec:quick}

In this subsection we collect some proofs of applications that are straightforward. We start with the proof of the sparse random $H$-density theorem.

\begin{proof}[Proof of Theorem~\ref{thm:Hdensity}]
We first explain why $g_H(\rho)$ is uniformly continuous. Indeed, for any $\eps>0$, take a graph which witnesses the value of $g_H(\rho,N)$, and arbitrarily add edges such that the result has $\lfloor(\rho+\eps)\binom{N}{k}\rfloor$ edges. Since each added edge lies in at most $N^{v(H)-k}$ copies of $H$, we have $g_H(\rho+\eps,N)\le g_H(\rho,N)+\eps N^{v(H)}$, which immediately implies uniform continuity of $g_H(\rho)$. A similar argument (removing edges) applies to $g'_H(\rho)$.

Given a $k$-uniform hypergraph $H$ with at least $2$ edges and $m_k(H)>\tfrac{1}{k}$, and given $\eps>0$, let $C$ be returned by Theorem~\ref{thm:colHGtrans} for input $e(H)$, $H$, $r=2$ and $\tfrac1{3k!}\eps$.
 Given $N$, suppose $p\ge CN^{-1/m_k(H)}$. Let $\eta_{N,p}$ be the probability that $G^{(k)}(N,p)$ contains more than $(1+\tfrac1{9k!}\eps)p^{e(H)}N^{v(H)}$ copies of $H$. By Theorem~\ref{thm:colHGtrans}, with probability at least $1-\eta_{N,p}-\exp(-pN^k/C)$, the random hypergraph $\Gamma=G^{(k)}(N,p)$ has the dense model property. Suppose this likely event occurs.

 Given $\rho>0$, let $G\subset\Gamma$ contain $\lfloor\rho p\binom{N}{k}\rfloor$ edges. Let $G'$ be a corresponding dense model, obtained by applying the likely event of Theorem~\ref{thm:colHGtrans} with $G_1=G$ and $G_2=\Gamma-G$. In particular, we have $e(G')=\big(\rho\pm\tfrac13\eps\big)\binom{N}{k}$. By definition, the number of copies of $H$ in $G'$ is between 
 \[g_H\Big(\rho-\frac13\eps,N\Big)\ge g_H(\rho,N)-\frac13\eps N^{v(H)}\ge \Big(g_H(\rho)-\frac23\eps\Big)N^{v(H)}\]
 and
 \[g'_H\Big(\rho+\frac13\eps,N\Big)\le g'_H(\rho,N)+\frac13\eps N^{v(H)}\le \Big(g'_H(\rho)+\frac23\eps\Big)N^{v(H)}\,,\]
 where we presume $N$ is sufficiently large for the limit statement in the definitions of~$g_H(\cdot)$ and~$g'_H(\cdot)$.
 Since $G'$ is a dense model of $G$, the number of copies of $H$ in $G$ is therefore between $\big(g_H(\rho)-\eps\big)p^{e(H)}N^{v(H)}$ and $\big(g'_H(\rho)+\eps\big)p^{e(H)} N^{v(H)}$, as required.
\end{proof}

Next, we prove the sparse Ramsey multiplicity theorem. We use the following lower dense model version of Theorem~\ref{thm:colHGtrans} with an improved failure probability.

\begin{theorem}[lower dense model version of simplified transference for hypergraphs with colours]\label{thm:lowercolHGtrans}
  \mbox{} \\
  Let $k\ge 2$ and let \( H \) be a $k$-uniform hypergraph with
  $m_k(H)>1/k$. Let $r\ge2$, and \( \eps > 0 \).  Let \( \eta_{N,p} \) be the
  probability that the number of $H$-copies in \(G^{(k)}(N,p) \) exceeds \(
  (1+\tfrac\eps3) p^{e(H)} N^{v(H)} \).  There exists a constant \( C > 0 \)
  such that for \( p\ge C N^{-1/m_k(H)} \) the random hypergraph $\Gamma =
  G^{(k)}(N,p)$ satisfies the following with probability at least
  $1-\exp(-pN^k/C)$.  For every $r$-colouring $E(\Gamma)=G_1\dcup\dots \dcup
  G_r$ of $\Gamma$, there exists an $r$-colouring $E(K^{(k)}_N)=G'_1\dcup\dots\dcup
  G'_r$ of $K^{(k)}_N$ such that for each $i\in[r]$
  \[
    e(G_i) p^{-1} = e(G'_i) \pm \eps N^k
    \qquad \text{and} \qquad
    c(H,G_i)p^{-e(H)}\ge c(H,G'_i)-\eps N^{v(H)}\,.
    \]
\end{theorem}

To see that this theorem holds, one simply follows the proof of
Theorem~\ref{thm:colHGtrans}, replacing Theorem~\ref{thm:GenSimp} with
Theorem~\ref{thm:lowerGenSimp}. We now show how this theorem implies Theorem~\ref{thm:multiplicity}.

\begin{proof}[Proof of Theorem~\ref{thm:multiplicity}]
 Given a $k$-uniform hypergraph $H$ with at least $2$ edges and $m_k(H)>\tfrac{1}{k}$, and $\eps>0$, let $C$ be returned by Theorem~\ref{thm:lowercolHGtrans} for input $e(H)$, $H$, $r$ and $\tfrac1{2r}\eps$.

 Given $N$, suppose $p\ge CN^{-1/m_k(H)}$. By Theorem~\ref{thm:lowercolHGtrans}, with probability at least $1-\exp(-pN^k/C)$, the random hypergraph $\Gamma=G^{(k)}(N,p)$ has the lower dense model property. Suppose this likely event occurs.

 Given an $r$-colouring of $E(\Gamma)$ with colour classes $G_1,\dots,G_r$, let $G'_1,\dots,G'_r$ be obtained from the likely event of Theorem~\ref{thm:lowercolHGtrans}. By definition, $G'_1,\dots,G'_r$ form an $r$-colouring of $E(K_N)$, so for $N$ sufficiently large for the limit statement there are at least $\big(c_{H,r}-\tfrac12\eps\big)$ monochromatic copies of $H$ in this $r$-colouring of $E(K_N)$. Applying the lower dense model property with each of $G_1,\dots,G_r$ separately, there are at least
 \[\big(c_{H,r}-\tfrac12\eps-r\cdot\tfrac{1}{2r}\eps\big)p^{e(H)}N^{v(H)}=\big(c_{H,r}-\eps\big)p^{e(H)}N^{v(H)}\]
 monochromatic copies of $H$ in the $r$-colouring of $E(\Gamma)$, as required.
\end{proof}

Finally, we provide the proof of the sparse Szemer\'edi multiplicity theorem.

\begin{proof}[Proof of Theorem~\ref{thm:sparse-szem}]
  We begin by observing a Lipschitz continuity property of $f(k,\delta)$. Recall that by definition of $f(k,\delta)$ for any $\tilde\eps>0$ and~$n$ sufficiently large any subset of $[n]$ of size at least $\delta n$ contains at least $\big(f(k,\delta)-\tilde\eps\big)n^2$ arithmetic progressions of length~$k$.
  Any given element of $[n]$ is contained in at most $kn$ arithmetic progressions of length $k$ in $[n]$ ($k$ choices for the term in the progression and $n$ for the common difference). It follows that $f(k,\delta+\eps)\le f(k,\delta)+k\eps$, since adding $\eps n$ elements to an optimal set with $\delta n$ elements can add at most $k\eps n^2$ progressions. Equivalently,
\[f(k,\delta-\eps)\ge f(k,\delta)-k\eps\,.\]

Given $k\ge3$ and $\eps>0$, let $C$ be returned by Theorem~\ref{thm:lowerGenSimp} for input $k$, $r=2$, $c=3$ and $\tfrac1{2k}\eps$.
Let $S$ be the ordered $k$-uniform hypergraph on $[n]$ whose edges are the $k$-term arithmetic progressions with positive common difference, in increasing order, and let $p\ge Cn^{-1/(k-1)}\ge (\log^C n)n^{-1}$. Let $\Omega=\{\one\}$ and $\Sigma=\{\one\}$. 
Observe that $\Delta_1(S)\le 3e(S) n^{-1}$, and for each $2\le\kappa\le k$ we have $\Delta_\kappa(S)=1\le cC^{1-\kappa}p^{\kappa-1}e(S)n^{-1}$ by choice of $p$.

We apply Theorem~\ref{thm:lowerGenSimp}, whose conditions we just verified, which states that with probability $1-\exp(-pn/C)$, the random binomial set $X=[n]_p$ has the $\tfrac1{2k}\eps$-good lower dense model property.
Given $Y\subset X$ of size at least $\delta p n$, let $Z$ be the corresponding $\tfrac1{2k}\eps$-good lower dense model obtained by using $Y_1=Y$ and $Y_2=X\setminus Y$ in the likely event of Theorem~\ref{thm:lowerGenSimp}. By~\ref{main:itm:sim} with $\sigma=\one$, we have $|Z|\ge(\delta-\tfrac1{2k}\eps)n$.
Assuming $n$ to be sufficiently large, $Z$ therefore contains at least $f\big(k,\delta-\tfrac{1}{2k}\eps\big)n^2-\tfrac{1}{4}\eps n^2$ arithmetic progressions of length $k$. As observed above, we have $f(k,\delta-\tfrac1{2k}\eps)\ge f(k,\delta)-\tfrac12\eps$, so $Z$ contains at least $f(k,\delta)n^2-\tfrac{3}{4}\eps n^2$ arithmetic progressions of length $k$.
Applying~\ref{main:itm:sub'} with $\omega=\one$, the set $Y$ contains at least
\[f(k,\delta)p^kn^2-\frac{3}{4}\eps p^kn^2-\frac{1}{2k}\eps p^kn^2\ge\big(f(k,\delta)-\eps\big)p^kn^2\]
$k$-term arithmetic progressions, as required.
\end{proof}

\subsection{Proof of the sparse pentagon theorem}
\label{sec:pentagon}

As written, Theorem~\ref{thm:dense-pentagon} provides an upper bound on the number of $C_5$ copies in a triangle-free graph, and not one with very few triangles. The proof method of~\cite{Grzesik,HHKNR} uses flag algebras and actually does provide such a bound, but for completeness we use the triangle removal lemma, originally proved by Ruzsa and Szemer\'edi~\cite{RuzSze}, to deduce it from Theorem~\ref{thm:dense-pentagon}.

\begin{lemma}[triangle removal lemma]
  There exists $\eps'>0$ such that for every~$N$ any $N$-vertex graph with at
  most $\eps'N^3$ copies of $K_3$ can be made $K_3$-free by removing at most
  $\delta N^2$ edges.
\end{lemma}

\begin{proof}[Proof of Theorem~\ref{thm:pentagon}]
Given $\eps>0$, let $\delta=\tfrac1{100}\eps$. The triangle removal lemma guarantees that there exists $\eps'>0$ such that any $N$-vertex graph with at most $\eps'N^3$ copies of $K_3$ can be made $K_3$-free by removing at most $\delta N^2$ edges. Without loss of generality, we assume $\eps'\le\tfrac1{100}\eps$.
We apply Theorem~\ref{thm:mainmulti} with input $k=5$, $r=t=2$, $c=2\cdot 5!$, $\eps'$, and $\rho=\tfrac12$, to obtain a constant $C^\ast>0$. We let $C=10\cdot 2^55!C^\ast$.

Given $N$ sufficiently large, let $S_1$ be the $3$-uniform ordered hypergraph on $[n]=\binom{[N]}{2}$ whose edges are all $3$-tuples of edges of~$K_N$ that form triangles, and let $S_2$ be the $5$-uniform ordered hypergraph on $\binom{[N]}{2}$ whose edges are all $5$-tuples of edges of~$K_N$ that in a cyclic order form copies of $C_5$ in $K_N$. We let $\Sigma=\emptyset$, $\Omega_1=\{\one\}$, and $\Omega_2=\{\one\}$.
Let $q\ge C N^{-1/2}$. Since $n=O(N^2)$ and $e(S_i)=\Omega(N^3)$ for each $i$, condition~\ref{itm:mainmulti:e} of Theorem~\ref{thm:mainmulti} is satisfied.
To see that also condition~\ref{itm:mainmulti:Delta} of this theorem is satisfied, we use the same argument as when we showed in the proof of Theorem~\ref{thm:colHGtrans} that condition~\ref{itm:GS:iv} of Theorem~\ref{thm:GenSimp} holds. We conclude that with high probability, the random set $X=[n]_q$, which corresponds to the random graph $\Gamma=G(N,q)$, satisfies the $\eps'$-good dense model property with respect to both $S_1$ and $S_2$. Suppose this likely event occurs.

Let now $G\subset\Gamma$ be a triangle-free subgraph of $\Gamma$. By the dense model property (applied with $Y_1=G$ and $Y_2=X\setminus G$) we obtain an $\eps'$-good dense model $G'$ of $G$. In particular, $G'$ contains at most $\eps'N^3$ copies of $K_3$. By the triangle removal lemma, there is a $K_3$-free subgraph $G''$ of $G'$ obtained by removing at most $\delta N^2$ edges of $G'$.
By Theorem~\ref{thm:dense-pentagon}, the dense graph $G''$ contains at most $\tfrac{2}{625}N^5$ copies of $C_5$. Adding back the $\delta N^2$ edges, $G'$ contains at most $\big(\tfrac{2}{625}+10\delta\big)N^5$ copies of $C_5$, and by the dense model property it follows that $G$ contains at most
\[\Big(\frac{2}{625}+10\delta+\eps'\Big)q^5N^5\le\Big(\frac{2}{625}+\eps\Big)q^5N^5\]
copies of $C_5$, as required.
\end{proof}

In this proof, we were not specific about the failure probability. We obviously obtain at least as good a bound as given in Theorem~\ref{thm:mainmulti}. However, we only need an upper bound for copies of $C_5$ (and this bound only has to be good enough for the final inequality in the above proof, not with an error term on the scale of $\eps'$), so that following this proof we get a failure probability bound $\eta_{N,q}+\exp\big(-qN^2/C\big)$, where $\eta_{N,q}$ is the probability that $G(N,q)$ contains more than $\big(1+\tfrac13\eps\big)$ times the expected number of copies of $C_5$.

\subsection{Proof of the sparse canonical van der Waerden theorem}\label{sec:vdw}

In this subsection we prove Theorem~\ref{thm:VdW}. Our proof (much as that of
Alvarado et al.~\cite{AKMMO}) uses the following idea. If there is a single
colour which appears $\Omega(pn)$ times in the colouring, then the dense model
of this colour will contain many monochromatic $k$-term progressions and we are
done. If not, we will find a rainbow $k$-term progression. In the latter case,
we can safely group colours together and assume that the total number of colours
is a large constant.

We use the following supersaturation version of Szemer\'edi's theorem, which can
be obtained by a simple averaging argument as first observed by
Varnavides~\cite{Varnavides} (see for example \cite[Lemma~4.2]{Balogh2015} for a
statement in the form below).

\begin{lemma}[supersaturation Szemer\'edi's theorem]
\label{lem:super}
 For every $\rho>0$ and integer~$k$, there exists $\delta>0$ such that for~$n$
 sufficiently large the following holds. Every subset of $[n]$ of size at least
 $\rho n$ contains at least $\delta n^2$ arithmetic progressions with~$k$
 terms.
\end{lemma}

\begin{proof}[Proof of Theorem~\ref{thm:VdW}]
  Given $k$, let $\rho=\tfrac{1}{16}k^{-4}$. Let $r=\lceil 2\rho^{-1}+1\rceil$,
  and let $\delta>0$ be returned by Lemma~\ref{lem:super} for input
  $\tfrac12\rho$ and~$k$. Let
  $\eps=\min(\tfrac12\delta,\tfrac{1}{8}k^{-2}r^{-k})$ and $c=k^2$. Let $C$ be
  returned by Theorem~\ref{thm:GenSimp}, and let $C^\ast=C$ be the constant promised by our theorem.

  Given $n$, let $S$ be the ordered $k$-uniform hypergraph on $[n]$ whose
  elements are the $k$-term progressions (in order) on $[n]$; it is easy to see
  we have $e(S)\ge n^2k^{-2}$. We take $\Sigma=\{\one\}$ and
  $\Omega=\{\one\}$. Finally, given $p\ge C^\ast n^{-1/(k-1)}$, we now argue
  that the conditions of Theorem~\ref{thm:GenSimp} are satisfied. The bounds on
  $|\Sigma|$ and $|\Omega|$ of~\ref{itm:GS:ii} hold trivially.
  For~\ref{itm:GS:iv}, consider $1\le\kappa\le k$ and a $k$-tuple
  $\mathbf{x}$ with exactly $\kappa$ entries not equal to $\ast$. If
  $\kappa\ge2$, there is at most one $k$-term progression in $[n]$ passing
  through the given integers in the given positions, so $\Delta_\kappa(S)\le
  1\le cC^{1-\kappa}p^{\kappa-1}\tfrac{e(S)}{n}$ follows from $cC^{1-k}p^{k-1}e(S)\ge n$. For
  $\kappa=1$, there are at most $n$ choices for a given $\ast$ entry of
  $\mathbf{x}$ and this then determines the element of $S$ uniquely, so
  $\Delta_1(S)\le n\le ce(S)/n$ by choice of $c$.
Suppose that the likely event of Theorem~\ref{thm:GenSimp} holds for $X=[n]_p$. Since $\one\in\Sigma$ and $\eps<1$, this implies that $|X|\le 2pn$. Let $\chi':X\to\mathbb{N}$ be any colouring of $X$.

If there is a colour class $Y$ of $\chi'$ of size at least $\rho pn$, then consider the dense model $Z\subset[n]$ of $Y$ as guaranteed by Theorem~\ref{thm:GenSimp}. Since $\one\in\Sigma$, we have $|Z|\ge\tfrac12\rho n$. By Lemma~\ref{lem:super} there are at least $\delta n^2$ arithmetic progressions with~$k$ terms in $Z$, and by choice of $\eps$ there are therefore at least $\tfrac12\delta p^kn^2>0$ arithmetic progressions with $k$ terms in $Y$; in particular there is a monochromatic $k$-term progression in $\chi'$.

Suppose now no colour class of $\chi'$ has size at least $\rho pn$. We iteratively identify pairs of colour classes each of size less than $\rho p n$ until no such pairs remain. In the resulting colouring, all but at most one colour class has size at least $\rho p n$ and so the number of colour classes is at most $\tfrac{2pn}{\rho pn}+1= 2\rho^{-1}+1\le r$. Relabelling the colours, this gives a colouring $\chi:X\to[r]$. By construction, the colour classes $Y_1,\dots,Y_r$ of $\chi$ have size at most $2\rho p n$. Importantly, observe that if a given $k$-term progression is rainbow according to $\chi$, it is also rainbow according to $\chi'$.
Let $Z_1,\dots,Z_r$ be the dense model of $Y_1,\dots,Y_r$ as guaranteed by Theorem~\ref{thm:GenSimp} and
let $\chi^*:[n]\to[r]$ be the corresponding colouring of $[n]$. Since $\one\in\Sigma$, we have $|Z_i|\le 4\rho n$ for each $i\in[r]$.

Consider the following random experiment: we choose $a$ and $d$ independently and uniformly at random in $[\lfloor n/k\rfloor]$. We obtain the (random) $k$-term progression $(a,a+d,\dots,a+(k-1)d)$ contained in $[n]$. Now, fix $0\le i<j\le k-1$ and consider the two elements $a+id,a+jd$ of this progression. These two elements form a pair in $\binom{[n]}{2}$, chosen uniformly at random from a set of possible pairs of size $\lfloor n/k\rfloor^2\ge n^2/(2k^2)$.
The total number of $\chi^*$-monochromatic pairs in $\binom{[n]}{2}$ is at most $n\cdot 4\rho n$: having chosen a first element of the pair in $[n]$, this element will be in some $Z_q$ and there are at most $|Z_q|$ choices of a second element which creates a monochromatic pair. Thus the probability that $(a+id,a+jd)$ is monochromatic is at most $\frac{4\rho n^2}{n^2/(2k^2)}\le 8k^2\rho$.

Taking a union bound over choices of $i$ and $j$, the probability that the random progression contains a monochromatic pair is at most $8k^4\rho=\tfrac12$. Thus $\chi^*$ contains at least $\tfrac12\lfloor\tfrac{n}{k}\rfloor^2\ge \tfrac14n^2k^{-2}$ rainbow $k$-term progressions. Fixing a colour pattern from the less than $r^k$ total colour patterns which occurs most often as a rainbow $k$-term progression, there is a specific rainbow colour pattern such that the number of $k$-term progressions with the given colour pattern occurs at least $\tfrac14n^2k^{-2}r^{-k}$ times in $\chi^*$. Because $\chi^*$ is a dense model of $\chi$, by choice of $\eps$ the same colour pattern occurs as a $k$-term progression at least $\tfrac{1}{8}p^kn^2k^{-2}r^{-k}>0$ times in $\chi$. In particular, there is a rainbow $k$-term progression in $\chi$ and so in $\chi'$, as required.
\end{proof}

Our proof above gives $o(1)$ bounds on the failure of the canonical van der Waerden property. However, it is trivial to enhance this to the optimal $e^{-O(pn)}$ bounds by replacing $X$ with the $\eps$-deletion $\tilde{X}$, matching~\cite{AKMMO}: obviously if every colouring of the latter set contains either a monochromatic or rainbow $k$-term progression, the same is true for the former set. We also note that our argument above showing $\chi^*$ contains many rainbow $k$-term progressions is essentially the same as~\cite[Lemma~3.5]{AKMMO}.

\section{A sparse counting lemma for hypergraphs}\label{sec:hypercount}

In this section we prove a sparse counting lemma for hypergraphs, Theorem~\ref{thm:hypercount}, which immediately implies Theorem~\ref{thm:graphcount}. This turns out to be a fairly easy application of Theorem~\ref{thm:GenSimp}. Most of what follows is simply dealing with the somewhat complicated hypergraph regularity setup and handling constant choices.

Let \(k\) be a positive integer. A \emph{\(k\)-complex} is a down-closed hypergraph in which all edges have size at most \(k\). Given a \(k\)-complex \(H\) with at least \(k+1\) vertices, we define its \emph{\(k\)-density} as \(d_k(H) := \frac{e_k(H) - 1}{v(H) - k}\), where \(e_k(H)\) is the number of edges of size \(k\) in \(H\), and \(v(H)\) denotes the number of vertices of \(H\). We also define \(m_k(H) := \max_{H' \subseteq H} d_k(H')\), where the maximum is taken over all sub-\(k\)-complexes \(H'\) of \(H\) with at least \(k+1\) vertices.

Given a vertex set \([N]\), a \emph{\(k\)-partition} \(\mathcal{V}\) with \(\ell\) \emph{clusters}  consists of a family of disjoint subsets \(V_{\{1\}}, \dots{}, V_{\{\ell\}} \subseteq [N]\) called \emph{clusters}, together with, for each integer \(2 \le i \le k\) and each subset \(E \subset [\ell]\) of size \(i\), a collection \(V_E\subset\binom{[N]}{i}\), called \emph{\(i\)-edges}. These \(V_E\) must satisfy the following compatibility condition: for every \(e \in V_E\) and every \(j \in E\), the set \(e\) intersects the cluster \(V_{\{j\}}\) in exactly one element, and the remaining \(i-1\) elements of \(e\) form an \((i-1)\)-edge in \(V_{E \setminus \{j\}}\). The \emph{supporting \((i-1)\)-graph} of \(V_E\) is the \((i-1)\)-uniform hypergraph consisting of all \((i-1)\)-sets that arise in this way from some edge of \(V_E\).

Let \(E \subset [\ell]\) be given, with \(|E| = i \ge 2\), and suppose \(V_E\) is given along with its supporting \((i-1)\)-graphs \(W_1, \dots{}, W_i\). 
For any subsets \(Q_1 \subset W_1, \dots{}, Q_i \subset W_i\), define \(R(Q_1, \dots{}, Q_i)\) to be the collection of \(i\)-element subsets of \([N]\) that contain one element from each \(Q_j\). 
In particular, \(R(W_1, \dots{}, W_i)\) contains \(V_E\).
If \(R(W_1, \dots{}, W_i)\) is nonempty, and given \(p \in (0,1]\), we define respectively the \emph{relative density} of \(V_E\) and the \emph{relative \(p\)-density} of \(V_E\) as
\[d^*(V_E) := \frac{|V_E|}{|R(W_1, \dots{}, W_i)|}\hspace{1cm}\text{and}\hspace{1cm}d^*_p(V_E) := \frac{|V_E|}{p \cdot |R(W_1, \dots{}, W_i)|}\,.\]
Finally, for singleton sets, we define  \(d^*(V_{\{i\}}) := |V_{\{i\}}| N^{-1}.\)

Let \(E \subset [\ell]\) be a set of size \(i\), with \(2 \le i \le k\), and let \(p \in (0,1]\). Consider \(V_E\) and let \(W_1, \dots{}, W_i\) denote the supporting \((i{-}1)\)-graphs of \(V_E\). 
We say that \(V_E\) is \emph{\((\varepsilon, r, p)\)-regular with respect to its supporting \((i{-}1)\)-graphs} if the following holds. For any set \(R^*\) of the form
\(
R^* = \bigcup_{j=1}^r R(Q_1^{(j)}, \dots{}, Q_i^{(j)})
\)
where \(Q_i^{(j)}\subseteq W_i\) with \(|R^*| \ge \varepsilon |R(W_1, \dots{}, W_i)|\) we have
\[
\frac{|V_E \cap R^*|}{p |R^*|} = d^*_p(V_E) \pm \varepsilon.
\]
If either of the parameters \(r\), \(p\), or both are omitted, they are understood to be equal to \(1\).

So far we gave the standard setup for defining the so-called strong regularity. At this point we generalise slightly and consider what we shall call $s$-regularity, for some $1\le s\le k-1$. Here $s=k-1$ will be the usual strong regularity, while $s=1$ is normally called weak regularity (and in this context it is normal to take $r=1$). Given a $k$-partition $\mathcal{V}$ with $\ell$ clusters, and a $k$-set $E\subset[\ell]$, we say the \emph{supporting $s$-graphs} of $V_E$ are the collection of $s$-graphs $V_{E'}$ for $E'\subset E$ with $|E'|=s$. For each $E'\subset E$ with $|E'|=s$, choose $Q_{E'}\subset V_{E'}$. Then we define $R\big(\{Q_{E'}\colon E'\subset E,|E'|=s\}\big)$ to be the collection of $k$-element subsets $e$ of $[N]$ with one member in each cluster of $V_E$ such that the intersection of $e$ with the clusters of $V_{E'}$ is in $Q_{E'}$ for each $E'\subset E$ with $|E'|=s$. We define the \emph{relative $(p;s)$-density} of $V_E$ to be
\[d^*_{p;s}(V_E):=\frac{|V_E|}{p\cdot\big|R\big(\{V_{E'}\colon E'\subset E, |E'|=s\}\big)\big|}\,,\]
and we say $V_E$ is \emph{$(\eps,r,p;s)$-regular with respect to $\{V_{E'}\colon E'\subset E, |E'|=s\}$} if the following holds. For any set $R^*$ of the form $R^*=\bigcup_{j=1}^r R\big(\{Q^{j}_{E'}\colon E'\subset E, |E'|=s\}\big)$ with $|R^*|\ge\eps \big|R\big(\{V_{E'}\colon E'\subset E,|E'|=s\}\big)\big|$ we have
\[\frac{|V_E\cap R^*|}{p|R^*|}=d^*_{p;s}(V_E)\pm\eps\,.\]
Observe that for $s=k-1$, this is simply repeating the definition of $(\eps,r,p)$-regularity, while for $s=r=1$ we obtain weak regularity: we are counting edges supported by given vertex subsets of the clusters. 
A \(k\)-partition $\mathcal{V}$ with $\ell$ clusters is said to be \emph{\((\varepsilon_{k}, \varepsilon, d_1, \dots{},d_s, d_{k}, r, p;s)\)-regular} if the following hold.
\begin{enumerate}[label=\rom]
    \item For each \(i \in [\ell]\), we have \(|V_{\{i\}}|=d_1 N\).
    \item For every \(E \subset [\ell]\) with \(2 \le |E| \le s\), the set \(V_E\) is \(\varepsilon\)-regular with respect to its supporting \((|E|-1)\)-graphs, and its relative density satisfies \(d^*(V_E)= d_{|E|}\pm\eps\).
    \item For every $E\subset[\ell]$ with $s+1\le|E|\le k-1$, we have $V_E=R(W_1,\dots,W_{|E|})$, where $W_1,\dots,W_{|E|}$ are the supporting ($|E|-1$)-graphs of $V_E$.
    \item For every \(E \subset [\ell]\) with \(|E| = k\), the set \(V_E\) is \((\varepsilon_{k}, r, p;s)\)-regular with respect to its supporting \(s\)-graphs, and its relative \((p;s)\)-density satisfies \(d^*_{p;s}(V_E) \ge d_{k}\).
\end{enumerate}
Note that the third condition enforces $d^*(V_E)=1$ for any $E$ with $s+1\le|E|\le k-1$.

Let \(H\) be a \(k\)-complex. An injective map \(\phi\colon V(H) \to [\ell]\) is called a \emph{\(k\)-complex homomorphism} if for every edge \(e \in E(H)\), the image \(\phi(e)\) has size \(|e|\). That is, \(\phi\) maps the vertices of each edge to distinct cluster indices.
Given a \(k\)-partition \(\mathcal{V}\) with \(\ell\) clusters over the vertex set \([N]\), a map \(\psi\colon V(H) \to [N]\) is said to be a \emph{\(\phi\)-partite copy of \(H\) in \(\mathcal{V}\)} if \(\psi\) is injective and for every edge \(e \in E(H)\), the image \(\psi(e)\) is an element of \(V_{\phi(e)}\). 

We now set up all the notation we need to formulate the hypergraph counting statements we are interested in. We begin with a dense counting assumption, that $H$-counting in the dense case follows from $s$-regularity, which is captured in the following definition. We shall then transfer this dense statement to the sparse setting.

\begin{definition}[$s$-regularity counting]\label{def:count}
  Given \(k\ge 2\), $1\le s\le k-1$, and a fixed \(k\)-complex \(H\), we say
  that $H$ has \emph{$s$-regularity counting} if the following holds.  For every
  \(\delta,d_k>0\), there exists \(\eps_k>0\) such that for every
  \(d_2,\dots{},d_s>0\) with $d_2^{-1},\dots,d^{-1}_s\in\mathbb{N}$ there exist
  \(\eps>0\) and \(r\in\mathbb{N}\) such that for any \(d_1>0\) the following
  holds for all sufficiently large $N$.  Given any \(k\le \ell\le v(H)\) and
  \((\eps_k,\eps,d_1,\dots{},d_s,d_k,r,1;s)\)-regular \(k\)-partition
  \(\mathcal{V}\) with \(\ell\) clusters on \([N]\), and given any \(k\)-complex
  homomorphism \(\phi\colon V(H)\to [\ell]\), the number of \(\phi\)-partite copies of
  \(H\) in \(\mathcal{V}\) is
 \[(1\pm\delta)N^{v(H)}\prod_{e\in E(H)}d^*\big(V_{\phi(e)}\big)\,.\]
\end{definition}

Observe that if for some $E\subset[\ell]$ with $|E|=k$ the set~$E$ is not in the image of $\phi$, then the edges of $V_E$ do not play a r\^ole in the above assumption, nor is the density of $V_E$ relevant. It follows that we do not need the assumption of regularity or $d^*(V_E)\ge d_k$ for such $E$. To see this formally, consider modifying the $k$-partition by replacing $V_E$ with all the $k$-edges supported by the supporting $(k-1)$-graphs of $V_E$. The resulting $k$-partition satisfies the required density and regularity conditions, without changing the partite $H$-count or the formula.

A result of Kohayakawa, Nagle, R\"odl and Schacht~\cite[Lemma~10]{KNRS} tells us that $H$ has $1$-regularity counting for any \emph{linear} $H$. That is, for any two maximal edges of $H$, the intersection of the two edges contains at most one vertex, and this implies that $H$ has $1$-regularity counting whenever $H$ is a linear $k$-complex. In fact, they show that in this setting we can even take $r=1$.

At the other end of the spectrum, it was proved that $H$ has $(k-1)$-regularity counting by Cooley, Fountoulakis, K\"uhn and Osthus~\cite[Lemma~4]{Cooley2009}, building on the seminal work of R\"odl and Schacht~\cite{RodSchCounting}, which is in turn related to earlier results of Gowers~\cite{GowCounting} and Nagle, R\"odl and Schacht~\cite{NagRodSch}. Formally,  Cooley, K\"uhn, Fountoulakis and Osthus only proved a special case, but a standard reduction gives the more general statement; we sketch the reduction in Appendix~\ref{app:count}.

Other values of~$s$ have to our knowledge not been considered.
In general, it seems reasonable to believe that $H$ has $s$-regularity counting if and only if any two $k$-edges of $H$ intersect in at most $s$ vertices. We give the standard example showing necessity in Appendix~\ref{app:count}, and also argue there that for sufficiency we only need to prove the special case that $H$ is the down-closure of a $k$-uniform hypergraph. For counting statements this setting of the down-closure of $k$-uniform hypergraphs is what the literature so far seems to have concentrated on.

\medskip

We now turn to the sparse setting.
In the following theorem, when $H$ is a $k$-complex, by `the number of copies of $H$ in $G^{(k)}(N,p)$' we mean the number of copies in the $k$-complex obtained by adding to $G^{(k)}(N,p)$ all possible edges of uniformity strictly smaller than $k$.

\begin{theorem}[Counting lemma for sparse hypergraphs]\label{thm:hypercount}
  Let \(k\ge 2\) and $1\le s\le k-1$, and let \(H\) be a \(k\)-complex that has
  $s$-regularity counting.  For any \(\delta,d_k>0\), there exists
  \(\eps_k>0\) such that for any \(d_2,\dots{},d_s>0\) with
  $d_2^{-1},\dots,d_s^{-1}\in\mathbb{N}$, there exist \(\eps>0\) and
  \(r\in\mathbb{N}\) such that for any \(d_1>0\) there exist $\eps^*>0$ and
  $C^\ast$ such that the following holds.  Suppose that \(N\) is sufficiently large,
  and \(p\ge\max\big(C^\ast N^{s-k},C^\ast N^{-1/m_k(H)}\big)\). Let $\eta_{N,p}$ be the
  probability that the number of copies of $H$ in $G^{(k)}(N,p)$ exceeds
  $\big(1+\tfrac13\eps^*\big)p^{e(H)}N^{v(H)}$.  With probability at least
  \[1-\eta_{N,p}-\exp\Big(-\frac{pN^k}{C^\ast k!}\Big)\,,\]
  the random \(k\)-uniform hypergraph \(\Gamma=G^{(k)}(N,p)\) has the following property. 

  Given any \(k\le \ell\le v(H)\), any
  \((\eps_k,\eps,d_1,\dots{},d_s,d_k,r,p;s)\)-regular \(k\)-partition
  \(\mathcal{V}\) with \(\ell\) clusters on \([N]\) such that for each
  \(E\subset [\ell]\) with \(|E|=k\) we have \(V_E\subset\Gamma\), and given any
  \(k\)-complex homomorphism \(\phi\colon V(H)\to [\ell]\), the number of
  \(\phi\)-partite copies of \(H\) in \(\mathcal{V}\) is
 \[(1\pm\delta)N^{v(H)}\prod_{e\in E(H)}d^*\big(V_{\phi(e)}\big)\,.\]
 \end{theorem}

We remark that our counting formula in this theorem does not contain~$p$ because the density $d^*\big(V_{\phi(e)}\big)$ for $|e|=k$ is the actual density of the set $V_{\phi(e)}$ of $k$-edges and $V_{\phi(e)}\subset\Gamma$.
Further, in this theorem, we do allow for the possibility that some edges of \(H\) of uniformity smaller than \(k\) are not contained in any \(k\)-edges of \(H\); that is, \(H\) need not be just the down-closure of a \(k\)-uniform hypergraph. This turns out to be useful in some applications for \(k\ge3\).

\begin{proof}[Proof of  Theorem~\ref{thm:hypercount}]
  The case \(e_k(H)=0\) of  Theorem~\ref{thm:hypercount} is implied by the fact that $H$ has $s$-regularity counting, viewing $H$ as a $(k-1)$-complex.

 The case \(e_k(H)=1\) is standard and does not require  Theorem~\ref{thm:GenSimp}. We give only a sketch. Let \(H'\) be the \((k-1)\)-complex \(H\) with the one \(k\)-edge removed. An application of the extension lemma \cite[Lemma 5]{Cooley2009} shows that all but a tiny fraction of \(k\)-sets supported by any given $(\eps,\eps,d_{1},\dots,d_s,1,\dots,1,r;k-1)$-regular \((k-1)\)-partition \(\mathcal{V}'\) (where the number of densities equal to $1$ is $k-1-s$) are in roughly the same number of \(\phi\)-partite copies of \(H'\), and that the exceptional \(k\)-sets account for only a tiny fraction of all \(\phi\)-partite copies of \(H'\). A standard application of the Chernoff bound from Theorem~\ref{thm:chernoff} shows that with very high probability, when \(G^{(k)}(N,p)\) is revealed, there are very few edges on these exceptional \(k\)-sets and the number of \(\phi\)-partite \(H\)-copies they generate is tiny compared to those on typical \(k\)-sets. Critically, this `very high probability' is sufficient for a union bound over choices of \(\mathcal{V}'\) and \(\phi\). Supposing now this likely event occurs, given any regular \(\mathcal{V}\), letting \(\mathcal{V}'\) the \((k-1)\)-partition obtained by removing the \(k\) layer, we see that (using the fact that \(\eps_k\) is much smaller than \(d_k\)) most of the \(k\)-edges of \(\mathcal{V}\) are on typical \(k\)-sets and a short calculation gives the desired count of \(\phi\)-partite \(H\)-copies.

 We now assume $e_k(H)\ge 2$, which is the interesting case. Let \(2\eps_k>0\) be small enough for the $s$-regularity counting of~$H$ with input \(\tfrac12\delta\). Given \(d_2,\dots{},d_s,d_k>0\), let \(\eps\) and \(r\) be given according to the $s$-regularity counting of~$H$ for input \(d_2,\dots{},d_{s},\tfrac12d_k\). Let finally \(d_1>0\) be given.
 We set \(c=2v(H)!\), and apply Theorem~\ref{thm:GenSimp} with input uniformity \(e_k(H)\), \(c\) and 
 \[\eps^*=\frac{\delta d_k\eps_k^2}{10v(H)!}\prod_{e\in E(H)}d_{|e|}\,.\] 
 Let \(C\) be the constant returned by  Theorem~\ref{thm:GenSimp} and let \(C^\ast =10r2^kCk!\).

 Order the \(k\)-edges of \(H\). Let \(n=\binom{N}{k}\) enumerate the edges of \(K^{(k)}_N\), and let \(S\) consist of the ordered subsets of \([n]\) corresponding to \(e_k(H)\)-sets in \([N]\) which form isomorphic copies of the \(k\)-uniform edges of \(H\), in the chosen order.
  We now verify the codegree conditions of Theorem~\ref{thm:GenSimp} hold for \(S\) and \(p\ge C^\ast N^{-1/m_k(H)}\). To begin with, we estimate \(e(S)\). Let \(q(H)\) be the number of vertices of \(H\) which are not in any \(k\)-uniform edge of \(H\). There are \((1+o(1))N^{v(H)-q(H)}\) injective maps from the vertices of \(H\) which are in \(k\)-edges to \([N]\), each of which gives one element of \(S\), so \(e(S)=(1+o(1))N^{v(H)-q(H)}\).

 Given \(1\le \kappa\le e_k(H)\), let \(\mathbf{x}\) be a sequence of length \(e_k(H)\) from \([n]\cup\{\ast\}\) with exactly \(\kappa\) entries not equal to \(\ast\). For \(\kappa=1\), by symmetry we have \(\deg_S(\mathbf{x})=\tfrac{e(S)}{n}\), which is as required. We now assume \(\kappa\ge2\). 
 Let \(W\subset[N]\) be the vertices of \(K^{(k)}_N\) which are contained in some edge in \(\mathbf{x}\). By definition, if \(\mathbf{x}\) has two identical non-\(\ast\) entries, then \(\deg_S(\mathbf{x})=0\), so we can assume that \(\mathbf{x}\) has at least two distinct non-\(\ast\) entries, and hence \(|W|\ge k+1\). By definition of \(m_k(H)\), we have
 \[\frac{\kappa-1}{|W|-k}\le m_k(H)\,,\qquad\text{so}\qquad |W|\ge\frac{\kappa-1}{m_k(H)}+k\,.\]
 To obtain a member of \(S\) which agrees with \(\mathbf{x}\) at the non-\(\ast\) coordinates, we can at most pick a further \(v(H)-|W|-q(H)\) vertices in \(k\)-edges of \(H\) and one of the at most \(v(H)!\) maps from the vertices of \(H\) in \(k\)-edges to the picked vertices together with \(W\). Thus, we have
 \begin{align*}
  \deg_S(\mathbf{x})&\le  N^{v(H)-|W|-q(H)}v(H)!
  \le 2v(H)!e(S)N^{-|W|}
  \le 2v(H)!e(S)N^{-\frac{\kappa-1}{m_k(H)}-k}\\
  &= 2v(H)!\frac{e(S)}{N^k}\big(N^{-1/m_k(H)}\big)^{\kappa-1}
  \le 2v(H)!\frac{e(S)}{n}(p/C^\ast )^{\kappa-1}\,,
 \end{align*}
which is the required bound.

We next give the set $\Sigma$ of structure functions. These are the `obvious' functions which allow us to transfer $s$-regularity: that is, for any collection of $\binom{k}{s}$ hypergraphs $Q_1,\dots,Q_{\binom{k}{s}}$ on $[N]$ of uniformity $s$, we consider the set $R(Q_1,\dots,Q_{\binom{k}{s}})$ of $k$-sets in $[N]$ which contain exactly one edge of each $Q_i$. We let \(\Sigma\) consist of the indicator functions of the unions of any up to \(r\) sets of the form \(R(Q_1,\dots{},Q_{\binom{k}{s}})\). Then we have
 \[|\Sigma|\le 2\cdot 2^{r\binom{k}{s}N^{s}}\le \exp\Big(\frac{pn}{C}\Big)\,,\]
 where we use the simple bound that the number of $s$-uniform hypergraphs on $[N]$ is at most $2^{N^s}$, and the inequality uses \(p\ge C^\ast N^{s-k}\) and the choice of \(C^\ast \).

 Finally, we give the set $\Omega$ of subcounts. For each \(k\le\ell\le v(H)\), consider each choice of a \(k\)-partition \(\mathcal{V}\) with \(\ell\) clusters whose \(s+1,\dots,k\) levels are complete (i.e.\ each \(V_E\) with \(s+1\le |E|\le k\) is equal to \(R(W_1,\dots{},W_{|E|})\) where \(W_1,\dots{},W_{|E|}\) are the supporting \((|E|-1)\)-graphs), and each \(\phi\colon V(H)\to[\ell]\). For each such \((\ell,\mathcal{V},\phi)\) we construct a subcount \(\omega\) as follows. For each member \(\mathbf{s}\) of \(S\) (recall $\mathbf{s}$ is an ordered tuple of $e_k(H)\ge 2$ members of $[n]=\binom{[N]}{k}$), we count the number \(\tilde w(\mathbf{s})\) of \(\phi\)-partite copies \(\psi\) of \(H\) in \(\mathcal{V}\) such that the \(i\)-th edge of \(H\) is mapped to the \(i\)-th member of \(\mathbf{s}\), for each \(1\le i\le e_k(H)\). Observe that necessarily \(0\le \tilde w(\mathbf{s})\le (v(H))!N^{q(H)}\), where \(q(H)\) is the number of vertices of \(H\) not in any \(k\)-edge of \(H\). We define \(\omega(\mathbf{s})=\tfrac{1}{(v(H))!N^{q(H)}}\tilde w(\mathbf{s})\), which is therefore in \([0,1]\). We say this is the subcount corresponding to \((\ell,\mathcal{V}'',\phi)\) for any choice \(\mathcal{V}''\) of a \(k\)-partition which is identical to \(\mathcal{V}\) on any level except perhaps the \(k\)-th.
 We now upper bound the size \(|\Omega|\) of the set of all such subcounts. There are \(v(H)\) choices of \(\ell\), and at most \(v(H)^\ell\le v(H)^{v(H)}\) choices of \(\phi\). What remains is to estimate the number of choices of \(\mathcal{V}\). Observe that \(\mathcal{V}\) is defined by the choices of \(V_E\) for \(1\le |E|\le s\). There are at most \(N^{v(H)+1}\) ways to choose the clusters, since the clusters are disjoint. Again since the clusters are disjoint, to define \(V_E\) for each \(2\le |E|\le s\) it suffices to choose a subset of each of \(\binom{[N]}{2}\) through \(\binom{[N]}{s}\), which can be done in at most \(2^{N^2}\cdots 2^{N^{s}}\) ways. We conclude
 \[|\Omega|\le v(H)^{v(H)+1}N^{v(H)+1}2^{sN^{s}}\le \exp\Big(\frac{pn}{C}\Big)\,,\]
 where as before the inequality uses \(p\ge C^\ast N^{s-k}\) and the choice of \(C^\ast \), and this time also that \(N\) is sufficiently large. This completes the check that $(S,\Sigma,\Omega)$ satisfies the $(c,C,p,n)$-boundedness conditions.

 Suppose now that \(X=[n]_p\) satisfies the likely event of  Theorem~\ref{thm:GenSimp} for this \(\eps^*\), \(S\), \(\Sigma\) and \(\Omega\). This occurs with probability at least
 \[1-\Prob\Big(\big|S\cap [n]_p^{e_k(H)}\big|\ge\Big(1+\frac13\eps^*\Big)p^{e_k(H)}e(S)\Big)-\exp\Big(-\frac{pn}{C^\ast }\Big)\]
 which, using the choice of our constants and $n\le N^k/k!$, is as required by our sparse counting lemma. Let \(\Gamma\) be the corresponding instance of \(G^{(k)}(N,p)\).

 Given \(k\le\ell\le v(H)\) and an \((\eps_k,\eps,d_1,\dots{},d_s,d_k,r,p;s)\)-regular \(k\)-partition \(\mathcal{V}\) with \(\ell\) clusters on \([N]\), such that for each \(V_E\) with \(|E|=k\) we have \(V_E\subset\Gamma\), let \(Y\) be the subset of \(X\) consisting of elements in any \(V_E\) with \(|E|=k\). Let \(Z\) be the dense model of~$Y$ guaranteed by the likely event of  Theorem~\ref{thm:GenSimp}, and let \(\mathcal{V}'\) be the \(k\)-partition with \(\ell\) clusters on \([N]\) obtained by replacing each \(V_E\) where \(|E|=k\) with \(V'_E\) corresponding to the elements of \(Z\) that are supported on the \((k-1)\)-graphs supporting \(V_E\).

 We claim that \(\mathcal{V}'\) is \((2\eps_k,\eps,d_1,\dots{},d_s,\tfrac12d_k,r,1;s)\)-regular and that the relative densities of the top level are close to the relative \(p\)-densities of \(\mathcal{V}\). The regularity of the levels from \(1\) to \(k-1\) follows from the regularity of \(\mathcal{V}\), and what needs to be proved is that each \(V'_E\) with \(|E|=k\) is \((2\eps_k,r,1;s)\)-regular with density \(d^*(V'_E)=\big(1\pm\tfrac{\delta}{10e_k(H)}\big)d_p^*(V_E)\ge\tfrac12d_k\). 
 To see this holds, fix \(E\) and let \(W_1,\dots{},W_{\binom{k}{s}}\) be the supporting \(s\)-graphs of \(V_E\) (so also of \(V'_E\)). Let \(R^*\) be a union of at most \(r\) sets of the form \(R(Q_1,\dots{},Q_{\binom{k}{s}})\) (as defined where we described the set \(\Sigma\) of structure functions), with the extra condition \(Q_i\subseteq W_i\) for each \(1\le i\le \binom{k}{s}\). Abusing notation slightly, we think of \(R^*\) as both a subset of \(\binom{[N]}{k}\) and of \([n]\). Because \(Z\) is a dense model of \(Y\), using the structure function \(\sigma\in\Sigma\) which takes the value \(1\) precisely on \(R^*\), we have
 \[p^{-1}|R^*\cap Y|=|R^*\cap Z|\pm\eps^*n\,.\]
 Taking the particular case that \(R^*\) is all \(k\)-sets supported by \(W_1,\dots{},W_{\binom{k}{s}}\), this immediately says that
 \[d^*(V'_E)=d^*_p(V_E)\pm\eps^*nN^{-k}\prod_{i=1}^{k-1}d_i^{-\binom{k}{i}}=\Big(1\pm \frac{\delta}{10e_k(H)}\Big)d^*_p(V_E)\ge\frac{d_k}{2}\,,\]
 where the final equality is by choice of \(\eps^*\).
 Suppose now \(R^*\) contains at least an \(\eps_k\)-fraction of all \(k\)-edges supported by \(W_1,\dots,W_{\binom{k}{s}}\). Because \(\mathcal{V}\) is regular, we have
 \[|R^*\cap Y|=|R^*\cap V_E|=(1\pm\eps_k)d^*(V_E)|R^*|=(1\pm\eps_k)pd^*_p(V_E)|R^*|\,.\]
 Putting these bits together, we have
 \begin{align*}
   |R^*\cap Z|&=(1\pm\eps_k)d_p^*(V_E)|R^*|\pm\eps^*n
   =(1\pm\eps_k)\Big(1\pm\frac{\delta}{10e_k(H)}\Big)d^*(V'_E)|R^*|\pm\eps^*n \\
   &=(1\pm2\eps_k)d^*(V'_E)|R^*|  
 \end{align*}
 which verifies \((2\eps_k,r,1;s)\)-regularity of \(V'_E\). Here again the final inequality is by choice of \(\eps^*\).

 Using that $H$ has $s$-regularity counting in \(\mathcal{V}'\), we see that the number of \(\phi\)-partite copies of \(H\) in \(\mathcal{V}'\) is
 \[\Big(1\pm\frac12\delta\Big)N^{v(H)}\prod_{e\in E(H)}d^*\big(V'_{\phi(e)}\big)\,.\]
 Letting \(\omega\) be the subcount corresponding to \((\ell,\mathcal{V},\phi)\), we have by definition of \(\omega\) that
 \begin{align*}
  \sum_{\mathbf{s}\in S}\mathbbm{1}(\mathbf{s}\subset Z)\omega(\mathbf{s})(v(H))!N^{q(H)}&=\Big(1\pm\frac12\delta\Big)N^{v(H)}\prod_{e\in E(H)}d^*\big(V'_{\phi(e)}\big)\\
  &=\Big(1\pm\frac34\delta\Big)N^{v(H)}p^{-e_k(H)}\prod_{e\in E(H)}d^*\big(V_{\phi(e)}\big)\,.
 \end{align*}
 Here the final equality uses that \(d^*_p(V_E)=p^{-1}d^*(V_E)=\Big(1\pm\frac{\delta}{8e_k(H)}\Big)d^*(V'_E)\) whenever \(|E|=k\).

 Since \(Z\) is a dense model of \(Y\), we have
 \[p^{-e_k(H)}\sum_{\mathbf{s}\in S}\mathbbm{1}(\mathbf{s}\subset Y)\omega(\mathbf{s})=\sum_{\mathbf{s}\in S}\mathbbm{1}(\mathbf{s}\subset Z)\omega(\mathbf{s})\pm\eps^*e(S)\,.\]
 We therefore get
 \begin{align*}
 \Big(1\pm\frac34\delta\Big)&N^{v(H)}p^{-e_k(H)}\prod_{e\in E(H)}d^*\big(V_{\phi(e)}\big)\\
 &=p^{-e_k(H)}\sum_{\mathbf{s}\in S}\mathbbm{1}(\mathbf{s}\subset Y)\omega(\mathbf{s})(v(H))!N^{q(H)}\pm\eps^*e(S)(v(H))!N^{q(H)}\\
 &=p^{-e_k(H)}\sum_{\mathbf{s}\in S}\mathbbm{1}(\mathbf{s}\subset Y)\omega(\mathbf{s})(v(H))!N^{q(H)}\pm\eps^*N^{v(H)}(v(H))!\,,
 \end{align*}
 and so
 \begin{align*}
   \sum_{\mathbf{s}\in S}\mathbbm{1}(\mathbf{s}\subset Y)\omega(\mathbf{s})(v(H))!N^{q(H)}&=\Big(1\pm\frac34\delta\Big)N^{v(H)}\begin{aligned}[t]
   &\prod_{e\in E(H)}d^*\big(V_{\phi(e)}\big)\\
   &\pm\eps^*(v(H))!N^{v(H)}p^{e_k(H)}\\   
   \end{aligned}\\
   &=\big(1\pm\delta\big)N^{v(H)}\prod_{e\in E(H)}d^*\big(V_{\phi(e)}\big)
 \end{align*}
 by choice of \(\eps^*\). Since the left-hand side of this is, by definition of \(\omega\), the number of \(\phi\)-partite copies of \(H\) in \(\mathcal{V}\), this completes the proof.
\end{proof}

\section{Concluding remarks}\label{sec:concl}

As discussed, our main results provide a full transference principle when certain codegree bounds are satisfied, with optimal failure probability.
We have already compared our results in detail to those of Conlon and Gowers~\cite{ConGow}. In these concluding remarks we compare or relate them also to various other relevant bodies of work.
Before that, though, we briefly explain where exactly in our proof we make use of our codegree bounds on $S$ and the lower bound on~$p$ in our results.

\subsection{Use of codegree bounds and the lower bound on~\texorpdfstring{$p$}{p}}

While there are several places where we use some codegree bound, there is actually only one place where we really need the precise bound; in all other places, a bound worse by a polynomial in $pn$ would still suffice for our proofs. The place we require the precise bound is in the proof of Lemma~\ref{lem:momentboundlemma}, specifically to prove~\eqref{eq:momentbound:expY} which gives an upper bound on the expected number of books in the inner product. Here the precise bound is used to say that the main contribution to this expectation consists of books whose pages intersect only in the spine, which is needed for our moment bound approach to work. We use the codegree bounds again in this lemma to prove that the number of books in the inner product is likely to concentrate close to its expectation, but this statement would continue to hold for $p$ being (slightly) asymptotically smaller; what changes is that the expectation would no longer be the `right' value, with other kinds of books dominating.

In terms of our proof strategy, what happens is that when $p$ is above the codegree threshold, if we look at the sum of the absolute values of a typical $\phi\in\Phi(\tmu,\one)$ (i.e.\ at $\langle\one,|\phi|\rangle$), what we see is that most of this sum comes from entries which are not too big---we prove smaller than $2c$---and we use moment bounds to argue that the contribution of the exceptionally large values to our inner products is tiny.

When $p$ goes slightly below the codegree threshold, what changes is that most values of $\phi$ will be equal to zero; for any given $x\in[n]$, finding even one edge of $S$ which contains $x$ and whose members are in the random set is already atypical and corresponds to an exceptionally large value in $\phi$ (if $S$ is `balanced', for example if it comes from counting balanced graphs, then this is true as written; if it is not balanced, then one edge does not actually correspond to an exceptionally large value, but typically if there is one edge, then there will be far more than the expectation many edges and these together do give an exceptionally large value in $\phi$). Our argument using Bernstein's inequality that $\langle\mu-\one,\phi^\smll\rangle$ is likely to be tiny continues to work perfectly well even though $p$ is below the codegree threshold, but it becomes useless: $\phi^\smll$ will consist mainly of zero entries and most of the sum $\langle\mu-\one,\phi\rangle$ comes from $\phi^\bg$, which we can no longer control and which will in fact typically not be small.

\medskip

We now comment briefly on the $\log^C n$ term in our lower bound on $p$ in~\ref{itm:GS:i} of Theorem~\ref{thm:GenSimp}. We need the high power $C$ in order to apply the Kim--Vu inequality. We expect that this could be avoided by doing more work to find a sharper concentration inequality. However, we need at least a single $\log n$ term to accommodate our union bounds. We made no particular effort to optimise this term because we do not know of any interesting application where it matters.

\subsection{Comparison to hypergraph container results}

As explained, our transference principle works under the same codegree conditions as the hypergraph containers theorems of Balogh, Morris and Samotij~\cite{Balogh2015} and Saxton and Thomason~\cite{Saxton2015}. However, while the hypergraph container results establish a positive density of~$H$ copies in the setting of our hypergraph counting lemma, Theorem~\ref{thm:hypercount}, they do not give precise counting bounds.

Roughly, hypergraph container results give a structural statement about sets in (in our notation) $S$ which are close to independent, in that they contain only a tiny fraction of $e(S)$ edges. These structural statements are useful in many contexts beyond the transference setting we consider (see~\cite{Balogh2015,Saxton2015} for examples), but in the setting of transference, what this means is that the best one can hope for via hypergraph containers is something like the lower dense model Theorem~\ref{thm:lowerGenSimp}, but instead of guaranteeing that the scaled count of edges of $S$ in the sparse colouring is within $\eps e(S)$ of the count in the dense colouring, one can only expect a guarantee that the scaled count is at least $\eps e(S)$ when the count in the dense colouring exceeds some (much larger than $\eps$) threshold $\zeta e(S)$.

\subsection{Comparison to Schacht\texorpdfstring{~\cite{schacht2016extremal}}{ }}

As mentioned in the introduction, Schacht~\cite{schacht2016extremal} was the first to prove many of the straightforward extremal results that follow from this paper (and from the hypergraph containers papers) without a limitation to strictly balanced graphs and hypergraphs. However, much as with containers, his results do not allow for sharp counting statements.

As we explained, Schacht's work uses a hypergraph framework that is rather similar to the one we use here. However, he does not require codegree bounds on his hypergraph, but instead a rather different condition (see~\cite[Definition~3.2]{schacht2016extremal}). As Balogh, Morris and Samotij observe (see~\cite[Remark~5.3]{Balogh2015}), the codegree bounds we (and they) assume on $S$ imply Schacht's condition, while if Schacht's condition is true, then by deleting a tiny fraction of edges from $S$ we reach an ordered hypergraph $S'$ whose codegrees are bounded as we require.

It follows that our main results could be slightly generalised to require only Schacht's condition on $S$, at the cost of an extra term in the failure probability dealing with the probability that an exceptionally large number of edges of $S\setminus S'$ appear in the binomial random subset. Doing this formally would require only changing the proof of Theorem~\ref{thm:GenSimp} from Theorem~\ref{thm:main}. We do not formally state these slightly more general results because we do not know of an application which is not already covered by our results.

\subsection{Comparison to Conlon, Gowers, Samotij, and Schacht\texorpdfstring{~\cite{CGSS}}{ }}

We already explained that Conlon, Gowers, Samotij, and Schacht~\cite{CGSS} established the special case of our sparse graph counting lemma, Theorem~\ref{thm:graphcount}, that~$H$ is strictly $2$-balanced. Our Theorem~\ref{thm:hypercount} generalises Theorem~\ref{thm:graphcount} to hypergraphs.

In our proof of Theorem~\ref{thm:hypercount}, we used subcounts. It seems likely that with quite a bit of extra technical work one could avoid this, but subcounts certainly make this proof easier. The reader might reasonably ask how it is that Conlon, Gowers, Samotij, and Schacht~\cite{CGSS} manage to prove their sparse graph counting lemma without having subcounts available. Their approach is that instead of applying transference to an entire subgraph $G$ of $G(N,p)$, they apply it only to the regular pairs where they want to count: thus, to prove a counting lemma for triangles between vertex sets $V_1,V_2,V_3$, they will first remove any edges which do not go between these three sets. For triangles, this works perfectly well. However, to count copies of $C_4$ going round four vertex sets $V_1,\dots,V_4$, a bit more care is needed: there will be `extra' copies of $C_4$ which for instance go only between $V_1$ and $V_2$, and the Conlon--Gowers transference statement will count these copies as well as the intended ones. It is not hard to deal with this (one should inductively assume we can count copies of $C_4$ in all smaller systems of regular pairs, and this works for any $H$ replacing $C_4$).

\medskip

We remark that~\cite[Theorem~1.6]{CGSS} also provides a counting lemma for general graphs~$H$ with only a lower bound
on the number of $H$-copies that guarantees a $\xi$-fraction of the expected count for a small $\xi>0$ and
has the optimal failure probability $\exp(-pn/C)$. With the tools developed in this paper, we can also strengthen this statement, achieving the same optimal failure probability, but with $\xi$ arbitrarily close (depending on $\eps$) to $1$. For this, we follow the proof of Theorem~\ref{thm:hypercount} but use Theorem~\ref{thm:lowerGenSimp}, the $\eps$-good lower dense model version, instead of Theorem~\ref{thm:GenSimp}.

In addition, we can (replacing Theorem~\ref{thm:GenSimp} with Theorem~\ref{thm:main}) prove a deletion version of Theorem~\ref{thm:hypercount} with $\exp(-pn/C)$ failure probability. This gives the existence of an $\eps$-deletion of $G^{(k)}(N,p)$ within which we obtain the two-sided counting statement of Theorem~\ref{thm:hypercount}.

\subsection{Relation to Freschi, Hancock and Treglown\texorpdfstring{~\cite{FHT}}{ }}

Recently Freschi, Hancock and Treglown~\cite{FHT} used the containers method to prove a general statement about Ramsey properties of discrete structures. Their main technical theorem is a statement about Ramsey properties of ordered hypergraphs (their definition of ordered hypergraph is the same as ours), which is~\cite[Theorem~7.7]{FHT}. The $1$-statement in their result almost follows from Theorem~\ref{thm:lowerGenSimp}, with the difference being that Theorem~\ref{thm:lowerGenSimp} requires $p\ge(\log^C n)n^{-1}$ while~\cite[Theorem~7.7]{FHT} requires only $p=\omega(n^{-1})$. As discussed, however, in typical applications the requirement on $p$ of Theorem~\ref{thm:lowerGenSimp}\ref{itm:lGS:iv} makes this difference immaterial.

We now explain how to deduce (with the stronger assumption on $p$)~\cite[Theorem~7.7]{FHT} from Theorem~\ref{thm:lowerGenSimp}. The former theorem is stated in terms of a sequence of $k$-uniform ordered hypergraphs $\mathcal{H}_n$ with $v(\mathcal{H}_n)$ tending to infinity with $n$, with the unspecified failure probability tending to zero. We have an explicit failure probability which tends to zero as the number of vertices tends to infinity. Their $\mathcal{H}_n$ is our $S$ (and their $v(\mathcal{H}_n)$ is our $n$, but we stick to $v(\mathcal{H}_n)$ in what follows to avoid confusion).

They define a quantity $\hat{p}_n>0$ by setting
\[\hat{p}_n=\max_{\substack{W\subset[k]\\|W|\ge2}}\Big(\frac{e(\mathcal{H}_{n,W})}{v(\mathcal{H}_n)}\Big)^{-1/(|W|-1)}\,,\]
where $\mathcal{H}_{n,W}$ is the $|W|$-uniform ordered hypergraph obtained by projecting $\mathcal{H}_n$ to the coordinates $W$. That is, letting $W=\{i_1,\dots,i_{|W|}\}$ in increasing order, $(x_{1},\dots,x_{{|W|}})$ is an edge of $\mathcal{H}_{n,W}$ if and only if there is an edge $(y_1,\dots,y_k)$ of $\mathcal{H}_n$ such that $x_{j}=y_{i_j}$ for each $1\le j\le |W|$. Given $W\subset[k]$, with again $W=\{i_1,\dots,i_{|W|}\}$ in increasing order, they define $\Delta_W(\mathcal{H}_{n})$ to be the maximum number of edges in $\mathcal{H}_{n}$ which project to any given edge of $\mathcal{H}_{n,W}$. That is, we take the maximum, over $|W|$-uniform ordered edges $(x_1,\dots,x_{|W|})$ in $V(\mathcal{H}_n)$, of the number of edges $(y_1,\dots,y_k)$ in $\mathcal{H}_{n}$ such that $x_j=y_{i_j}$ for each $j\in W$.

An ordered hypergraph $\mathcal{H}$ is said to be \emph{$r$-Ramsey} if every $r$-colouring of $V(\mathcal{H})$ has a monochromatic edge of $\mathcal{H}$. Strengthening this, a sequence $(\mathcal{H}_n)_{n\in\mathbb{N}}$ is $r$-supersaturated if there exists $\rho>0$ such that for all sufficiently large $n$, every $r$-colouring of $V(\mathcal{H}_n)$ contains at least $\rho e(\mathcal{H}_n)$ monochromatic edges. The supremum $\rho^*$ of values $\rho$ which work in the above definition corresponds to the $r$-Ramsey multiplicity constant for $(\mathcal{H}_n)_{n\in\mathbb{N}}$ as in our Theorem~\ref{thm:multiplicity}.

Finally, they define the random ordered hypergraph $\mathcal{H}^v_{n,q_n}$ to be obtained by sampling the random vertex set $X=V(\mathcal{H}_n)_{q_n}$ and letting $\mathcal{H}^v_{n,q_n}=\mathcal{H}_n[X]$ be the induced subhypergraph. The superscript $v$ indicates that vertices are randomly sampled. This corresponds to our $S[X]$.

Having made these definitions, their $1$-statement is implied by the following (we removed some of their conditions, which are not needed for our proof: part of~\ref{itm:FHT:1} and we set $Y=[k]$ in~\ref{itm:FHT:3} and~\ref{itm:FHT:4}).

\begin{theorem}[{\cite[Theorem~7.7]{FHT}}]\label{thm:FHT}
Let $r, k \ge 2$ be integers, and let $b > 0$. Let $(\mathcal{H}_n)_{n\in\mathbb{N}}$ be a sequence of $k$-uniform
ordered hypergraphs and let $\hat{p}_n$ be as defined above. Suppose that
\begin{enumerate}[label=\itmarab{P}]
 \item\label{itm:FHT:1} $\hat{p}_n v(\mathcal{H}_n)\to\infty$ as $n\to\infty$,
\item\label{itm:FHT:2} $(\mathcal{H}_n)_{n\in\mathbb{N}}$ is $r$-supersaturated,
\item\label{itm:FHT:3} for any $W\subset[k]$, and any sufficiently large $n\in\mathbb{N}$, we have $\Delta_W(\mathcal{H}_{n})\le b\frac{e(\mathcal{H}_{n})}{e(\mathcal{H}_{n,W})}$,
\item\label{itm:FHT:4} for any $W\subset [k]$ with $|W|=1$, and any sufficiently large $n\in\mathbb{N}$, we
have $\Delta_W(\mathcal{H}_{n})\le b\frac{e(\mathcal{H}_{n})}{v(\mathcal{H}_n)}$.
\end{enumerate}
Then there exists a constant $C > 0$ such that if $q_n \ge C \hat{p}_n$ then $\lim_{n\to\infty}\Prob[\mathcal{H}^v_{n,q_n}\text{ is $r$-Ramsey}\,]=1$.
\end{theorem}
We should stress that~\cite[Theorem~7.7]{FHT} also contains a $0$-statement, roughly that under an extra condition, if $q_n\le C^{-1}\hat{p}_n$ then $\mathcal{H}^v_{n,q_n}$ is likely to be \emph{not} $r$-Ramsey. Our work does not address this part of the statement.

\begin{proof}[Proof of Theorem~\ref{thm:FHT} from Theorem~\ref{thm:lowerGenSimp}]
Given a sequence of hypergraphs satisfying the conditions of Theorem~\ref{thm:FHT} for some $k,r\ge2$ and $b>0$, let $\rho>0$ be a constant returned by the $r$-su\-per\-sa\-tu\-ra\-tion~\ref{itm:FHT:2}. We apply Theorem~\ref{thm:lowerGenSimp} with the same $k$ and $r$, with $c=b$, and with $\eps:=\frac12\rho$. Let $C$ be returned by Theorem~\ref{thm:lowerGenSimp}, and Theorem~\ref{thm:FHT} returns this same $C$. Given $q_n\ge C\hat{p}_n$, we let $p=q_n$. We set $\Sigma=\emptyset$ and $\Omega=\{\one\}$. 

We only prove the large $\hat{p}$ case of Theorem~\ref{thm:FHT}, which means assuming Theorem~\ref{thm:lowerGenSimp}\ref{itm:lGS:i} holds. Similarly,~\ref{itm:lGS:ii} holds trivially for all sufficiently large $n$ since $pv(\mathcal{H}_n)\to\infty$. So what we need to do is verify~\ref{itm:lGS:iv}. Recall that our $\Delta_\kappa(S)$ is (see Definition~\ref{def:degrees}) the maximum of $\deg_S(\mathbf{x})$ over $\mathbf{x}$ having exactly $\kappa$ entries not equal to $\ast$. By definition, this is identical to the maximum over $W\subset[k]$ with $|W|=\kappa$ of $\Delta_W(\mathcal{H}_n)$: the set $W$ specifies the places at which $\mathbf{x}$ is not $\ast$. Given $\kappa\in[k]$, let $W$ be a subset of $[k]$ of size $\kappa$.

If $\kappa\ge 2$, then~\ref{itm:FHT:3} says $\Delta_W(\mathcal{H}_n)\le b\frac{e(\mathcal{H}_n)}{e(\mathcal{H}_{n,W})}$, and the definition of $\hat{p}$ guarantees $e(\mathcal{H}_{n,W})\le\hat{p}_n^{-(\kappa-1)}v(\mathcal{H}_n)$. Substituting this in, we have
\[\Delta_W(\mathcal{H}_n)\le be(\mathcal{H}_n)\hat{p}_n^{\kappa-1}v(\mathcal{H}_n)^{-1}\le cC^{1-\kappa}q_n^{\kappa-1}e(S)v(\mathcal{H}_n)^{-1}\]
which gives the required bound on $\Delta_\kappa(S)$.

For $\kappa=1$, letting $W$ range over subsets of $[k]$ of size $1$,~\ref{itm:FHT:4} is precisely the required bound on $\Delta_1(S)$. Thus the conditions of Theorem~\ref{thm:lowerGenSimp} are met.

Now Theorem~\ref{thm:lowerGenSimp} gives us that, with probability at least $1-\exp(-q_nv(\mathcal{H}_n)/C)$, the random set $X=[V(\mathcal{H}_n)]_p$, which induces the random hypergraph $\mathcal{H}^v_{n,q_n}=\mathcal{H}_n[X]$, has the following property. Given any $r$-colouring $X=Y_1\dcup\dots\dcup Y_r$, there is an $\eps$-good lower dense model $V(\mathcal{H}_n)=Y'_1\dcup\dots\dcup Y'_r$. Suppose $\mathcal{H}^v_{n,q_n}$ has this likely property.

Given any such $r$-colouring of $X$, fix a corresponding $\eps$-good lower dense model. By the $r$-su\-per\-sa\-tu\-ra\-tion property~\ref{itm:FHT:2}, there are at least $\rho e(\mathcal{H}_n)$ monochromatic edges of $\mathcal{H}_n$ in the lower dense model $r$-colouring, so the $\eps$-good lower dense model property says that there are at least $\frac12\rho q_n^k e(\mathcal{H}_n)$ monochromatic edges in the given $r$-colouring of $\mathcal{H}^v_{n,q_n}$. In particular, there is at least one monochromatic edge, verifying that $\mathcal{H}^v_{n,q_n}$ is $r$-Ramsey.
\end{proof}
Our proof gives the optimal exponential in $q_nv(\mathcal{H}_n)$ failure probability for Theorem~\ref{thm:FHT}. In~\cite{FHT}, an explicit failure probability is not given, but we believe they also obtain the same shape of failure probability.
By letting $\eps$ tend to zero as $n\to\infty$, we obtain not just the $r$-Ramsey property but also the $r$-Ramsey multiplicity statement that the number of monochromatic edges in any $r$-colouring of $\mathcal{H}^v_{n,q_n}$ is at least $(1-o(1))\rho^* q_n^k e(\mathcal{H}_n)$, where $\rho^*$ is the Ramsey multiplicity constant of $(\mathcal{H}_n)_{n\in\mathbb{N}}$. We believe this stronger result is not available via the containers method of~\cite{FHT}.

\subsection{Relation to Gerke, Marciniszyn, and Steger\texorpdfstring{~\cite{GerMarSte}}{ }}

A conjecture stronger than our graph counting lemma, Theorem~\ref{thm:graphcount}, was formulated by Gerke, Marciniszyn, and Steger~\cite{GerMarSte}.

\begin{conjecture}[Gerke, Marciniszyn, and Steger~\cite{GerMarSte}]
\label{conj:stronger}
  Let~$H$ be a fixed graph. For each $\beta,\delta>0$ there exist constants $\eps,C,N_0>0$ such that for all $m\ge C N^{2-1/m_2(H)}$ and $N\ge N_0$ the following holds. Let~$\cG$ be the class of graphs obtained from~$H$ by replacing each vertex~$x$ of~$H$ with a vertex set~$V_x$ of size~$N$, and each edge $xy$ of~$H$ with an $\eps$-regular pair $(V_x,V_y)$ with~$m$ edges. Let~$\cG'$ be the class of graphs in~$\cG$ that contain less than
  \[(1-\delta)N^{v(H)}\Big(\frac{m}{N^2}\Big)^{e(H)}\]
  copies of~$H$ that embed~$x$ in~$V_x$ for all $x\in V(H)$. Then
  \[|\cG'|\le\beta^m\binom{N^2}{m}^{e(H)}\,.\]
\end{conjecture}

Our transference principle does not allow for a counting statement that strong. 
Conjecture~\ref{conj:stronger} has been shown for $H=K_3$ in~\cite{GerMarSte}, and very recently for $H=K_4$ by Veeranonchai~\cite{Veeranonchai}. It remains wide open otherwise.

\subsection{Extension to quasirandom structures}

It is very interesting to ask under which conditions in the results obtained in this paper the random set $[n]_p$ can be replaced by sets with suitably defined quasirandomness properties. Since in general the theory of Tur\'an and Ramsey type results in quasirandom structures is much less well understood than in the random setting, not much is known in this direction. In graphs, most is known for `jumbled' graphs, in particular general counting results of Conlon, Fox and Zhao~\cite{CFZ1} and (in some cases) sharper results of Allen, B\"ottcher, Skokan and Stein~\cite{ABSS}. For hypergraphs, a general result of Conlon, Fox and Zhao~\cite{CFZ2} gives bounds for counting relative to sets which themselves satisfy `typical counting' conditions. However, these results are generally not believed to be quantitatively sharp, and it would be very interesting to improve the quantitative bounds. The recent survey of Conlon~\cite{ConSurv} gives details and several further questions.

In addition, some results have been obtained recently for very specific properties, such as Ramsey type properties of matchings~\cite{KeeMich}.

\printbibliography

\appendix

\section{Omitted technicalities for \texorpdfstring{$s$}{s}-regularity counting}\label{app:count}

In this appendix we prove various claims we made concerning $s$-regularity counting in Section~\ref{sec:hypercount}.
Specifically, we stated there that we believe that $H$ has $s$-regularity counting only if any two $k$-edges of $H$ intersect in at most $s$ vertices. We argue here that this is a necessary condition and that if it is sufficient for all $H$ which are the down-closure of a $k$-uniform hypergraph, then it is sufficient in general. Moreover, we stated in Section~\ref{sec:hypercount} that $(k-1)$-regularity counting follows from~\cite[Lemma~4]{Cooley2009}. We also provide this reduction in this appendix.

\subsection*{Necessity of intersecting in at most \texorpdfstring{$s$}{s} vertices for \texorpdfstring{$s$}{s}-regularity counting}

Take $t\ge s+1$ vertex sets $V_1,\dots,V_{t}$ of size $N$, and put $t$-edges between them independently with probability $\tfrac12$. Let the resulting $t$-graph be $F$. We now extend this to a $(2k-t)$-cluster $k$-partition by adding $V_{t+1},\dots,V_k,V'_{t+1},\dots,V'_k$ of size $N$, and adding all crossing edges of all uniformities between $2$ and $k-1$. Finally, we add all $k$-edges crossing $V_1,\dots,V_k$ which contain an edge of $F$, and all $k$-edges crossing $V_1,\dots,V_{t},V'_{t+1},\dots,V'_k$ which do not contain an edge of $F$. It is easy to check, using the Chernoff bound from Theorem~\ref{thm:chernoff}, that with high probability (as $N\to\infty$) this construction is $(o(1),o(1),d_1,\dots,d_s,\tfrac14,r,1;s)$-regular. But by definition it does not contain any two $k$-edges which intersect on $V_1,\dots,V_{t}$. Embedding this construction in a larger $k$-partition, for any given $H$ which contains two $k$-edges whose intersection has exactly $t$ vertices, gives a counterexample to $s$-regularity counting for such $H$.

\subsection*{Sufficiency of intersecting in at most \texorpdfstring{$s$}{s} vertices for \texorpdfstring{$s$}{s}-regularity counting: reduction to down-closures}

We now argue that, for any given $s$, if $H$ has $s$-regularity counting for all $H$ such that no two $k$-edges of $H$ intersect in more than $s$ vertices, and such that $H$ is the down-closure of a $k$-uniform hypergraph, then $H$ has $s$-regularity counting for all $H$ such that no two $k$-edges of $H$ intersect in more than $s$ vertices.

To see this, given $H$ which has no two $k$-edges intersecting in more than $s$ vertices and is not the down-closure of a $k$-uniform hypergraph, let $H'$ be obtained by, separately for each maximal edge of $H$ of size $s$ or smaller, adding new vertices to increase the edge to size $k$ and then taking the down-closure of the resulting hypergraph. Observe that $H'$ continues to satisfy the condition that no two $k$-edges intersect in more than $s$ vertices.

Given a $k$-partition of $[N]$ for which we are considering $s$-regularity counting for~$H$, we construct a
new $k$-partition of $[N']$ with $N'\le v(H')N$ on which we are considering $s$-regularity counting for~$H'$ as follows.
We add one new cluster of size $d_1N$, disjoint from all others, for each vertex in $V(H')\setminus V(H)$. For each $V_E$ with $E\in E(H')\setminus E(H)$
and $2\le|E|\le s$ (in order of increasing $|E|$), we let $V_E$ consist of a $d_{|E|}$-random subset of the $|E|$-sets supported by the collection $V_{E'}$ for $E'\subset E$ of size $|E|-1$,
and for each $V_E$ with $E\in E(H')\setminus E(H)$
and $s+1\le |E|\le k$ (in order of increasing $|E|$), we let
$V_E$ consist of all $|E|$-sets supported by the collection $V_{E'}$ for $E'\subset E$ of size $|E|-1$. It is straightforward to check that the constructed $k$-partition is $(\eps_k,\eps,d_1,\dots,d_s,d_k,r;s)$-regular if and only if the original $k$-partition is. Letting $\phi'$ be the obvious extension of $\phi$ mapping the new vertices to their clusters, counting $\phi'$-partite copies of $H'$ in the constructed instance gives a good estimate for the number of $\phi$-partite copies of $H$ in the original instance, so that if $H'$ has $s$-regularity counting then so does~$H$.

\subsection*{\texorpdfstring{$(k-1)$}{(k-1)}-regularity counting}

Finally, we argue that all~$H$ have $(k-1)$-regularity counting. As described above, it is enough to prove this for $H$ which are the down-closure of a $k$-uniform hypergraph, so suppose this is the case. Cooley, Fountoulakis, K\"uhn and Osthus~\cite[Lemma~4]{Cooley2009} proved this under the further restriction that $d_k^{-1}\in\mathbb{N}$ and $d^*(V_E)=d_k\pm\eps_k$ holds for all $|E|=k$.

We now explain why this proves $(k-1)$-regularity counting with no further restriction.
We proceed by induction on~$k$, where the case $k=2$ is simply the graph counting lemma.
For general~$k$, given any $\delta,d_k>0$, we set $\delta'=\tfrac14\delta$ and $d'_k=1/(2\lceil d_k^{-1}\rceil)$, so that $(d'_k)^{-1}\in\mathbb{N}$. Let $\eps'_k$ be returned by~\cite[Lemma~4]{Cooley2009} for input $k,\delta',d'_k$. Let $\eps_k\le\tfrac14\eps'_k$ be such that $d'_k(d_k-\eps_k)^{-1}<1$.
For given $d_2,\dots,d_{k-1}$ we let $\eps,r$ be returned by~\cite[Lemma~4]{Cooley2009} for input $k,\delta',d'_k,d_2,\dots,d_{k-1}$.

Given $\mathcal{V}$ as required for $(k-1)$-regularity counting, we construct $\mathcal{V}'$ as follows: for each $E$ of size $k$ and $V_E$ in $\mathcal{V}$, we obtain $V'_E$ by selecting edges of $V_E$ independently with probability $d'_k(d^*(V_E))^{-1}$; by our assumption on $\eps_k$, this number is indeed in $[0,1]$. For $E$ of size smaller than $k$, we set $V'_E=V_E$.

We claim that with probability $1-o(1)$, the $k$-partition $\mathcal{V}'$ is $(\eps'_k,\eps,d_1,\dots,d_{k-1},d'_k,r,1;k-1)$-regular with $d^*(V_E)=d'_k\pm\eps'_k$ for each $E$ of size $k$.

Assuming this claim is true, we now calculate the number $M$ of $\phi$-partite copies of $H$ in $\mathcal{V}$ as follows. On the one hand, by linearity of expectation, the expected number of $\phi$-partite copies of $H$ in $\mathcal{V}'$ is
\[M\cdot\prod_{e\in E(H), |e|=k}\frac{d'_k}{d^*(V_{\phi(e)})}\,.\]
On the other hand, whenever $\mathcal{V}'$ is $(\eps'_k,\eps,d_1,\dots,d_{k-1},d'_k,r;k-1)$-regular with $d^*(V'_E)=d'_k\pm o(1)$ for each $E$ of size $k$, by~\cite[Lemma~4]{Cooley2009} the number of $\phi$-partite copies of $H$ in $\mathcal{V}'$ is
\[(1\pm\delta')N^{v(H)}\prod_{e\in E(H)}d^*(V'_{\phi(e)})\,.\]
Since the number of $\phi$-partite copies of $H$ in $\mathcal{V}'$ is always between $0$ and $N^{v(H)}$, the expected number of $\phi$-partite copies of $H$ in $\mathcal{V}'$ is
\[(1-o(1))\cdot (1\pm\delta')N^{v(H)}\prod_{e\in E(H)}d^*(V'_{\phi(e)})\pm o(1)N^{v(H)}=(1\pm2\delta')N^{v(H)}\prod_{e\in E(H)}d^*(V'_{\phi(e)})\,.\]
Equating the two estimates of the expectation and rearranging, and by choice of $\delta'$, we obtain the required estimate for $M$.

To establish the claim, consider some $V_E$ with $|E|=k$. Without loss of generality, we can assume $E=[k]$. Let $W=R(V_{E\setminus\{1\}},\dots,V_{E\setminus\{k\}})$. By
induction we know that for $(k-1)$-complexes in $(k-1)$-partitions we have $(k-2)$-regularity counting, which implies that $|W|=\Omega(N^k)$. By definition $\tfrac{|V_E|}{|W|}=d^*(V_E)$, so by the Chernoff bound from Theorem~\ref{thm:chernoff}, $\tfrac{|V'_E|}{|W|}\neq d'_k\pm o(1)$ with probability $o(1)$.

Suppose now $W'$, which is a union of at most $r$ sets of the form $R(Q^{(j)}_1,\dots,Q^{(j)}_k)$ where each $Q^{(j)}_i$ is a subset of $V_{E\setminus\{i\}}$, is a witness of $(\eps'_k,r,1)$-regularity with density $d'_k\pm\eps'_k$ failing. Then we have $|W'|\ge\eps'_k|W|$, so $|V_E\cap W'|=(d^*(V_E)\pm\eps_k)|W'|$. Again by Chernoff's inequality, the probability that
\[|V'_E\cap W'|\neq(1\pm\eps_k)(d^*(V_E)\pm\eps_k)\tfrac{d'_k}{d^*(V_E)}|W'|=(d'_k\pm\eps'_k)|W'|\]
is at most $\exp(-\Omega(N^k))$. Taking a union bound over the at most $\binom{\ell}{k}r2^{rk\cdot N^{k-1}}$ events, we see that the probability of any of these events occurring is, as claimed, $o(1)$.
\end{document}